\documentclass[leqno,centertags,11pt]{amsart}

\usepackage{amssymb,amsmath,amsfonts,amssymb}
\usepackage{hyperref}
\usepackage{enumerate}

-\marginparwidth \oddsidemargin -9mm \evensidemargin \oddsidemargin
\usepackage{latexsym}
\advance\hoffset by 5mm

\def\@abssec#1{\vspace{.05in}\footnotesize \parindent .2in
{\bf #1. }\ignorespaces}

\newtheorem{theorem}{Theorem}[section]

\newtheorem{lemma}[theorem]{Lemma}
\newtheorem{proposition}[theorem]{Proposition}

\newtheorem{definition}[theorem]{Definition}
\newtheorem{remark}[theorem]{Remark}

\newcounter{hypo}

\DeclareMathOperator{\divg}{div}

\newcommand{\R}{\ensuremath{\mathbb{R}}}
\newcommand{\RR}{\ensuremath{\mathcal{R}}}

\newcommand{\N}{\ensuremath{\mathbb{N}}}
\newcommand{\LL}{\ensuremath{\mathcal{L}}}

\newcommand{\T}{\ensuremath{\mathbb{T}}}

\newcommand{\beq}[1]{\begin{equation}\label{#1}}
\newcommand{\eeq}{\end{equation}}
\newcommand{\beqs}{\begin{equation*}}
\newcommand{\eeqs}{\end{equation*}}
\newcommand{\set}[1]{\left\{#1\right\}}

\newcommand{\tend}{\bibliographystyle{plain}\bibliography{ccrituniq}\end{document}}
\newcommand{\dd}{\mathrm{d}}

\allowdisplaybreaks \numberwithin{equation}{section}

\begin{document}

\title[Regularity issues of the drift-diffusion equation with nonlocal diffusion]{On the regularity issues of a class of drift-diffusion equations with nonlocal diffusion}
\author{Changxing Miao}
\address{Institute of Applied Physics and Computational Mathematics, P.O. Box 8009, Beijing 100088, P.R. China.}
\email{miao\_{}changxing@iapcm.ac.cn}
\author{Liutang Xue}
\address{Laboratory of Mathematics and Complex Systems (MOE), School of Mathematical Sciences, Beijing Normal University, Beijing 100875, P.R. China}
\email{xuelt@bnu.edu.cn}
\date{}
\maketitle

\begin{abstract}
  In this paper we address the regularity issues of drift-diffusion equation with nonlocal diffusion,
where the diffusion operator is in the realm of stable-type L\'evy operator and the velocity field is defined from the considered quantity by a zero-order pseudo-differential operator.
Through using the method of nonlocal maximum principle in a unified way, we prove the eventual regularity result in the supercritical type cases and the global regularity at some logarithmically supercritical cases.
The feature of these results is that the time after which the solution is smoothly regular in the supercritical type cases can be evaluated explicitly.
\end{abstract}

\section{Introduction}
We consider the Cauchy problem of the following drift-diffusion equation with nonlocal diffusion
\begin{equation}\label{DDeq}
  \partial_t \theta + u\cdot\nabla \theta +\mathcal{L}\theta = 0,\qquad \theta|_{t=0}(x)=\theta_0(x),
\end{equation}
where $x\in \R^d$ (or $\T^d$), $d\in \N^+$, $t\in\R^+$, $\theta:\R^+\times\R^d\rightarrow \R$ is a scalar-valued quantity, and the velocity field $u=\mathcal{P}(\theta):\R^+\times\R^d\rightarrow \R^d$ is defined from $\theta$
by the zero-order pseudo-differential operator:
\begin{equation}\label{uexp}
  u(x)=\mathcal{P}(\theta)(x)= a \,\theta(x) + \; \mathrm{p.v.}\int_{\R^d} S(y)\, \theta(x+y)\,\mathrm{d}y,
\end{equation}
with $a=(a_1,\cdots,a_d)\in \R^d$, and $S(x)=\frac{\Psi( x/|x|)}{|x|^d}=\left(\frac{\Psi_1(x/|x|)}{|x|^d},\cdots,\frac{\Psi_d(x/|x|)}{|x|^d}\right)\in C\left( \R^d\setminus \{0\}; \R^d\right)$
composed of Calder\'on-Zygmund kernels (\cite{Stein}).
The nonlocal diffusion operator $\mathcal{L}$ is given by
\begin{equation}\label{Lexp}
  \mathcal{L}f(x) =  \, \mathrm{p.v.}\int_{\R^d} \big(f(x)-f(x+y)\big)\, K(y)\,\mathrm{d}y,
\end{equation}
where the radially symmetric kernel function $K(y)=K(|y|)$ defined on $\R^d\setminus\{0\}$ satisfies that there exist some $\alpha\in ]0,1]$, $\tilde\alpha>0$
and $c_0>0$ ($c_0$ may be dependent on $\alpha$ and $\sigma$), $c_1\geq 1$ such that
\begin{equation}\label{Kcond1}
  c_1^{-1}\frac{m(|y|^{-1})}{|y|^d}\leq K(y) \leq c_1 \frac{m(|y|^{-1})}{|y|^d},\quad \forall 0 <|y| \leq c_0,\quad \textrm{and}
\end{equation}
\begin{equation}\label{Kcond2}
  \;\;0\leq K(y) \leq \frac{ c_1 }{|y|^{d+\tilde\alpha}},\qquad \;\; \forall |y|\geq c_0,
\end{equation}
with $m(y)= m(|y|)$ a radially symmetric function satisfying the following assumptions
\begin{enumerate}[(\textbf{A}1)]
\item
  \textrm{$m(|y|)$ is smooth away from zero, non-decreasing, with $m(0)=0$, $\lim_{|y|\rightarrow \infty} m(|y|)=\infty$};
\item there exists $\sigma\in [0,\alpha[$ such that
\begin{equation}\label{mDec}
  (\alpha-\sigma)\frac{m(|y|)}{|y|} \leq m'(|y|) \leq \alpha \frac{m(|y|)}{|y|},\quad \forall |y|\geq c_0^{-1}.
\end{equation}
\end{enumerate}
In some cases concerned, the condition \eqref{Kcond2} can be replaced by a more general condition
\begin{equation}\label{Kcond5}
  - \frac{c_1}{|y|^{d+\tilde\alpha}}\leq K(y) \leq \frac{c_1}{|y|^{d+\tilde\alpha}},\quad \forall |y|\geq c_0.
\end{equation}
Besides, we also consider the nonlocal operator $\mathcal{L}$ defined by \eqref{Lexp}-\eqref{mDec} with ``$c_0= \infty$", i.e., the kernel $K(y)=K(|y|)$ is given by
\begin{equation}\label{Kcond3}
  c_1^{-1}\frac{m(|y|^{-1})}{|y|^d}\leq K(y) \leq c_1 \frac{m(|y|^{-1})}{|y|^d},\quad \forall\, |y|>0,
\end{equation}
with $c_1\geq 1$ and $m(y)=m(|y|)$ satisfying (\textbf{A}1) and
\begin{enumerate}[(\textbf{A}3$)$]
\item there exists a constant $\sigma\in [0,\alpha[$ such that
\begin{equation}\label{mDec2}
  (\alpha-\sigma)\frac{m(|y|)}{|y|} \leq m'(|y|) \leq \alpha \frac{m(|y|)}{|y|},\quad \forall |y|> 0.
\end{equation}
\end{enumerate}

The diffusion operator \eqref{Lexp} defined above is in the realm of L\'evy operator; indeed, according to \eqref{mDec} and Lemma \ref{lem-mfact} below,
we deduce that for $\alpha\in ]0,1]$ and $\sigma\in [0,\alpha[$,
\begin{equation}
  c_0^{\alpha-\sigma}m(c_0^{-1})\, \frac{1}{|y|^{\alpha-\sigma}} \leq m(|y|^{-1})\leq c_0^{\alpha}m(c_0^{-1})\, \frac{1}{|y|^{\alpha}},\quad \forall 0<|y|\leq c_0,
\end{equation}
which leads to
\begin{equation}\label{Kcond4}
  c_1^{-1}c_0^{\alpha-\sigma}m(c_0^{-1})\, \frac{1}{|y|^{d+\alpha-\sigma}}\leq K(y)\leq c_1 c_0^{\alpha}m(c_0^{-1})\, \frac{1}{|y|^{d+\alpha}},\quad \forall 0<|y|\leq c_0,
\end{equation}
and we know that the operator given by \eqref{Lexp} satisfying \eqref{Kcond4} and
$\int_{\R^d} \left( \min\{1,|y|^2\}\right) K(y)\dd y\leq C$
corresponds to the infinitesimal generator of the stable-type L\'evy process (\cite{CSZ}).
By taking the Fourier transform on $\mathcal{L}$, we get
\begin{equation}\label{Lsymb}
  \widehat{\LL \, f} (\zeta) = A(\zeta)  \widehat{f}(\zeta),\quad \forall \zeta\in\R^d,
\end{equation}
where the symbol $A(\zeta)$ is given by the following L\'evy-Khintchine formula (\cite[Eq. 3.217]{Jacob})
\begin{equation}\label{LKf}
  A(\zeta) := \int_{\R^d\setminus \{0\}}\left( 1- \cos(\zeta\cdot y)\right) K(y)\dd y.
\end{equation}

The diffusion operator $\LL$ defined by \eqref{Lexp} under \eqref{Kcond1}-\eqref{Kcond2} or \eqref{Kcond1}, \eqref{Kcond5}
contains a large class of multiplier operators $\LL= m(D)$ such as
$$\LL= |D|^\beta,\; (\beta\in [\alpha-\sigma,\alpha]),\quad\textrm{and}\quad \LL= \frac{|D|^\alpha}{\left(\log(\lambda + |D|)\right)^\mu},\;\left(\alpha\in ]0,1],\mu> 0,\lambda\geq 1\right),$$
which we shall explain in the subsection \ref{subsec-pre1} below.
Among them, an important case, which is also a particular case of $\LL$ under \eqref{Kcond3}-\eqref{mDec2},
is the fractional Laplacian operator $ |D|^\alpha:= (-\Delta)^{\frac{\alpha}{2}}$ with $\alpha\in ]0,1]$, which has the following representation formula
\begin{equation}\label{}
  |D|^\alpha f(x) = c_{d,\alpha}\,\mathrm{p.v.}\int_{\R^d}\frac{f(x) - f(x+y)}{|y|^{d+\alpha}}\dd y,
\end{equation}
with $c_{d,\alpha}>0$. The operator $\LL =|D|^\alpha$ corresponds to the infinitesimal generator of the symmetric stable L\'evy process,
and recently has been intensely studied in many theoretical problems.
For the drift-diffusion equation \eqref{DDeq}-\eqref{uexp} with $\LL=|D|^\alpha$,
we conventionally call the cases $\alpha<1$, $\alpha=1$ and $\alpha>1$
as supercritical, critical and subcritical cases, respectively.

The drift-diffusion equation \eqref{DDeq}-\eqref{uexp} has various physical background from the geophysics, fluid dynamics, dislocation theory and other fields.
The typical examples are the surface quasi-geostrophic equation, the Burgers equation, the C\'ordoba-C\'ordoba-Fontelos equation and the incompressible porous media equation,
and below we will specifically review some noticeable results related to these models. For other interesting models expressed as \eqref{DDeq}-\eqref{uexp},
one can also to \cite{BalC,DL,LMX} etc.

The surface quasi-geostrophic (SQG) equation writes as the equation \eqref{DDeq} with
\begin{equation}\label{vel-SQG}
  d=2 \quad\textrm{and}\quad u= \mathcal{R}^\perp \theta= (-\RR_2,\RR_1)\theta,
\end{equation}
where $\RR_i=\partial_i |D|^{-1}$ ($i=1,2$) is the usual Riesz transform (\cite{Stein}).
The inviscid model (i.e. $\LL=0$) arises from the geostrophic study of the highly rotating fluid (\cite{Ped}),
and partially due to the formal analogue with 3D Euler/Navier-Stokes equations (\cite{CMT}) and its simple form, the SQG equation has received much attention.
For the SQG equation with fractional operator $\LL=|D|^\alpha$, the subcritical case (i.e. $\alpha\in ]1,2]$)
has been known for decades
 that it is globally well-posed for suitably regular data (e.g. \cite{Resn});
while for the subtle critical case (i.e. $\alpha=1$), the issue of global regularity was independently settled by \cite{KNV} and \cite{CV}:
Kiselev et al \cite{KNV} developed an original method
called the ``nonlocal maximum principle"; 
and Caffarelli et al \cite{CV} exploited the De Giorgi's iteration method and a novel extention.
For other different proofs resolving the critical problem, one can refer to \cite{KisN} which uses the duality method,
and \cite{ConV,CTV} which apply the ``nonlinear maximum principle" method.
However, the global regularity issue in the supercritical case remains to be an outstanding open problem.
So far, for the SQG equation with supercritical diffusion (i.e. $\alpha\in ]0,1[$), we only know some partial results:
the local well-posedness result for large data and global well-posedness result under some smallness condition (e.g. \cite{ChenMZ}),
the conditional regularity criterion (e.g. \cite{ConW}), and the eventual regularity of the global weak solution (\cite{Dab,Kis,CzV}).
More precisely, for the eventual regularity issue, which means the global weak solution is smoothly regular after some finite time,
the progress was first made by Dabkowski \cite{Dab} by adapting the duality method of \cite{KisN}
and later achieved by Kiselev \cite{Kis} by using the nonlocal maximum principle method, and one refer to \cite{CzV} for a third proof by applying the method of \cite{CTV}.
Notice that Coti Zelati and Vicol \cite{CzV} also proved a somewhat global result that for $\theta_0\in H^2$ with
$\|\theta_0\|_{L^2}^{\alpha/2}\|\theta_0\|_{\dot H^2}^{1-\alpha/2}\leq R$,
the supercritical SQG equation has a unique global solution as long as $\alpha$ depending on $R$ sufficiently close to $1$.
For the SQG equation with general diffusion operator $\LL$, Dabkowski et al \cite{DKSV} considered \textit{the slightly supercritical case},
where the operator $\LL$ defined by \eqref{Lexp} and \eqref{Kcond3} satisfies \eqref{m-int} below,
and they obtained the global well-posedness of smooth solution by applying the method of nonlocal maximum principle.
They also showed the global result for the multiplier operator $\LL=m(D)$ under some suitable assumptions on $m(\zeta)=m(|\zeta|)$.

The Burgers equation is just the equation \eqref{DDeq} with
\begin{equation}\label{}
d=1,\quad\textrm{and}\quad u=\theta,
\end{equation}
which was studied by Burgers in 1940s as a 1D equation modeling the nonlinearity of 3D Euler/Navier-Stokes equations.
It is known that the inviscid Burgers equation with some smooth data forms the shock singularity at finite time.
For the Burgers equation with fractional diffusion,
the subcritical and critical cases can be treated as the corresponding cases of SQG equation to obtain the global results; while for the supercritical case,
Kiselev et al \cite{KNS} proved that the shock singularity similar to the inviscid case occurs in the supercritical case (see also \cite{DDL,ADV}).
For the Burgers equation with a general $\LL$ defined by \eqref{Lexp} and \eqref{Kcond3}, the authors in \cite{DKSV} proved that under \eqref{m-int} below and other mild conditions on $m$,
the equation is globally well-posed for smooth data; whereas under $\lim_{\nu\rightarrow 0+}\int_\nu^1 m(r^{-1})\dd r <\infty$,
finite time blowup will also happen for some smooth data.

The C\'ordoba-C\'ordoba-Fontelos (CCF) equation corresponds to the equation \eqref{DDeq} with
\begin{equation}\label{}
  d=1,\quad \textrm{and}\quad u=H\theta,
\end{equation}
and $H$ is the usual 1D Hilbert transform.
C\'ordoba et al \cite{CCF} introduced this model as a 1D simple equation of 3D Euler/Navier-Stokes equations which has the nonlocal velocity;
and they proved there exists smooth data so that the inviscid CCF equation forms singularity at finite time.
For the CCF equation with fractional diffusion, Dong \cite{Dong} considered the subcritical and critical cases, and showed the global results,
while in the supercritical case with $\alpha\in ]0,1/2[$, Li et al \cite{LiR}
showed there is an occurrence of finite-time blowup similar to the inviscid case.
Up to now, the problem concerning the global regularity of solution for the supercritical CCF equation with $\alpha\in [1/2,1[$ is still open.
We mention that Do \cite{Do} proved the eventual regularity of the limit function of regularized solution $\theta^\epsilon$ at the supercritical case $\alpha\in ]0,1[$
by applying the method of \cite{Kis}, and also obtained the global well-posedness result of CCF equation at slightly supercritical cases equipped with smooth data.

The incompressible porous media equation is the equation \eqref{DDeq} with the following velocity field
\begin{equation}\label{}
  u=\nabla p + \theta e_d,\qquad \divg u =0,
\end{equation}
where $p$ is a scalar quantity and $e_d$ is the last canonical vector of $\R^d$. By a direct computation, we can show that the velocity $u$ can be exactly expressed as \eqref{uexp},
e.g., for $d=2$ (\cite{CGO}),
\begin{equation*}
  a=\left(0,-\frac{1}{2}\right),\quad S(x)=\frac{1}{2\pi}\left(\frac{2 x_1x_2}{|x|^4}, \frac{x_2^2 -x_1^2}{|x|^4} \right),
\end{equation*}
and for $d=3$ (\cite{CCGO}),
\begin{equation*}
  a=\left(0,0,-\frac{2}{3}\right),\quad S(x)=\frac{1}{4\pi}\left(\frac{3 x_1x_3}{|x|^5}, \frac{3 x_2 x_3}{|x|^5},\frac{2x_3^2 -x_1^2 -x_2^2}{|x|^5} \right).
\end{equation*}
In \cite{CCGO,CGO}, C\'ordoba et al, among other issues, proved the global well-posedness result for the equation in the subcritical and critical cases.
Similarly as the SQG equation, the issue of global regularity in the supercritical case remains unsolved.

In this paper we focus on the drift-diffusion equation \eqref{DDeq}-\eqref{uexp} with general $\LL$ defined by \eqref{Lexp},
and we mainly are concerned with the following cases
\begin{align}
  &\textrm{Case (I): $\big(K(y),m(y)\big)$ satisfies \eqref{Kcond3} and (\textbf{A}1), (\textbf{A}3)}; \label{case3} \\
  &\textrm{Case (II): $\big(K(y),m(y)\big)$ satisfies \eqref{Kcond1}-\eqref{Kcond2} and (\textbf{A}1)-(\textbf{A}2)};\label{case2}  \\
  &\textrm{Case (III): $\big(K(y),m(y)\big)$ satisfies \eqref{Kcond1}, \eqref{Kcond5} and (\textbf{A}1)-(\textbf{A}2), symbol $A(\zeta)\geq 0$, $\divg u=0$}. \label{case1}
\end{align}
By applying the method of nonlocal maximum principle in a unified way,
we show the eventual regularity of global weak solution for the supercritical type equation \eqref{DDeq}-\eqref{uexp} at \textrm{Case (I)}.
Compared with the eventual result obtained in \cite{Kis} for the supercritical SQG equation, we have an explicit control on the eventual regularity time (i.e., the time after which the solution is regular)
which is small enough as $\sigma\rightarrow 0$, $\alpha=1$.
In accordance with this point, we further prove the global regularity result for the logarithmically supercritical drift-diffusion equation \eqref{DDeq}-\eqref{uexp} at either \textrm{Case (II)} or \textrm{Case (III)}.

More precisely, our first main result is the eventual regularity of the vanishing viscosity weak solution for the drift-diffusion equation \eqref{DDeq}-\eqref{uexp}.
\begin{theorem}\label{thmevRe}
  Assume that \textrm{Case (I)} is considered with $\alpha\in ]0,1]$, $\sigma\in[0,1[$, $\theta_0\in L^2(\R^d)$ and $\divg u=0$.
Then for every $T>0$ large, the drift-diffusion equation \eqref{DDeq}-\eqref{uexp} admits a weak solution $\theta \in L^\infty([0,T]; L^2(\R^d))\cap L^2([0,T];\dot H^{\frac{\alpha-\sigma}{2}}(\R^d))$,
which satisfies $\theta\in C^\infty(]t_0+t_1,T]\times \R^d)$,
where $t_0>0$ can be chosen arbitrarily small and $t_1>0$ is a number depending only on
$\alpha$, $\sigma$, $d$, $t_0$ and $\|\theta_0\|_{L^2}$.

Moreover, if $\alpha\in ]0,1[$ and $\sigma=0$ in condition (\textbf{A}3), i.e., $m(y)\equiv C_0 |y|^\alpha$, $\forall y\neq 0$, we can set $T=\infty$, and we explicitly have
\begin{equation}\label{t1-bdd}
  t_1\leq \frac{C}{\alpha} \left( C 2^{d/\alpha} t_0^{-1} \right)^{\frac{d}{2(1-\alpha)}} \Big( \frac{C(1-\alpha)}{\alpha^5}\Big)^{\frac{\alpha}{1-\alpha}}
  \|\theta_0\|_{L^2}^{\frac{\alpha}{1-\alpha}},
\end{equation}
with $C>0$ some constant depending only on $d$.
\end{theorem}

Our second result is the global regularity of the solution for some logarithmically supercritical drift-diffusion equations \eqref{DDeq}-\eqref{uexp}.

\begin{theorem}\label{thm:glb2}
  Assume that either \textrm{Case (II)} or \textrm{Case (III)} is considered for $\alpha=1$ and $\sigma\in [0,1[$ with some constant $c_0=c_0(\sigma)>0$.
Additionally suppose that there exist $\mu\in [0,1]$ and $c_2\geq 1$ such that
\begin{equation}\label{mcd1}
  \frac{1}{c_2} \frac{|y|}{(\log |y|)^\mu}\leq m(y) \leq c_2 |y|,\qquad \forall |y| \geq c_2.
\end{equation}
We have the following two statements.
\begin{enumerate}[(1)]
\item
If \textrm{Case (III)} is considered and $\theta_0\in L^2\cap L^\infty(\R^d)$. Then for any $t_*>0$ small and $T>0$ large,
the vanishing viscosity solution $\theta\in L^\infty([0,T];L^2(\R^d))\cap L^2([0,T]; \dot H^{\frac{1-\sigma}{2}}(\R^d))$ of the drift-diffusion equation \eqref{DDeq}-\eqref{uexp} satisfies $\theta\in C^\infty([t_*, T]\times \R^d)$.
\item
If \textrm{Case (II)} is considered, $\theta_0\in C_0(\R^d)$ (i.e., the space composed of continuous functions which decay to zero at infinity) and let $\theta$ be the limit function of $\theta^\epsilon$ which solves the regularized drift-diffusion equation
\begin{equation}\label{reDDeq}
  \partial_t \theta^\epsilon + u^\epsilon\cdot\nabla \theta^\epsilon + \LL \theta^\epsilon - \epsilon \Delta \theta^\epsilon = 0,\quad u^\epsilon = \mathcal{P}(\theta^\epsilon),\quad  \theta^\epsilon|_{t=0}= \phi_\epsilon*\big(\theta_0 1_{B_{1/\epsilon}}\big),
\end{equation}
where $\epsilon>0$, $\phi\in C_c^\infty(\R^d)$ is the standard mollifier, $\phi_\epsilon(x)= \epsilon^{-d}\phi(x/\epsilon)$, and $1_{B_{1/\epsilon}}$ is the indicator function on the ball $B_{1/\epsilon}$.
Then for any $t_*>0$ small, we have $\theta\in C^\infty([t_*,\infty[\times \R^d)$ and $\theta$ on the time period $[t_*,\infty[$ satisfies the drift-diffusion equation \eqref{DDeq}-\eqref{uexp}.
\end{enumerate}
\end{theorem}

The main method in proving the above results is the nonlocal maximum principle originated from \cite{KNV,Kis},
whose basic idea is to show the evolution strictly preserves some appropriate modulus of continuity (abbr. MOC, see Subsection \ref{sec-MOC} below).

In the proof of Theorems \ref{thmevRe} and \ref{thm:glb2}, the following proposition concerned with the uniform-in-$\epsilon$ improvement of the eventual H\"older regularity from the $L^\infty$-estimate plays a core role.

\begin{proposition}\label{prop-evRe}
  Assume that \textrm{Case (I)} is considered with $\alpha\in ]0,1]$, $\sigma\in[0,1[$, and $\theta^\epsilon\in C([0,\infty[; H^s(\R^d))$, $s>1+\frac{d}{2}$
is a smooth solution for the regularized drift-diffusion equation \eqref{reDDeq} with $\epsilon>0$, $\theta_0\in L^\infty(\R^d)$.
Then there exists a time $t_1>0$ independent of $\epsilon$ 
such that for every $\beta \in ]1-\alpha+\sigma,1[$,
\begin{equation}\label{Cbetaest0}
  \sup_{t\in [t_1,\infty[}\|\theta^\epsilon(t)\|_{\dot C^\beta(\R^d)}\leq C(\|\theta_0\|_{L^\infty},d,\alpha,\beta,\sigma),
\end{equation}
with $C $ independent of $\epsilon$.
Moreover, if $\alpha\in ]0,1[$ and $\sigma=0$ in the condition (A3), we have the explicit estimates on $t_1$ and $\sup_{t\in [t_1,\infty[}\|\theta^\epsilon(t)\|_{\dot C^\beta}$
as \eqref{t1-bdd2}-\eqref{Cbet-bdd} below.
\end{proposition}

For the proof of Proposition \ref{prop-evRe}, the new ingredient is the MOC  $\omega(\xi,\xi_0)$ given by \eqref{MOC3.1}-\eqref{MOC3.2},
which is derived from suitably modifying the MOC $\omega(\xi)$ defined by \eqref{MOC2}, 
and by virtue of a careful analysis according to the values of $\xi$ and $\xi_0$,
we manage to show that
the solution $\theta^\epsilon(x,t)$ of the regularized equation \eqref{reDDeq} uniformly-in-$\epsilon$ strictly obeys the MOC $\omega(\xi,\xi_0(t))$,
which combined with the regularity preservation criterion in terms of MOC \eqref{MOC2} (see Lemma \ref{lem-HoldRC}) further guarantees the desired uniform-in-$\epsilon$ H\"older regularity estimate after some time.
We stress that there is no factor like $1-\alpha+\sigma$ or $1-\alpha$
in the conditions of $\kappa,\gamma,\rho$ (see \eqref{rkg-cdsum}) appearing in $\omega(\xi,\xi_0)$,
so that we can estimate the eventual regularity time $t_1$ as \eqref{t1-bdd} in the case $\alpha\in ]0,1[$, $\sigma=0$,
which has the property that $t_1\rightarrow 0$ as $\alpha\rightarrow 1$ for the fixed data $\theta_0$.

For the proof of Theorem \ref{thmevRe}, we first prove the global existence of a vanishing viscosity solution satisfying the $L^2$-energy estimate,
then by using De Giorgi's method we show the crucial $L^\infty_x$-improvement for all $t\geq t_0$ with any $t_0>0$,
and then Proposition \ref{prop-evRe} ensures the eventual H\"older regularity of this weak solution for every $t\geq t_0+t_1$ with some $t_1>0$,
which in combination with the regularity criterion Lemma \ref{lem-RC} further leads to the desired eventual regularity result.

For Theorem \ref{thm:glb2}, we observe that under the condition \eqref{mcd1}, the eventual regularity time $t_1$ can be arbitrarily small,
and thus by applying Proposition \ref{prop-evRe} and by appropriately choosing the coefficients in the MOC $\omega(\xi,\xi_0)$ and $\xi_0=\xi_0(t)$, we can show the desired global regularity result.
Notice that in the considered cases it suffices to justify the criterion \eqref{Targ4} for small $\xi$ and $\xi_0(t)$,
so that we can treat the more general diffusion operator $\LL$ than that in Proposition \ref{prop-evRe}.

Next we list some remarks as follows.

\begin{remark}\label{rmk-glb2}
Since $m(y)=\frac{|y|}{(\log(\lambda+|y|))^\mu}$ with $\mu\in[0,1]$, $\lambda\geq 0$ satisfies \eqref{mDec} with $\alpha=1$, $\sigma\in ]0,1[$, $c_0=e^{-\frac{\mu}{\sigma}}$,
and also satisfies \eqref{mcd1} with $c_2=2$,
thus Theorem \ref{thm:glb2} can be applied to the drift-diffusion equation \eqref{DDeq}-\eqref{uexp} under either \textrm{Case (II)} or \textrm{Case (III)} with these $m$ and $c_0$.
Recalling that the improvement from $L^\infty$ to H\"older regularity is a crucial step in proving the global regularity of weak solution for the critical SQG equation (i.e. $\LL=|D|$) by Caffarelli-Vasseur \cite{CV}
and also Kiselev-Nazarov \cite{KisN}, we here as a nontrivial generalization achieve such an improvement for vanishing viscosity solution of the drift-diffusion equation \eqref{DDeq}-\eqref{uexp} at some logarithmically supercritical cases,
and we even remove the divergence-free assumption of the velocity field at \textrm{Case (II)}.
\end{remark}

\begin{remark}\label{rmk:glob}
As a counterpart of Theorem \ref{thm:glb2}, we can also prove the following global well-posedness result for \eqref{DDeq}-\eqref{uexp} at the slightly supercritical case complemented with regular data:
assume that $\theta_0\in H^s(\R^d)$, $s>\frac{d}{2}+1$, and either \textrm{Case (II)} or \textrm{Case (III)} is considered with $\alpha\in ]0,1]$, $\sigma\in[0,1[$, $c_0=c_0(\alpha,\sigma)>0$,
and additionally
\begin{equation}\label{m-int}
  \lim_{\nu\rightarrow 0+} \int_\nu^{c_0} m(\xi^{-1})\dd \xi =\infty,
\end{equation}
then the associated drift-diffusion equation \eqref{DDeq}-\eqref{uexp} generates a uniquely global smooth solution
$\theta \in C([0,\infty[; H^s(\R^d)) \cap C^\infty (]0,\infty[\times \R^d)$. We shall justify this statement at the appendix section.
This global well-posedness result is concerned with the slightly supercritical drift-diffusion equation \eqref{DDeq}-\eqref{uexp},
and it generalizes the corresponding result of \cite{DKSV} on the slightly supercritical SQG and Burgers equations.
Note also that the MOC given by \eqref{MOC4.3} has a simper form than that appeared in \cite{DKSV}, and we use a different way to estimate the contribution \eqref{Ixi} so that we can avoid the difficulty encountered in considering the general $u$ defined by \eqref{uexp}.
\end{remark}

\begin{remark}
  Motivated by Coti Zelati and Vicol \cite{CzV} and in a different method, we can also prove the following global result: assume that either \textrm{Case (II)} or \textrm{Case (III)} is considered for $\alpha=1$ and $\sigma\in [0,1[$ with some $c_0>0$ (independent of $\sigma$),
and let $\theta_0\in H^s(\R^d)$, $s>1+\frac{d}{2}$ be satisfying $\|\theta_0\|_{H^s(\R^d)}\leq R$ with some $R>0$,
then there exists a constant $\sigma_1=\sigma_1(R,d)>0$ such that for every $\sigma\leq \sigma_1$, the associated drift-diffusion equation \eqref{DDeq}-\eqref{uexp}
has a unique global solution $\theta(x,t)\in C([0,\infty[;H^s(\R^d))\cap C^\infty([0,\infty[\times\R^d)$. Indeed, the classical local well-posedness result first ensures that there is $T_1=T_1(d,R)>0$ such that the equation \eqref{DDeq}-\eqref{uexp} admits a smooth solution $\theta$ on $[0,T_1]$;
then similarly as obtaining \eqref{t1-bdd} and \eqref{t1-est} (we also adopt the different points in proving \eqref{Targ4} compared with proving \eqref{Targ3}), one can show that the eventual time $t_1\rightarrow 0$ as $\sigma\rightarrow 0$, which implies $t_1<T_1$ for $\sigma$ small enough, and thus we conclude the statement.
\end{remark}


The outline of the paper is as follows. In Section \ref{sec-prel}, we introduce a class of multiplier operators as examples of the diffusion operator $\LL$, and we present some useful auxiliary lemmas,
and we also collect the definition and useful lemmas related to the modulus of continuity.
In Section \ref{sec-evRe}, we give the detailed proof of Proposition \ref{prop-evRe}.
The proof of Theorems \ref{thmevRe} and \ref{thm:glb2} are respectively placed in the subsections of Section \ref{sec-thms}.
At last, the appendix section justifies the statement in Remark \ref{rmk:glob}.

Throughout this paper, $C$ stands for a constant which may be different from line to line.
The notation $X\lesssim Y$ means that $X\leq CY$. Denote by $B_r(x_0):=\{x\in \mathbb{R}^d: |x-x_0|< r\}$ the ball of $\mathbb{R}^d$ and we abbreviate $B_r(0)$ as $B_r$.
Denote $\mathcal{S}'(\mathbb{R}^d)$ the space of tempered distributions.
We use $\widehat{f}$ and $\check{g}$ to denote the
Fourier transform and the inverse Fourier transform of a tempered distribution, that is, $\widehat f(\zeta)=\int_{\mathbb{R}^d}e^{-i x\cdot \zeta} f(x)\mathrm{d} x$
and $\check{g}(x) = \frac{1}{(2\pi)^d}\int_{\R^d} e^{i x \cdot\zeta} g(\zeta)\dd \zeta  $.

\section{Preliminary and auxiliary lemmas}\label{sec-prel}

In this section, we introduce a class of multiplier operators as examples of the operator $\LL$,
and also compile some useful auxiliary lemmas.

\subsection{Multiplier operators as examples of $\LL$}\label{subsec-pre1}
In addition to the conditions (\textbf{A}1)-(\textbf{A}2) stated in the introduction, we assume that $m(\zeta)=m(|\zeta|)$ also may satisfy the following assumptions:
\\
(\textbf{A}4)\; $m$ is of the Mikhlin-H\"ormander type, i.e. there is some constant $c_3\geq 1$ so that
\begin{equation}\label{MHc}
  |\partial_\zeta^k m(\zeta)| \leq c_3 |\zeta|^{- k} m(\zeta),\quad \forall \zeta\neq 0,
\end{equation}
for all $k\in\N$ and $k\leq k_0$, with $k_0$ a positive constant depending only on $d$.
\\
(\textbf{A}5)\; $m$ satisfies that
\begin{equation}\label{InvMH}
  (-\Delta)^d m(\zeta)\geq c_4 |\zeta|^{-2d} m(\zeta),\quad \forall |\zeta|\; \textrm{large enough},
\end{equation}
with some $c_4>0$.
\\
(\textbf{A}6)\;
$m$ satisfies that
\begin{equation}\label{PosC}
  (-1)^{k-1}m^{(k)}(|\zeta|)\geq 0,\quad \forall |\zeta|>0, \, k\in\{1,2,\cdots,d\}.
\end{equation}

Note that there do exist a large class of nontrivial examples satisfying all the needing conditions;
in fact, as shown by \cite[Proposition 3.6]{Hmidi}, the functions $m(\zeta)= \frac{|\zeta|^\alpha}{\left(\log(\lambda+|\zeta|)\right)^\beta}$ with
$\lambda\geq e^{\frac{3+2\beta}{\alpha}}$, $\alpha\in ]0,1]$, $\beta\geq 0$ and $d=1,2,3$ satisfy \eqref{PosC}, and they also satisfy (\textbf{A}1)-(\textbf{A}2), (\textbf{A}4)-(\textbf{A}5) by a direct computation.

The following lemma relates the multiplier operator with the conditions of $K$ in the introduction.
\begin{lemma}\label{lem-Linteg}
  Suppose that $m(\zeta)=m(|\zeta|)$ is a radial function satisfies the conditions (\textbf{A}1)-(\textbf{A}2), (\textbf{A}4)-(\textbf{A}5). Then the multiplier operator $m(D)$
has the following representation formula
\begin{equation}\label{mDexp}
  m(D)\theta(x)= \left( m(\zeta) \widehat{\theta}(\zeta)\right)^\vee(x) = \mathrm{p.v.}\int_{\R^d} K(y)\left( \theta(x)-\theta(x+y)\right) \,\mathrm{d}y,
\end{equation}
where the radial kernel $K$ satisfies
\begin{equation}\label{Kest1}
  |K(y)|\leq C |y|^{-d} m(|y|^{-1}), \quad \forall |y|>0,
\end{equation}
and
\begin{equation}\label{Kest2}
  K(y) \geq c_5 |y|^{-d} m(|y|^{-1}),\quad \forall\, 0<|y|\leq c_0,
\end{equation}
with two generic constants $c_0,c_5>0$.
Besides, if $m(\zeta)=m(|\zeta|)$ additionally satisfies the condition (\textbf{A}6), then the kernel function $K$ in \eqref{mDexp} also satisfies
\begin{equation}\label{Kest3}
  K(y)\geq 0,\quad \forall |y|>0.
\end{equation}
\end{lemma}
Notice that \eqref{Kest1}-\eqref{Kest2} just correspond to \eqref{Kcond1}, \eqref{Kcond5},
and \eqref{Kest1}-\eqref{Kest3} correspond to \eqref{Kcond1}-\eqref{Kcond2}.

\begin{proof}[Proof of Lemma \ref{lem-Linteg}]
The properties \eqref{Kest1}-\eqref{Kest2} were proved in \cite[Lemmas 5.1, 5.2]{DKSV}. We only prove \eqref{Kest3}. By arguing as \cite[Proposition 3.6 and Lemma 3.8]{Hmidi},
we can show that, thanks to (\textbf{A}6), the kernel function $G_t(x)$ associated with the operator $e^{-t \LL}$ satisfies
\begin{equation*}
  G_t(x)\geq 0,\quad\textrm{and}\quad \int_{\R^d} G_t(x)\,\mathrm{d}x = \widehat{G_t}(\cdot)|_{\zeta=0} = 1.
\end{equation*}
In light of the semigroup representation formula of the operator $\LL$,
\begin{equation*}
  \LL f(x) = \lim_{t\rightarrow 0+}\frac{f(x)-e^{-t \LL}f(x)}{t}= \lim_{t\rightarrow 0+} \int_{\R^d} \frac{G_t(y)}{t}\left( f(x)-f(x+y)\right)\,\mathrm{d}y,
\end{equation*}
we see that $K(y)=\lim_{t\rightarrow 0} \frac{G_t(y)}{t} \geq 0$ for all $|y|>0$.
\end{proof}

\subsection{Auxiliary lemmas}


First we give a useful lemma on the function $m$ satisfying \eqref{mDec}.
\begin{lemma}\label{lem-mfact}
Let $m(y)=m(|y|)$ be the radial function satisfying the condition \eqref{mDec} for some $\alpha\in ]0,1[$ and $\sigma\in [0,\alpha[$, then
\begin{equation}\label{mProp1}
  \textrm{the mapping $|y|\mapsto |y|^{\beta_1}m(|y|^{-1})$, $\beta_1\geq \alpha$ is non-decreasing,}
\end{equation}
and
\begin{equation}\label{mProp2}
  \textrm{the mapping $|y|\mapsto |y|^{\beta_2}m(|y|^{-1})$, $\beta_2\leq \alpha-\sigma$ is non-increasing.}
\end{equation}
\end{lemma}

\begin{proof}[Proof of Lemma \ref{lem-mfact}]
  Let $f_i(r)=r^{\beta_i} m(r^{-1})$ for $i=1,2$ and $r>0$, then by the direct computation,
\begin{equation*}
  f_1'(r)= r^{\beta_1-1}\left( \beta_1 m(r^{-1})-r^{-1} m'(r^{-1})\right)\geq (\beta_1-\alpha)r^{\beta_1-1} m(r^{-1})\geq 0,
\end{equation*}
which yields \eqref{mProp1}, and similarly,
\begin{equation*}
  f_2'(r)=r^{\beta_2-1}\left( \beta_2 m(r^{-1}) -r^{-1} m'(r^{-1})\right)\leq \left(\beta_2 -(\alpha-\sigma)\right) r^{\beta_2-1} m(r^{-1})\leq 0,
\end{equation*}
which yields \eqref{mProp2}.
\end{proof}

The next lemma concerns the pointwise lower bound estimate of the symbol of the operator $\LL$. 
\begin{lemma}\label{lem-symb}
  Let $\LL$ be defined by \eqref{Lexp} with $K(y)$ satisfying \eqref{Kcond1}-\eqref{Kcond2} and $m(y)$ satisfying (\textbf{A}1)-(\textbf{A}2),
then the associated symbol $A(\zeta)$ given by \eqref{LKf} satisfies that
\begin{equation}\label{Aest}
  A(\zeta)\geq  C^{-1} |\zeta|^{\alpha-\sigma} - C,\quad \forall \zeta\in\R^d,
\end{equation}
where $\alpha\in ]0,1]$, $\sigma\in [0,\alpha[$ and $C$ is a positive constant depending only on $d$, $\alpha$ and $\sigma$. Besides, if $K(y)$ satisfies \eqref{Kcond1}, \eqref{Kcond5} with $m(y)$ satisfying (\textbf{A}1)-(\textbf{A}2),
we can also get \eqref{Aest} with a different constant $C$. In particular, if $K(y)$ satisfies \eqref{Kcond3} with $m(y)=|y|^\alpha$ ($\alpha\in ]0,1]$), $\forall y\neq 0$, we get
\begin{equation}\label{Aest2}
  A(\zeta)\geq  C^{-1} |\zeta|^{\alpha},\quad \forall \zeta\in\R^d,
\end{equation}
with $C$ a positive constant depending only on $d$ and $\alpha$.
\end{lemma}

Note that if $m(y)\equiv |y|^\alpha$, then we can get \eqref{Aest} with $\sigma=0$ for the associated operator $\LL$,
and this special result in fact has appeared in the literatures, e.g. \cite[Lemma 2.2]{ChaM}.

\begin{proof}[Proof of Lemma \ref{lem-symb}]
  Recalling that for every $\alpha\in ]0,2[$ we have (e,g, see \cite[Eq. (3.219)]{Jacob})
\begin{equation}\label{eq:fact1}
  |\zeta|^\alpha = c_{d,\alpha}\;   \int_{\R^d\setminus\{0\}}\left( 1-\cos(y\cdot \zeta)\right) \frac{1}{|y|^{d+\alpha}} \dd y,\quad \forall \zeta\in\R^d
\end{equation}
and by virtue of the lower bound of $K$ in \eqref{Kcond1}-\eqref{Kcond2} and the fact $|y|^{\alpha-\sigma}m(|y|^{-1})\geq  c_0^{\alpha-\sigma}m(c_0^{-1})$ for all $0<|y|\leq c_0$, we obtain
\begin{align*}
  A(\zeta) 
  & \geq  c_1^{-1}\;  \int_{0<|y|\leq c_0} \left( 1-\cos(y\cdot \zeta)\right) \frac{m(|y|^{-1})}{|y|^d} \dd y \\
  & \geq c_1^{-1} c_0^{\alpha-\sigma} m(c_0^{-1})  \int_{0<|y|\leq c_0} \left( 1-\cos(y\cdot \zeta)\right) \frac{1}{|y|^{d+(\alpha-\sigma)}}\dd y \\
  & \geq c_1^{-1}c_0^{\alpha-\sigma} m(c_0^{-1})  \Big(  c_{d,\alpha}^{-1} |\zeta|^{\alpha-\sigma}
  - \int_{|y|\geq c_0}\frac{1}{|y|^{d+\alpha-\sigma}}\dd y\Big)  \\
  & \geq C^{-1} |\zeta|^{\alpha-\sigma} - C.
\end{align*}
If $K$ satisfies \eqref{Kcond1} and \eqref{Kcond5}, we similarly deduce
\begin{align*}
  A(\zeta) & \geq c_1^{-1}\int_{0<|y|\leq c_0} \left( 1-\cos(y\cdot\zeta)\right) \frac{m(|y|^{-1})}{|y|^d}\dd y - c_1 \int_{|y|\geq c_0} \left( 1-\cos(y\cdot \zeta) \right) |K(y)| \dd y \\
  & \geq c_1^{-1}c_0^{\alpha-\sigma} m(c_0^{-1}) \int_{0<|y|\leq c_0} \left( 1-\cos(y\cdot\zeta)\right) \frac{1}{|y|^{d+\alpha-\sigma}}\dd y
  - c_1 c_0^{\alpha-\sigma} m(c_0^{-1})  \int_{|y|\geq c_0}  \frac{1}{|y|^{d+\tilde\alpha}} \dd y \\
  & \geq C^{-1} |\zeta|^{\alpha-\sigma} -C .
\end{align*}
If $K$ satisfies \eqref{Kcond3} with $m(y)=|y|^\alpha$ ($\alpha\in ]0,1]$), $\forall y\neq 0$, from \eqref{eq:fact1}
we see that $A(\zeta)\geq  c_1^{-1} c_{d,\alpha}^{-1} |\zeta|^{\alpha}$, which leads to \eqref{Aest2}.
\end{proof}

The following lemma is about the $L^\infty$-estimate of smooth solution for the equation \eqref{DDeq}-\eqref{uexp}.
\begin{lemma}\label{lem-theMP}
Let $\theta\in C([0,T^*[; H^s(\R^d))$, $s>1+\frac{d}{2}$ be a smooth solution to the drift-diffusion equation \eqref{DDeq}-\eqref{uexp}.
If {\textrm Case (II)} (i.e. \eqref{case2}) is supposed, then we have
\begin{equation}\label{theMP1}
  \|\theta(t)\|_{L^\infty}\leq \|\theta_0\|_{L^\infty}, \quad \textrm{for all   } t\in [0,T^*[.
\end{equation}
Besides, if {\textrm Case (III)} (i.e. \eqref{case1}) is assumed, we get
\begin{equation}\label{theMP2}
  \|\theta(t)\|_{L^\infty} \leq C(\|\theta_0\|_{L^2\cap L^\infty}, \alpha,\sigma,d),  \quad \textrm{for all   } t\in [0,T^*[.
\end{equation}

\end{lemma}

\begin{proof}[Proof of Lemma \ref{lem-theMP}]
Due to that the kernel $K$ is nonnegative on $\R^d\setminus\{0\}$, the proof of \eqref{theMP1} is classical (cf. \cite[Theorem 4.1]{CorC} for $\LL=|D|^\alpha$), and we here omit the details.

Next we prove \eqref{theMP2}. Thanks to the assumptions that $\mathrm{div}\,u =0$ and $A(\zeta)\geq 0$, by the $L^2$-energy estimate (cf. \eqref{ene-est2} below), we get
$\|\theta(t)\|_{L^2_x} \leq \|\theta_0\|_{L^2}$ for all $t\in [0,T^*[$. Now for every $t\in ]0,T^*[$, assume that $x_t\in \R^d$ is some point satisfying $\theta(x_t,t)=\|\theta(t)\|_{L^\infty_x}=:M(t)$.
According to
\begin{equation*}
  \Big|\Big\{y\in \R^d: |\theta(x_t +y)|\geq \frac{M(t)}{2}\Big\}\Big|\leq \Big(\frac{2 \|\theta(t)\|_{L^2}}{M(t)}\Big)^2 \leq  \frac{4 \|\theta_0\|_{L^2}^2}{M(t)^2} ,
\end{equation*}
and denoting by $r_t:= \frac{4^{1/d}}{|B_1(0)|^{1/d}}\frac{\|\theta_0\|_{L^2}^{2/d}}{M(t)^{2/d}}$,
we may set $M(t)$ large enough so that $r_t\leq \frac{c_0}{2}$. Taking advantage of \eqref{Kcond1}, \eqref{Kcond5} and Lemma \ref{lem-mfact},
we find that (by arguing as \cite[Lemma 4.1]{KisN}),
\begin{align*}
  (\LL\theta)(x_t,t) & \geq\, c_1^{-1} \int_{0<|y|\leq c_0} (\theta(x_t,t)-\theta(x_t+y,t)) \frac{m(|y|^{-1})}{|y|^d}\dd y - 2c_1 M(t) \int_{|y|\geq c_0}\frac{1}{|y|^{d+\tilde\alpha}}\dd y \\
  & \geq \, c_1^{-1} \frac{M(t)}{2} \int_{r_t \leq |y|\leq c_0} \frac{m(|y|^{-1})}{|y|^d} \dd y  -2 c_1 M(t) \int_{|y|\geq c_0} \frac{1}{|y|^{d+\tilde\alpha}}\dd y \\
  & \geq \, c_1^{-1}c_0^{\alpha-\sigma} m(c_0^{-1}) \frac{M(t)}{2}\int_{r_t \leq |y|\leq c_0} \frac{1}{|y|^{d+\alpha-\sigma}}\dd y - 2 c_1 M(t)\int_{|y|\geq c_0} \frac{1}{|y|^{d+\tilde\alpha}}\dd y \\
  & \geq \, c_1^{-1} c_0^{\alpha-\sigma} m(c_0^{-1})\frac{M(t)}{2}  \frac{|B_1(0)|}{\alpha-\sigma} \frac{1}{2 r_t^{\alpha-\sigma}} - 2c_1 M(t) \frac{|B_1(0)|}{\tilde\alpha} \\
  & = \, \frac{C_{\alpha,\sigma,d}}{\|\theta_0\|_{L^2}^{2(\alpha-\sigma)/d}} M(t)^{1+ \frac{2(\alpha-\sigma)}{d}}- C_{\tilde\alpha,d} M(t).
\end{align*}
Hence we see that
\begin{equation*}
  \frac{d}{dt}M(t)\leq - C_{\alpha,\sigma,d} \|\theta_0\|_{L^2}^{-\frac{2(\alpha-\sigma)}{d}} M(t)^{1+\frac{2(\alpha-\sigma)}{d}} + C_{\tilde\alpha,d}M(t),
\end{equation*}
and for $M(t)$ larger than the quantity $\|\theta_0\|_{L^2} \left( \frac{C_{\tilde\alpha,d}}{C_{\alpha,\sigma,d}}\right)^{\frac{d}{2(\alpha-\sigma)}}$, we have $\frac{d}{dt}M(t)\leq 0$,
which readily implies that $M(t)\leq \max\Big\{\|\theta_0\|_{L^\infty},\left( \frac{C_{\tilde\alpha,d}}{C_{\alpha,\sigma,d}}\right)^{\frac{d}{2(\alpha-\sigma)}}\|\theta_0\|_{L^2} \Big\}$
and concludes the lemma.
\end{proof}

Finally, we state the following key regularity criterion for the drift-diffusion equation \eqref{DDeq}.
\begin{lemma}\label{lem-RC}
\begin{enumerate}[(1)]
\item
Assume that Case (III) is considered, $\theta_0\in L^p(\mathbb{R}^d)$ for some $p\in [2,\infty[$. If the drift $u$ satisfies that for any $T>0$,
\begin{equation}\label{ucd1}
  u\in L^\infty([0,T]; C^\delta(\R^d)),\quad \textrm{for every  } \delta\in ]1-\alpha+\sigma,1[,
\end{equation}
then the drift-diffusion equation \eqref{DDeq} admits a unique weak solution (in the distributional sense) $\theta\in L^\infty([0,T]; L^p(\R^d))$ which satisfies
$\theta\in L^{\infty}(]0,T],C^{1,\gamma}(\mathbb{R}^d))$ with any $\gamma \in ]0,\delta+\alpha-\sigma-1[$.
Moreover, if the drift field $u$ is given by \eqref{uexp}, we have $\theta\in C^\infty(]0,T]\times\R^d)$.
\item
  Suppose that Case (II) is considered, $\theta_0\in C_0(\mathbb{R}^d)$, and the drift $u$ satisfy \eqref{ucd1} for any $T>0$.
Then for the following approximate equation of the drift-diffusion equation \eqref{DDeq}:
\begin{equation*}
  \partial_t \theta^\epsilon + u^\epsilon \cdot\nabla \theta^\epsilon +\mathcal{L} \theta^\epsilon=0,\quad u^\epsilon= \phi_\epsilon*u ,\quad \theta^\epsilon|_{t=0} = \theta_01_{B_{1/\epsilon}}(x) ,
\end{equation*}
with $\phi_\epsilon(x)=\epsilon^{-d}\phi(x/\epsilon)$, $\phi$ the standard mollifier and $1_{B_{1/\epsilon}}$ the indicator function on the ball $B_{1/\epsilon}$, the corresponding regularized solution $\theta^\epsilon$ uniformly-in-$\epsilon$ satisfies that $\theta^\epsilon\in L^\infty([0,T]; C_0(\R^d))\cap L^{\infty}(]0,T],\,C^{1,\gamma}(\mathbb{R}^d))$ for any $\gamma\in ]0,\delta+\alpha-\sigma-1[$. Moreover, if the drift field $u$ is given by \eqref{uexp}, we have $\theta^\epsilon\in C^\infty(]0,T] \times\R^d)$ uniformly in $\epsilon$.
\end{enumerate}
\end{lemma}

For the proof of Lemma \ref{lem-RC}, one can refer to \cite[Theorems 1.1, 1.2 and Remark 1.3]{XueY} for the detailed proof of the same result for the drift-diffusion equation \eqref{DDeq} with more general L\'evy-type operator $\mathcal{L}$.

\subsection{Modulus of Continuity}\label{sec-MOC}

In this subsection we gather some results related to the modulus of continuity, which play an important role on the method of nonlocal maximum principle.

First is the definition of the modulus of continuity.
\begin{definition}
  A function $\omega:[0,\infty[ \rightarrow [0,\infty[$ is called a modulus of continuity (abbr. MOC) if $\omega$ is continuous on $]0,\infty[$, nondecreasing, concave, and
  piecewise $C^2$ with one-sided derivatives defined at every point in $]0,\infty[$.
  We say a function $f:\mathbb{R}^d\rightarrow \mathbb{R}^l$ obeys the modulus of continuity $\omega$ if $|f(x)-f(y)| \leq \omega(|x-y|)$
  for all $x, y\in \mathbb{R}^d$, and say $f:\mathbb{R}^d\rightarrow \mathbb{R}^l$ strictly obeys the modulus of continuity $\omega$ if the above inequality is strict for every $x\neq y\in \mathbb{R}^d$.
\end{definition}

Then we recall the general criterion of the nonlocal maximum principle for the whole-space drift-diffusion equation
(for the proof see \cite[Proposition 3.2]{MX-gQG} or \cite[Theorem 2.2]{Kis}).

\begin{proposition}\label{prop-GC}
 Let $\theta\in C([0,\infty[;H^s(\R^d))$, $s>\frac{d}{2}+1$ be a smooth solution of the following whole space drift-diffusion equation
\begin{equation}\label{appActS}
 \partial_t \theta + u\cdot\nabla \theta +  \mathcal{L} \theta -\epsilon \Delta \theta=0,\quad \theta(0,x)=\theta_0(x), \; x\in\mathbb{R}^d,
\end{equation}
 with $\epsilon\geq 0$. Assume that
\\
(1) for every $t\geq 0$, $\omega(\xi,t)$ is a MOC and satisfies that its inverse function $\omega^{-1}(3\|\theta(\cdot,t)\|_{L^\infty_x},t)<\infty$;
\\
(2) for every fixed point $\xi$, $\omega(\xi,t)$ is piecewise $C^1$ in the time variable with one-sided derivatives defined at each point,
 and that for all $\xi$ near infinity, $\omega(\xi,t)$ is continuous in $t$ uniformly in $\xi$;
\\
(3) $\omega(0+,t)$ and $\partial_\xi \omega(0+,t)$ are continuous in $t$ with values in $\R\cup\{\pm \infty\}$,
 and satisfy that for every $t\geq 0$, either $\omega(0+,t)>0$ or $\partial_\xi\omega( 0+,t)=\infty$ or $\partial_{\xi\xi}\omega (0+,t) =-\infty$.

Let the initial data $\theta_0(x)$ strictly obey $\omega(\xi,0)$, then for every $T>0$, $\theta(x,T)$ strictly obeys the modulus of continuity $\omega(\xi,T)$
provided that for all $t\in ]0, T]$ and $\xi\in \set{\xi>0:\omega(\xi,t)\leq 2\| \theta(\cdot,t)\|_{L^\infty_x}}$, $\omega(\xi,t)$ satisfies
\begin{equation}\label{keyGC}
  \partial_t \omega(\xi,t)> \Omega(\xi,t)\,\partial_\xi\omega(\xi,t) +  D(\xi,t) + 2\epsilon \partial_{\xi\xi}\omega(\xi,t),
\end{equation}
where $\Omega(\xi,t)$ and $D(\xi,t)$ are respectively defined from that for every $x\in\mathbb{R}^d$ and every unit vector $e\in \mathbb{S}^{d-1}$ in \eqref{scena} (noting that we suppress the dependence of $x,e$ in $\Omega(\xi,t)$ and $D(\xi,t)$),
\begin{equation}\label{GC-Om}
  \qquad\quad\Omega(\xi,t):= |(u(x+\xi e,t)-u(x,t))\cdot e|, \quad \textrm{and }
\end{equation}
\begin{equation}\label{GC-D}
  D(\xi,t):= -\big( \LL \theta(x,t) - \LL\theta(x+\xi e,t)\big),
\end{equation}
under the scenario that
\begin{equation}\label{scena}
\begin{split}
  \theta(x,t)-\theta(x+\xi e ,t) = \omega(\xi,t),\quad &\textrm{and}\\
  |\theta(y,t)-\theta(z,t)| \leq \omega(|y-z|,t), \quad &\forall y,z\in\R^d .
\end{split}
\end{equation}

In \eqref{keyGC}, at the points where $\partial_t\omega(\xi,t)$ (or $\partial_\xi\omega(\xi,t)$) does not exist,
the smaller (or larger) value of the one-sided derivative should be taken.
\end{proposition}

The following lemma is concerned with the estimate of \eqref{GC-D} under the scenario \eqref{scena}.
\begin{lemma}\label{lem-mocdiss}
Assume that the diffusion operator $\LL$ is defined by \eqref{Lexp} with the radial kernel $K$, then we have the following estimates on $D(\xi,t)$ defined by \eqref{GC-D} under the scenario \eqref{scena}.
\begin{enumerate}[(1)]
\item
If $K$ satisfies \eqref{Kcond3} with $m$ satisfying (\textbf{A}1) and (\textbf{A}3), then for any $\xi>0$,
\begin{equation}\label{D2}
\begin{split}
  D(\xi,t)\leq &
  \,C_1 \int_0^{\frac{\xi}{2}}\left(\omega(\xi+2\eta,t)+\omega (\xi-2\eta,t) -2\omega (\xi,t)\right) \frac{m(\eta^{-1})}{\eta}\textrm{d} \eta \\
  & + C_1 \int_{\frac{\xi}{2}}^{\infty} \left(\omega (2\eta+\xi,t)-\omega (2\eta-\xi,t) -2\omega (\xi,t)\right) \frac{m(\eta^{-1})}{\eta}\textrm{d}\eta,
\end{split}
\end{equation}
with $C_1>0$ a constant depending only on $d$.
\item
If $K$ satisfies \eqref{Kcond1}-\eqref{Kcond2} with $m$ satisfying (\textbf{A}1)-(\textbf{A}2), then for every $\xi\in ]0,\frac{c_0}{2}]$,
\begin{equation}\label{Dxi-est}
\begin{split}
  D(\xi,t)\leq  &
  \,C_1 \int_0^{\frac{\xi}{2}}\left(\omega(\xi+2\eta,t)+\omega (\xi-2\eta,t) -2\omega (\xi,t)\right) \frac{m(\eta^{-1})}{\eta}\textrm{d} \eta \\
  & + C_1 \int_{\frac{\xi}{2}}^{\frac{c_0}{2}} \left(\omega (2\eta+\xi,t)-\omega (2\eta-\xi,t) -2\omega (\xi,t)\right) \frac{m(\eta^{-1})}{\eta}\textrm{d}\eta.
\end{split}
\end{equation}
\item
If $K$ satisfies \eqref{Kcond1}, \eqref{Kcond5} with $m$ satisfying (\textbf{A}1)-(\textbf{A}2), then for every $\xi\in]0, \frac{c_0}{2}]$,
\begin{equation}\label{Dxi-est2}
  D(\xi,t)\leq  \,  C_1' \omega(\xi,t)+ \textrm{R.H.S. of \eqref{Dxi-est}} ,
\end{equation}
where $C'_1>0$ is a constant depending on $d$, $\tilde{\alpha}$, $c_0$ and $c_1$.
\end{enumerate}

\end{lemma}

\begin{proof}[Proof of Lemma \ref{lem-mocdiss}]
  According to \eqref{Lexp} and \eqref{scena}, we see that
\begin{equation}\label{Dexp}
  D(\xi,t)= \int_{\R^d} K(y) \left( \theta(x+y,t)-\theta(x+\xi e+ y,t)-\omega(\xi,t)\right)\dd y,
\end{equation}
where the integral will be understood in the sense of principle value if needed.
By arguing as the proof of \cite[Lemma 2.3]{DKSV}, we get
\begin{equation}\label{D3}
\begin{split}
  D(\xi,t) \leq &
  \int_0^{\frac{\xi}{2}}\left(\omega(\xi+2\eta,t)+\omega (\xi-2\eta,t) -2\omega (\xi,t)\right) \widetilde{K}(\eta)\textrm{d} \eta \\
  & + \int_{\frac{\xi}{2}}^{\infty} \left(\omega (2\eta+\xi,t)-\omega (2\eta-\xi,t) -2\omega (\xi,t)\right)\widetilde{K}(\eta) \textrm{d}\eta,
\end{split}
\end{equation}
with
$\widetilde{K}(\eta)=\int_{\R^{d-1}}K(\eta,\nu)\dd \nu$. Note that due to the concavity of $\omega(\cdot,t)$, both terms $\omega(\xi+2\eta,t)+\omega (\xi-2\eta,t) -2\omega (\xi,t)$ and
$\omega (2\eta+\xi,t)-\omega (2\eta-\xi,t) -2\omega (\xi,t)$ are non-positive.
\\
(1) If $K$ satisfies \eqref{Kcond3} with $m$ satisfying (\textbf{A}1) and (\textbf{A}3), by using \eqref{mProp1}, we infer that for every $\eta>0$,
\begin{align}\label{eq:estim1}
  \widetilde{K}(\eta) & \geq c_1^{-1} \int_{\R^{d-1}} \frac{m\left((\eta^2 + |\nu|^2)^{-1/2}\right)}{(\eta^2 + |\nu|^2)^{d/2}}\dd \nu \nonumber \\
  & \geq c_1^{-1} \eta^\alpha m(\eta^{-1}) \int_{\R^{d-1}} \frac{1}{\left( \eta^2+|\nu|^2\right)^{(d+\alpha)/2}}\dd \nu \nonumber \\
  & \geq c_1^{-1} \frac{m(\eta^{-1})}{\eta} \int_{\R^{d-1}} \frac{1}{\left( 1+ |\nu'|^2\right)^{(d+\alpha)/2}} \dd\nu'
  \geq C_1  \frac{m(\eta^{-1})}{\eta},
\end{align}
where in the last inequality we used
\begin{equation*}
  c_1^{-1}\int_{\R^{d-1}}\frac{1}{(1+|\nu'|^2)^{(d+\alpha)/2}}\dd\nu'\geq c_1^{-1}\int_{|\nu'|\leq 1}\frac{1}{2^{(d+\alpha)/2}}\dd \nu'\geq c_1^{-1}\frac{1}{2^{(d+1)/2}}|B_1(0)|=C_1.
\end{equation*}
Inserting the above estimate into \eqref{D3} leads to \eqref{D2}.
\\
(2) If $K$ satisfies \eqref{Kcond1}-\eqref{Kcond2} with $m$ satisfying (\textbf{A}1)-(\textbf{A}2) and $\xi\leq c_0/2$ is concerned,
we mainly consider the scope $\eta \in ]0,\frac{c_0}{2}]$ and $|\nu|\in ]0, \frac{c_0}{2}]$ so that $(\eta^2+|\nu|^2)^{1/2}\in ]0,c_0]$,
thus similarly as \eqref{eq:estim1}, we get that for all $\eta\in ]0,\frac{c_0}{2}]$,
\begin{equation*}
\begin{split}
  \widetilde{K}(\eta) & \geq c_1^{-1} \int_{\nu\in\R^{d-1},|\nu|\leq \frac{c_0}{2}} \frac{m\left( (\eta^2 + |\nu|^2)^{-1/2} \right)}{(\eta^2 + |\nu|^2)^{d/2}}\dd\nu\\
  & \geq c_1^{-1} \frac{m(\eta^{-1})}{\eta} \int_{\nu'\in\R^{d-1},|\nu' |\leq 1} \frac{1}{(1+|\nu'|^2)^{\frac{d+\alpha}{2}}}\dd \nu' \geq C_1 \frac{m(\eta^{-1})}{\eta},
\end{split}
\end{equation*}
which ensures \eqref{Dxi-est}.
\\
(3) If $K$ satisfies \eqref{Kcond1}, \eqref{Kcond5} with $m$ satisfying (\textbf{A}1)-(\textbf{A}2), and $\xi\leq \frac{c_0}{2}$ is concerned, we divide the $(\eta,\nu)$ integral region of the R.H.S. of \eqref{D3}
into several parts $\set{\eta\in [\frac{c_0}{2},\infty[}$, $\set{\eta\in ]0,\frac{c_0}{2}], |\nu|\in ]0,\frac{c_0}{2}]}$
and $ \set{\eta\in]0,\frac{c_0}{2}],|\nu|\in [\frac{c_0}{2},\infty[}$. The part $\eta\in ]0,\frac{c_0}{2}]$ and $|\nu|\in ]0,\frac{c_0}{2}]$ can be treated as above and the bound is the R.H.S. of \eqref{Dxi-est}.
For $\eta\geq \frac{c_0}{2}$, the kernel $K(\eta,\nu)$ may be non-positive, and from \eqref{Kcond5} we deduce
\begin{equation*}
\begin{split}
  - \widetilde{K}(\eta) & \leq - \int_{(\eta^2+|\nu|^2)^{1/2}\leq c_0} K(\eta,\nu)\,\dd \nu - \int_{(\eta^2+|\nu|^2)^{1/2}\geq c_0} K(\eta,\nu)\,\dd\nu \\
  & \leq c_1 \int_{\R^{d-1}} \frac{1}{(\eta^2 + |\nu|^2)^{\frac{d+ \tilde\alpha}{2}}} \dd \nu
  \leq c_1 \frac{1}{\eta^{1+\tilde\alpha}} \int_{\R^{d-1}}\frac{1}{(1+|\nu'|^2)^{\frac{d+\tilde\alpha}{2}}}\dd \nu' \leq c_1C_d \frac{1}{\eta^{1+\tilde\alpha}} ,
\end{split}
\end{equation*}
and thus the contribution from this part is
\begin{equation*}
\begin{split}
  \int_{\frac{c_0}{2}}^\infty \left(2\omega(\xi,t) +\omega(2\eta-\xi,t)- \omega(2\eta+\xi,t)\right) \left(-\widetilde{K}(\eta)\right) \dd\eta
  \leq\; c_1 C_d 2 \omega(\xi,t)  \int_{\frac{c_0}{2}}^\infty \frac{1}{\eta^{1+\tilde\alpha}} \dd \eta \leq \frac{C'}{2} \omega(\xi,t).
\end{split}
\end{equation*}
For the part $\eta\in ]0,\frac{c_0}{2}]$ and $|\nu|\geq\frac{c_0}{2}$, from \eqref{Kcond5} we get
\begin{equation*}
\begin{split}
  -\int_{\nu\in\R^{d-1},|\nu|\geq \frac{c_0}{2}}K(\eta,\nu)\dd\nu & \leq   -\int_{\nu\in\R^{d-1},|\nu|\geq \frac{c_0}{2}, (\eta^2+|\nu|^2)^{1/2}\geq c_0}K(\eta,\nu)\dd\nu \\
  & \leq c_1 \int_{\nu\in\R^{d-1},|\nu|\geq \frac{c_0}{2}}\frac{1}{(\eta^2 +|\nu|^2)^{\frac{d+\tilde\alpha}{2}}} \dd \nu \leq C_{d,\tilde\alpha} c_1 c_0^{\tilde\alpha},
\end{split}
\end{equation*}
and thus the contribution from this part is bounded by
\begin{align*}
  & c_1 c_0^{\tilde\alpha} C_{d,\tilde\alpha}
  \bigg( \int_0^{\frac{\xi}{2}} \Big( 2\omega(\xi,t) - \omega(\xi+2\eta,t)-\omega(\xi-2\eta,t)\Big) + \int_{\frac{\xi}{2}}^{\frac{c_0}{2}} \Big( 2\omega(\xi,t) + \omega(2\eta-\xi,t)-\omega(2\eta+\xi,t)\Big)\bigg) \\
   &\leq \, c_1 c_0^{\tilde\alpha} C_{d,\tilde\alpha} \Big( \omega(\xi,t) \frac{\xi}{2} + 2 \omega(\xi,t)\frac{c_0-\xi}{2}\Big) \leq \frac{C_1'}{2} \omega(\xi,t).
\end{align*}
Hence, gathering the above estimates yields \eqref{Dxi-est2}.
\end{proof}

Next we consider the estimation of \eqref{GC-Om} under the scenario \eqref{scena}.
\begin{lemma}\label{lem-mocdrf}
  Assume that $u= \mathcal{P}(\theta)$ is defined by \eqref{uexp},
and the diffusion operator $\LL$ is given by \eqref{Lexp} with the radial kernel $K$,
then we have the following estimates on $\Omega(\xi,t)$ under the scenario \eqref{scena}.
\begin{enumerate}[(1)]
\item
If $K$ satisfies \eqref{Kcond3} with $m$ satisfying (\textbf{A}1) and (\textbf{A}3), then for all $\xi>0$,
\begin{equation}\label{Ome-es1}
  \Omega(\xi,t) \leq -\frac{C_2 }{m(\xi^{-1})} D(\xi,t) + C_2 \omega(\xi,t) + C_2 \xi \int_\xi^\infty \frac{\omega(\eta,t)}{\eta^2} \dd \eta,
\end{equation}
with $C_2>0$ defending only on $d$ (and $|a|,|\Psi|$).
\item
If $K$ satisfies \eqref{Kcond1}-\eqref{Kcond2} with $m$ satisfying (\textbf{A}1)-(\textbf{A}2), then we also get \eqref{Ome-es1} for all $0<\xi\leq \frac{c_0}{2}$.
\item
If $K$ satisfies \eqref{Kcond1} and \eqref{Kcond5} with $m$ satisfying (\textbf{A}1)-(\textbf{A}2), then for all $0<\xi \leq \frac{c_0}{2}$,
\begin{equation}\label{Ome-es1.0}
  \Omega(\xi,t) \leq - \frac{C_2}{m(\xi^{-1})} D(\xi,t) + \left(  C_2' + C_2 \right) \omega(\xi,t) +
  C_2 \xi \int_\xi^\infty \frac{\omega(\eta,t)}{\eta^2}\dd \eta,
\end{equation}
with some $C_2'>0$ depending on $d$, $\alpha$, $\tilde{\alpha}$, and $c_0,c_1$.
\item
There exists a constant $C_3>0$ depending only on $d,|a|,|\Psi|$ such that
\begin{equation}\label{Ome-es2}
  \Omega(\xi,t) \leq C_3 \omega(\xi,t) + C_3 \int_0^\xi \frac{\omega(\eta,t)}{\eta}\dd\eta
  + C_3 \xi \int_\xi^\infty \frac{\omega(\eta,t)}{\eta^2} \dd \eta.
\end{equation}
\end{enumerate}

\end{lemma}

Notice that for $\LL = |D|^\alpha$ and $u = H(\theta)$ with $H$ the 1D Hilbert transform, an estimate similar to \eqref{Ome-es1} was obtained in \cite[Lemma 2.7]{Do}.

\begin{proof}[Proof of Lemma \ref{lem-mocdrf}]
  For simplicity, we suppress the time variable $t$ in $\omega(\xi,t)$, $\Omega(\xi,t)$ and $D(\xi,t)$.
By virtue of \eqref{uexp}, we see that
\begin{equation*}
\begin{split}
  |u(x)-u(x+\xi e)| &  = \left| a\, \omega(\xi) + \mathrm{p.v.}\,\int_{\R^d} \frac{\Psi(\hat y)}{|y|^d} \theta(x+ y) \dd y - \mathrm{p.v.}\,\int_{\R^d} \frac{\Psi(\hat y)}{|y|^d} \theta(x+\xi e+ y) \dd y  \right| \\
  & \leq |a|\, \omega(\xi) + |I(\xi)| + |II(\xi)|,
\end{split}
\end{equation*}
with $\hat y=\frac{y}{|y|}\in \mathbb{S}^{d-1}$, and
\begin{equation}\label{Ixi}
  I(\xi) := \mathrm{p.v.}\, \int_{|y|\leq 2\xi} \frac{\Psi(\hat y)}{|y|^d}\theta(x+y) \dd y - \mathrm{p.v.} \,\int_{|y|\leq 2\xi} \frac{\Psi(\hat y)}{|y|^d} \theta(x+\xi e+ y) \dd y ,
\end{equation}
\begin{equation*}
 \textrm{and}\;\, II(\xi) := \int_{|y|\geq 2\xi} \frac{\Psi(\hat y)}{|y|^d}\theta(x+y) \dd y - \int_{|y|\geq 2\xi} \frac{\Psi(\hat y)}{|y|^d} \theta(x+\xi e+ y) \dd y .
\end{equation*}
First we note that the estimation of $II(\xi)$ and the proof of \eqref{Ome-es2} are classical, and one can refer to \cite[Lemma]{KNV} or \cite[Lemma 3.2]{MX-mQG} to see that
\begin{equation*}
  |II(\xi)| \leq C \xi \int_\xi^\infty \frac{\omega(\eta)}{\eta^2} \dd \eta, \quad \textrm{and}\quad
  |I(\xi)|\leq C \int_0^\xi \frac{\omega(\eta)}{\eta}\dd\eta.
\end{equation*}
Thus for the statements (1)-(3), it suffices to estimate $I(\xi)$ by virtue of $D(\xi)$. Thanks to the zero-average property of $\Psi(\hat y)$ and the scenario \eqref{scena}, we have
\begin{equation*}
\begin{split}
  I(\xi) & =\, \int_{|y|\leq 2\xi} \frac{\Psi(\hat y)}{|y|^d}(\theta(x+y)-\theta(x)) \dd y
  -\,\int_{|y|\leq 2\xi} \frac{\Psi(\hat y)}{|y|^d} (\theta(x+\xi e+ y)-\theta(x+\xi e)) \dd y  \\
  & = \, \int_{|y|\leq 2\xi} \frac{\Psi(\hat y)}{|y|^d}\big(\theta(x+y)-\theta(x+\xi e +y) -\omega(\xi)\big) \dd y ,
\end{split}
\end{equation*}
where the integral will be understood in the sense of principle value if needed.
\\
(1) If $K$ satisfies \eqref{Kcond3} with $m$ satisfying (\textbf{A}1) and (\textbf{A}3), recalling that $D(\xi)$ defined by \eqref{GC-D} has the formula \eqref{Dexp},
and using \eqref{mProp1}-\eqref{mProp2},
we obtain that for some constant $B >0$ chosen later,
\begin{align*}
  I(\xi) + \frac{ B}{m(\xi^{-1})} D(\xi) \leq \, & \, \int_{|y|\leq 2\xi} \Big( \frac{\Psi(\hat y)}{|y|^d} - c_1^{-1} \frac{B}{m(\xi^{-1})} \frac{m(|y|^{-1})}{|y|^d}\Big)
  \big(\omega(\xi) +\theta(x+\xi e +y)-\theta(x+y)\big) \dd y \\
  &  \, -\int_{|y|\geq 2\xi} K(y) \big(\omega(\xi) +\theta(x+\xi e +y)-\theta(x+y)\big)\dd y \\
  \leq \, & \, \int_{|y|\leq 2\xi} \Big( \frac{\Psi(\hat y)}{|y|^d} -  2^{-\sigma} c_1^{-1}B \frac{ \xi^{\alpha-\sigma}}{|y|^{d+\alpha-\sigma}}\Big)
  \big(\omega(\xi) +\theta(x+\xi e +y)-\theta(x+y)\big) \dd y \\
  \leq \, & \, \int_{|y|\leq 2\xi} \big( 2^{\alpha-\sigma} \Psi(\hat y) -2^{-\sigma} c_1^{-1}B \big)\frac{ \xi^{\alpha-\sigma}}{|y|^{d+\alpha-\sigma}}
  \big(\omega(\xi) +\theta(x+\xi e +y)-\theta(x+y)\big) \dd y,
\end{align*}
where in the third line we used $|y|^{\alpha-\sigma} m(|y|^{-1}) \geq (2\xi)^{\alpha-\sigma} m((2\xi)^{-1})\geq 2^{-\sigma}\xi^{\alpha-\sigma} m(\xi^{-1})$ for all $0<|y|\leq 2\xi$.
Thus by choosing $B=2 c_1\left( \max_{\hat y\in \mathbb{S}^{d-1}}|\Psi(\hat y)|\right)$, we get
\begin{equation}\label{I-es1}
  |I(\xi)| \leq -\frac{B}{m(\xi^{-1})} D(\xi).
\end{equation}
\\
(2) If $K$ satisfies \eqref{Kcond1}-\eqref{Kcond2} with $m$ satisfying (\textbf{A}1)-(\textbf{A}2), and we only consider $\xi$ in the range $0<\xi \leq c_0 /2$,
then due to that $K\geq 0$ on all $\R^d\setminus\{0\}$, we similarly obtain \eqref{I-es1}.
\\
(3) If $K$ satisfies \eqref{Kcond1} and \eqref{Kcond5} with $m$ satisfying (\textbf{A}1)-(\textbf{A}2), then for the same $B$ as above and for all $0<\xi \leq c_0/2$,
\begin{equation*}
\begin{split}
  I(\xi) + \frac{B}{m(\xi^{-1})} D(\xi) & \leq \frac{B}{m(\xi^{-1})} \int_{|y|\geq 2\xi}\left( \omega(\xi)+\theta(x+\xi e +y) -\theta(x+y)\right) \left(-K(y) \right) \dd y \\
  & \leq \frac{c_1 B}{m(2/c_0)} \int_{|y|\geq c_0}\left( \omega(\xi)+\theta(x+\xi e+y) -\theta(x+y)\right) \frac{1}{|y|^{d+\tilde\alpha}}\dd y .
\end{split}
\end{equation*}
By arguing as obtaining \eqref{Dxi-est2}, we find that
\begin{equation*}
  I(\xi) + \frac{B}{m(\xi^{-1})}D(\xi) \leq C_2' \omega(\xi).
\end{equation*}

Therefore, collecting the above estimates leads to the desired results \eqref{Ome-es1}-\eqref{Ome-es2}.
\end{proof}

\section{Proof of Proposition \ref{prop-evRe}: uniform-in-$\epsilon$ eventual H\"older estimate of the $\epsilon$-regularized solution}\label{sec-evRe}

\subsection{Sketch of the main proof}

In this section, we denote $\theta^\epsilon\in C([0,\infty[; H^s(\R^d))$, $s>1+\frac{d}{2}$
to be a smooth solution for the regularized drift-diffusion equation
\begin{equation}\label{reDDeq2}
  \partial_t \theta^\epsilon + u^\epsilon\cdot\nabla \theta^\epsilon + \LL \theta^\epsilon - \epsilon \Delta \theta^\epsilon = 0,\quad u^\epsilon = \mathcal{P}(\theta^\epsilon),\quad  \theta^\epsilon|_{t=0} = \theta^\epsilon_0= \phi_\epsilon*\big(\theta_01_{1/\epsilon}\big),
\end{equation}
where $\epsilon>0$, $\theta_0\in L^\infty$, $1_{B_{1/\epsilon}}$ is the indicator function on the ball $B_{1/\epsilon}$,
$\phi\in C^\infty_c(\R^d)$ is a test function satisfying $\int_{\R^d}\phi=1$,
and $\phi_\epsilon(x)=\epsilon^{-d}\phi(\epsilon^{-1}x)$. From the maximus principle \eqref{theMP1}, we immediately get the uniform $L^\infty$-estimate $\|\theta^\epsilon(t)\|_{L^\infty(\R^d)}\leq \|\theta^\epsilon_0\|_{L^\infty(\R^d)}\leq \|\theta_0\|_{L^\infty(\R^d)}$ for all $t\geq 0$.

Our main method is the nonlocal maximum principle. We first introduce a nonnegative function that for $\alpha\in ]0,1]$, $\sigma\in [0,\alpha[$ and $\beta\in ]1-\alpha+\sigma,1[$,
\begin{equation}\label{MOC2}
\omega(\xi)=
\begin{cases}
  \kappa\, m(\delta^{-1}) \delta^{1-\beta} \xi^\beta, & \quad \textrm{for}\;\; 0\leq \xi \leq \delta, \\
  \kappa\, m(\delta^{-1})\delta + \gamma \int_\delta^\xi m(\eta^{-1})\dd\eta, & \quad \textrm{for} \;\; \xi > \delta,
\end{cases}
\end{equation}
with $\delta>0$, $0<\gamma<\kappa<1$ chosen later.

We show that $\omega(\xi)$ is indeed a modulus of continuity (MOC) satisfying the needing properties. Clearly, $\omega(0+)=0$, $\omega'(0+)= \kappa \beta m(\delta^{-1}) \delta^{1-\beta} \lim_{\xi\rightarrow 0+} \xi^{\beta-1}=\infty$,
which satisfies the condition (3) in Proposition \ref{prop-GC}.
Observe that for every $0<\xi <\delta$,
\begin{equation}\label{omeg''}
  \omega'(\xi)= \kappa \beta m(\delta^{-1}) \delta^{1-\beta} \xi^{\beta-1}>0,\quad\textrm{and}\quad \omega''(\xi) = -\kappa \beta (1-\beta) m(\delta^{-1}) \delta^{1-\beta} \xi^{\beta-2}<0,
\end{equation}
and for every $\xi >\delta$ (from \eqref{mDec}),
\begin{equation}\label{omeg-fact1}
  \omega'(\xi) = \gamma m(\xi^{-1})>0,\quad \textrm{and}\quad \omega''(\xi) = -\gamma \frac{m'(\xi^{-1})}{\xi^2}\leq -\gamma (\alpha-\sigma) \frac{m(\xi^{-1})}{\xi}<0,
\end{equation}
and for $\xi =\delta$,
\begin{equation*}
  \omega'(\delta-)= \kappa \beta m(\delta^{-1}),\quad \textrm{and} \quad \omega'(\delta+)=\gamma m(\delta^{-1}),
\end{equation*}
thus if $\gamma< \kappa\beta$, we infer that $\omega$ is nondecreasing and concave for all $\xi>0$. We also find that
\begin{equation}\label{efact1}
  \textrm{the mapping $\xi\mapsto \frac{\omega(\xi)}{\xi^\beta}$ for every $\xi >0$ is non-increasing.}
\end{equation}
Indeed, if $\xi\in ]0,\delta]$, \eqref{efact1} is a direct consequence of \eqref{MOC2}; while if $\xi \in ]\delta,\infty[$,
we have $\left( \frac{\omega(\xi)}{\xi^\beta}\right)'=\frac{\xi \omega'(\xi) -\beta \omega(\xi)}{\xi^{\beta+1}} $, and noticing that by \eqref{omeg-fact1}, $\beta>1-\alpha+\sigma$ and $\gamma<\beta\kappa$,
\begin{equation*}
\begin{split}
  \left( \xi \omega'(\xi)-\beta\omega(\xi)\right)' =\omega'(\xi) + \xi \omega''(\xi)-\beta\omega'(\xi)
  < \left(1-\beta- (\alpha-\sigma)\right)\gamma m(\xi^{-1}) m(\xi^{-1})<0,
\end{split}
\end{equation*}
and
\begin{equation*}
  \delta\omega'(\delta+)-\beta\omega(\delta)=\gamma m(\delta^{-1})\delta -\beta \kappa m(\delta^{-1})\delta <0,
\end{equation*}
we deduce that $\frac{d}{d\xi}(\frac{\omega(\xi)}{\xi^\beta})<0$, which implies \eqref{efact1} in the range $\xi\in ]\delta,\infty[$.

Now we give the following key lemma concerned with the uniform-in-$\epsilon$ preservation of the H\"older regularity by using the modulus of continuity \eqref{MOC2}.
\begin{lemma}\label{lem-HoldRC}
There exist two constants $\gamma$ and $\kappa$ (cf. \eqref{kap-gam-cd1}) independent of $\delta,\epsilon$ such that
\begin{equation}\label{eq:HoldRC}
  \textrm{if for some time $T_0\geq0$, $\theta^\epsilon(T_0)$ uniformly-in-$\epsilon$ strictly obeys MOC $\omega(\xi)$},
\end{equation}
then for any $t>T_0$, $\theta^\epsilon(t)$ also strictly obeys this MOC $\omega(\xi)$, which further implies the $\beta$-H\"older regularity of $\theta^\epsilon(t)$ for $t>T_0$.
\end{lemma}

The proof of Lemma \ref{lem-HoldRC} is placed in Subsection \ref{subsec:lem}. We only note that under the uniform-in-$\epsilon$ preservation of MOC $\omega(\xi)$ by $\theta^\epsilon(t)$ for all $t\in [T_0,\infty[$,
we deduce from \eqref{efact1} that
\begin{equation}\label{Cbetaest}
  \sup_{t\in [T_0,\infty[}\|\theta^\epsilon(t)\|_{\dot C^\beta}= \sup_{t\in [T_0,\infty[} \sup_{x\neq y\in\R^d}\frac{|\theta^\epsilon(x,t)-\theta^\epsilon(y,t)|}{|x-y|^\beta}
  \leq \sup_{x\neq y\in\R^d}\frac{\omega(|x-y|)}{|x-y|^\beta} \leq \kappa m(\delta^{-1})\delta^{1-\beta}.
\end{equation}

Next our goal is to justify the condition \eqref{eq:HoldRC} at some time $T_0>0$.
We consider the following family of moduli of continuity that for $\xi_0 >\delta$,
\begin{equation}\label{MOC3.1}
  \omega(\xi,\xi_0)=
  \begin{cases}
    (1-\beta)\kappa m(\delta^{-1})\delta + \gamma\int_\delta^{\xi_0}m(\eta^{-1})\dd \eta -\gamma m(\xi_0^{-1})(\xi_0-\delta)+ \beta\kappa m(\delta^{-1})\xi, & \; \textrm{for}\; 0 \leq \xi \leq \delta, \\
    \kappa m(\delta^{-1})\delta + \gamma\int_\delta^{\xi_0}m(\eta^{-1})\dd \eta -\gamma m(\xi_0^{-1})\xi_0 + \gamma m(\xi_0^{-1})\xi,&\; \textrm{for}\;\delta<\xi\leq \xi_0, \\
    \kappa m(\delta^{-1})\delta + \gamma \int_\delta^\xi m(\eta^{-1}) \dd \eta, &\; \textrm{for}\; \xi>\xi_0,
  \end{cases}
\end{equation}
and for $\xi_0 \leq \delta$,
\begin{equation}\label{MOC3.2}
  \omega(\xi,\xi_0)=
  \begin{cases}
    (1-\beta)\kappa m(\delta^{-1})\delta^{1-\beta}\xi_0^\beta + \beta \kappa m(\delta^{-1})\delta^{1-\beta}\xi_0^{\beta-1}\xi,
    &\quad \textrm{for}\; 0\leq \xi< \xi_0, \\
    \kappa m(\delta^{-1}) \delta^{1-\beta} \xi^\beta, &\quad \textrm{for}\; \xi_0\leq \xi\leq \delta, \\
    \kappa m(\delta^{-1}) \delta + \gamma \int_\delta^\xi m(\eta^{-1})\,\dd \eta, & \quad \textrm{for}\; \xi>\delta,
  \end{cases}
\end{equation}
where $\beta\in ]1-\alpha+\sigma,1[$, and $\kappa,\gamma,\delta$ are positive constants chosen later. 
Note that for $\xi_0=0+$, $\omega(\xi,0+)$ just reduces to the MOC $\omega(\xi)$ given by \eqref{MOC2}.
Motivated by \cite{Kis}, the basic idea of constructing $\omega(\xi,\xi_0)$ is through taking a tangent line at $\xi=\xi_0$ to $\omega(\xi)$ given by \eqref{MOC2}
and replacing $\omega(\xi)$ with this tangent line at the range $0\leq \xi \leq\xi_0$.
But since the one-sided derivatives of $\omega(\xi)$ at the point $\xi=\delta$ do not coincide, thus in order to control $\partial_{\xi_0}\omega(\xi,\xi_0)$ at the point $\xi_0=\delta$,
we make a modification in the case $\xi_0 >\delta$,
that is, the tangent line mentioned above at the range $\delta\leq\xi\leq \xi_0$ is still adopted,
but at the range $0\leq \xi\leq \delta$ it is replaced by a straight line crossing $\omega(\delta+,\xi_0)$ with the larger slope $\omega'(\delta-)=\beta \kappa m(\delta^{-1})$.

Clearly, for all $\xi_0>0$, $\omega(0+,\xi_0)>0$, which guarantees the condition (3) in Proposition \ref{prop-GC}.
Similarly as $\omega(\xi)$ defined by \eqref{MOC2}, $\omega(\xi,\xi_0)$ is also a increasing and concave function for all $\xi>0$ and $\xi_0>0$.
For $\xi_0=A_0>\delta$, by virtue of \eqref{mProp2}, we get
\begin{align}\label{omeg0+}
  \omega(0+,A_0) & =(1-\beta)\kappa m(\delta^{-1})\delta + \gamma\int_\delta^{A_0}m(\eta^{-1})\dd \eta -\gamma m(A_0^{-1})(A_0-\delta) \nonumber \\
  & \geq (1-\beta)\kappa m(\delta^{-1})\delta + \gamma\, m(A_0^{-1})\, A_0^{\alpha-\sigma}\int_\delta^{A_0} \eta^{-(\alpha-\sigma)}\dd \eta
  -\gamma m(A_0^{-1})A_0 \nonumber \\
  & \geq (1-\beta)\kappa m(\delta^{-1})\delta + \frac{\gamma}{1-\alpha+\sigma}\, m(A_0^{-1})\, A_0^{\alpha-\sigma} \left( A_0^{1-\alpha+\sigma}-\delta^{1-\alpha-\sigma}\right)
  -\gamma m(A_0^{-1})A_0 \nonumber \\
  & \geq \big( (1-\beta)\kappa -\gamma\big) m(\delta^{-1})\delta + \frac{(\alpha-\sigma)\gamma}{1-\alpha+\sigma} m(A_0^{-1})A_0^{\alpha-\sigma}\left( A_0^{1-\alpha+\sigma}- \delta^{1-\alpha+\sigma}\right).
\end{align}
In view of $\gamma<(1-\beta)\kappa$, we have that the initial data $\theta_0^\epsilon$ uniformly-in-$\epsilon$ strictly obeys the MOC
$\omega(\xi,A_0)$ provided that
\begin{equation}\label{A0-del-cd}
  \frac{(\alpha-\sigma)\,\gamma}{1-\alpha+\sigma} m(A_0^{-1})A_0^{\alpha-\sigma} \left( A_0^{1-\alpha+\sigma} -\delta^{1-\alpha+\sigma} \right) \geq 2\|\theta_0\|_{L^\infty}.
\end{equation}

We next state the following crucial lemma.
\begin{lemma}\label{lem-mocev}
  Suppose that \textrm{Case (I)} is considered, and the initial data $\theta_0^\epsilon$ uniformly-in-$\epsilon$ strictly obeys the MOC $\omega(\xi,A_0)$ given by \eqref{MOC3.1}.
For $\rho >0$, let $\xi_0=\xi_0(t)$ be a function satisfying
\begin{equation}\label{xi0}
  \frac{d}{dt}\xi_0= - \rho \, m(\xi_0^{-1}) \xi_0, \quad \xi_0(0)= A_0.
\end{equation}
Then for some positive constants $\delta$, $\kappa$, $\gamma$, $\rho$ small enough, the solution $\theta^\epsilon(x,t)$ of the regularized drift-diffusion equation \eqref{reDDeq2}
strictly obeys the MOC $\omega(\xi,\xi_0(t))$ for all $t$ such that $\xi_0(t)> 0$.
\end{lemma}

Now with Lemma \ref{lem-mocev} at our disposal, whose proof is postponed in Subsection \ref{subsec:lem},
we can conclude Proposition \ref{prop-evRe} as follows. Thanks to \eqref{xi0}, and by integrating on the time variable over $[0,t]$,
we get
\begin{equation*}
\begin{split}
  \rho t = \int^{A_0}_{\xi_0(t)} \frac{1}{m(\xi_0^{-1})\xi_0}\dd\xi_0 & \leq \frac{1}{A_0^{\alpha-\sigma} m(A_0^{-1})} \int_{\xi_0(t)}^{A_0} \frac{1}{\xi_0^{1-\alpha+\sigma}}\dd\xi_0 \\
  & = \frac{1}{A_0^{\alpha-\sigma} m(A_0^{-1})}  \frac{1}{\alpha-\sigma}\left( A_0^{\alpha-\sigma}- \xi_0(t)^{\alpha-\sigma}\right),
\end{split}
\end{equation*}
which yields that
\begin{equation}\label{}
  \xi_0(t)\leq A_0 \left( 1 - m(A_0^{-1})(\alpha-\sigma)\rho\, t \right)^{\frac{1}{\alpha-\sigma}}.
\end{equation}
Thus there exists a time $t_1$ satisfying
\begin{equation}\label{t1-est}
  t_1\leq \frac{1}{(\alpha-\sigma) \rho\, m(A_0^{-1})},
\end{equation}
so that $\xi_0(t_1)\equiv 0$ and also
$\theta^\epsilon(x,t_1)$ obeys the MOC $\omega(\xi,0+) = \omega(\xi)$ with $\omega(\xi)$ given by \eqref{MOC2}.
Moreover, we claim that
\begin{equation}\label{eq:claim}
  \textrm{$\theta^\epsilon(x,t_1)$ uniformly-in-$\epsilon$ strictly obeys the MOC $\omega(\xi)$ with $\omega(\xi)$ given by \eqref{MOC2}}.
\end{equation}
Indeed, the proof is in the spirit of that of Proposition \ref{prop-GC}; denoting by $F^\epsilon(x,y,t_1)= \frac{|\theta^\epsilon(x,t_1)-\theta^\epsilon(y,t_1)|}{\omega(|x-y|)}$
for every $x\neq y\in \R^d$, and according to \cite[Proposition 3.1]{MX-gQG}, we find that there exist positive constants $C,c$ depending on $t_1$ so that $F^\epsilon(x,y,t_1)<1$
for every $(x,y)\notin \mathcal{K}$, with $\mathcal{K}:=\{(x,y)\in \R^d\times\R^d: |x|,|y|\leq C, |x-y|\geq c\}$, while for the continuous function $F^\epsilon(x,y,t_1)$
on the compact set $\mathcal{K}$, the maximum can be achieved by some pair $(x,y)$ on $\mathcal{K}$: if the maximum is just strictly less than $1$, the claim \eqref{eq:claim} is followed;
otherwise, there do exist some points $(x,y)\in \mathcal{K}$ so that $F^\epsilon(x,y,t_1)=1$, and by writing $y= x+\xi e$ with $\xi>0$ and $e\in \mathbb{S}^{d-1}$,
it yields that the scenario \eqref{scena} holds with $t=t_1$ and $\omega(\xi,t_1)=\omega(\xi)$;
but then from $\xi_0(t_1)=0$, \eqref{omegxi0-3} below and \eqref{GC-Om}-\eqref{GC-D},
\begin{equation}\label{eq:t1-targ}
  \frac{\dd }{\dd t}\left(\frac{\theta^\epsilon(x,t)-\theta^\epsilon(x+\xi e,t)}{\omega(\xi,\xi_0(t))}\right)\bigg|_{t=t_1}= \frac{\Omega(\xi,t_1) \omega'(\xi)+ D(\xi,t_1)+2\epsilon \omega''(\xi)}{\omega(\xi) },
\end{equation}
and by arguing as \eqref{Targ1} below we can show that the right-hand side of \eqref{eq:t1-targ} is strictly less than 0 (for $\kappa,\gamma$ in $\omega(\xi)$ satisfying \eqref{kap-gam-cd1}),
which clearly leads to a contradiction; hence we justify the assertion \eqref{eq:claim}.

Then the condition \eqref{eq:HoldRC} with $T_0=t_1$ is satisfied,
and Lemma \ref{lem-HoldRC} ensures that such a MOC $\omega(\xi)$ given by \eqref{MOC2} is strictly preserved by the solution
$\theta^\epsilon(x,t)$ for all $t\geq t_1$, which leads to
\begin{equation}\label{the-Ces}
  \sup_{t\in [t_1,\infty[}\|\theta^\epsilon(t)\|_{\dot C^\beta(\R^d)}\leq \kappa m(\delta^{-1})\delta^{1-\beta},
\end{equation}
with some fixed $\delta>0$ satisfying \eqref{A0-del-cd}, and thus we finish the proof of \eqref{Cbetaest0}.

In particular, if $\alpha\in ]0,1[$ and $\sigma=0$ in the condition \eqref{mDec2}, then \eqref{A0-del-cd}, \eqref{t1-est} and \eqref{the-Ces} reduce to
\begin{equation*}
\begin{cases}
  \frac{\alpha\,\gamma}{1-\alpha} \left( A_0^{1-\alpha} -\delta^{1-\alpha} \right) \geq 2\|\theta_0\|_{L^\infty},\quad  \\
   t_1\leq A_0^\alpha/(\alpha \rho),\quad \\
  \sup_{t\in [t_1,\infty[}\|\theta^\epsilon(t)\|_{\dot C^\beta(\R^d)}\leq \kappa \delta^{1-\alpha-\beta},
\end{cases}
\end{equation*}
where $\kappa,\gamma,\rho$ are fixed positive constants satisfying \eqref{rkg-cdsum} below, that is, we can choose
\begin{equation}\label{rkg-main}
  \rho= \frac{1-\beta}{C \alpha}, \quad \kappa=\frac{1}{C}(1-\beta)^2,\quad
  \gamma= \frac{1}{C}\min\set{(1-\beta)^3\alpha, \beta(1-\beta)^2},
\end{equation}
with some $C=C(d)>0$. By choosing
\begin{equation*}
  A_0= \left( \frac{4(1-\alpha)}{\alpha\gamma}\|\theta_0\|_{L^\infty}\right)^{\frac{1}{1-\alpha}},\quad
  \delta = \left( \frac{(1-\alpha)}{\alpha\gamma}\|\theta_0\|_{L^\infty}\right)^{\frac{1}{1-\alpha}},
\end{equation*}
we see that
\begin{equation}\label{t1-bdd2}
  t_1\leq \frac{C}{1-\beta} \left( \frac{4(1-\alpha)}{\alpha\gamma}\right)^{\frac{\alpha}{1-\alpha}} \|\theta_0\|_{L^\infty}^{\frac{\alpha}{1-\alpha}},
\end{equation}
and for every $\beta\in ]1-\alpha,1[$, we have
\begin{equation}\label{Cbet-bdd}
  \sup_{t\in [t_1,\infty[}\|\theta^\epsilon(t)\|_{\dot C^\beta(\R^d)}\leq \frac{(1-\beta)^2}{C}
  \left( \frac{1-\alpha}{\alpha\gamma}\right)^{-\frac{\beta-1+\alpha}{1-\alpha}} \|\theta_0\|_{L^\infty}^{-\frac{\beta-1+\alpha}{1-\alpha}},
\end{equation}
where $C>0$ is some constant depending only on $d$, and thus we conclude Proposition \ref{prop-evRe}.

\subsection{Proof of Lemma \ref{lem-HoldRC} and Lemma \ref{lem-mocev}}\label{subsec:lem}

\begin{proof}[Proof of Lemma \ref{lem-HoldRC}]

According to Proposition \ref{prop-GC}, it suffices to show that for all $t>T_0$ and $\xi>0$,
\begin{equation}\label{Targ1}
  \Omega(\xi,t)\omega'(\xi) + D(\xi,t) + \epsilon \omega''(\xi) < 0,
\end{equation}
where $\Omega(\xi,t)$, $D(\xi,t)$ are respectively defined by \eqref{GC-Om} and \eqref{GC-D}
under the scenario \eqref{scena} with $\omega(\cdot)$ in place of $\omega(\cdot,t)$.
By using Lemma \ref{lem-mocdiss} and Lemma \ref{lem-mocdrf}, we get
\begin{equation}\label{Dxit}
\begin{split}
  D(\xi,t) \leq & \, C_1 \int_0^{\frac{\xi}{2}}\left(\omega(\xi+2\eta)+\omega (\xi-2\eta) -2\omega (\xi)\right) \frac{m(\eta^{-1})}{\eta}\textrm{d} \eta \\
  & + C_1 \int_{\frac{\xi}{2}}^\infty \big(\omega (2\eta+\xi)-\omega (2\eta-\xi) -2\omega (\xi)\big) \frac{m(\eta^{-1})}{\eta}\textrm{d}\eta,
\end{split}
\end{equation}
and
\begin{equation}\label{Omega2}
  \Omega(\xi,t) \leq -\frac{C_2}{m(\xi^{-1})}  D(\xi,t) + C_2  \omega(\xi) + C_2 \xi \int_\xi^\infty \frac{\omega(\eta)}{\eta^2} \dd \eta,
\end{equation}
where $C_1=C_1(d), C_2=C_2(d)>0$.

In order to prove \eqref{Targ1}, we divide into two cases.

\textbf{Case 1: $0<\xi \leq \delta$.}

In this case, we have $\omega(\xi)= \kappa m(\delta^{-1}) \delta^{1-\beta} \xi^\beta$, and $\omega'(\xi)= \kappa \beta  m(\delta^{-1}) \delta^{1-\beta} \xi^{\beta-1}$,
and from \eqref{efact1} we see that
\begin{align*}
  \int_\xi^\infty \frac{\omega(\eta)}{\eta^2} \,\dd \eta  & = \int_\xi^\delta \frac{\omega(\eta)}{\eta^2} \dd\eta + \int_\delta^\infty \frac{\omega(\eta)}{\eta^2} \dd \eta = \kappa m(\delta^{-1}) \delta^{1-\beta} \int_\xi^\delta \eta^{\beta-2}\dd \eta +   \int_\delta^\infty \frac{\omega(\eta)}{\eta^\beta} \frac{1}{\eta^{2-\beta}} \,\dd\eta \\
  & \leq \kappa m(\delta^{-1}) \delta^{1-\beta} \frac{1}{1-\beta} (\xi^{\beta-1}-\delta^{\beta-1}) + \kappa m(\delta^{-1})\delta^{1-\beta} \frac{1}{1-\beta}\delta^{\beta-1} \\
  & \leq\frac{\kappa}{1-\beta} m(\delta^{-1}) \delta^{1-\beta} \xi^{\beta-1}.
\end{align*}
Thus
\begin{equation*}
  \Omega(\xi,t)\omega'(\xi)\leq - \frac{C_2}{m(\xi^{-1})}\omega'(\xi) D(\xi,t) + \frac{2 C_2}{1-\beta} \left( \kappa m(\delta^{-1})\delta^{1-\beta} \right)^2 \beta \xi^{2\beta-1}.
\end{equation*}
Observing that for every $\beta>1-\alpha+\sigma$ and $\xi\in ]0,\delta]$,
\begin{equation*}
\begin{split}
  \frac{C_2}{m(\xi^{-1})}\omega'(\xi)   =C_2\beta\kappa \frac{ m(\delta^{-1})\delta^{1-\beta} \xi^{\beta-1+\alpha-\sigma}}{\xi^{\alpha-\sigma} m(\xi^{-1})}
   \leq C_2\beta\kappa \frac{ m(\delta^{-1})\delta^{1-\beta} \delta^{\beta-1+\alpha-\sigma}}{\delta^{\alpha-\sigma} m(\delta^{-1})} = C_2 \beta \kappa,
\end{split}
\end{equation*}
we further get that by letting $\kappa < 1/(2 C_2\beta)$,
\begin{equation}\label{omeg-est000}
  \Omega(\xi,t)\omega'(\xi)\leq - \frac{1}{2} D(\xi,t) + \frac{2C_2}{1-\beta} \left( \kappa m(\delta^{-1})\delta^{1-\beta} \right)^2 \beta \xi^{2\beta-1}.
\end{equation}

For the contribution from the diffusion term, by virtue of the following estimate
\begin{equation}\label{omega-fact}
\begin{split}
  \omega(\xi+2\eta)+\omega(\xi-2\eta)-2\omega(\xi) & = 4\eta^2 \int_0^1 \int_{-1}^1 \lambda \omega''(\xi + 2\lambda \tau\,\eta )\,\dd \tau \dd\lambda \\
  & \leq 4\eta^2 \int_0^1 \int_{-1}^0 \lambda \omega''(\xi)\,\dd\tau \dd\lambda  \leq \omega''(\xi) \eta^2,
\end{split}
\end{equation}
and \eqref{omeg''}, \eqref{mProp2}, we directly get
\begin{equation}\label{Dxi-es1}
\begin{split}
  D(\xi,t) & \leq C_1 \int_0^{\frac{\xi}{2}}\left(\omega(\xi+2\eta)+\omega (\xi-2\eta) -2\omega (\xi)\right) \frac{m(\eta^{-1})}{\eta}\textrm{d} \eta \\
  & \leq  C_1 \omega''(\xi) \int_0^{\frac{\xi}{2}} \eta m(\eta^{-1})\dd\eta \\
  & \leq  - C_1\beta(1-\beta)\kappa m(\delta^{-1}) \delta^{1-\beta} \xi^{\beta-2}  \int_0^{\frac{\xi}{2}} \left( \eta^{\alpha-\sigma} m(\eta^{-1}) \right) \eta^{1-\alpha+\sigma} \dd\eta \\
  & \leq  - \frac{C_1}{8}  \beta (1-\beta)\kappa\left( m(\delta^{-1})\right)^2 \delta^{1-\beta+\alpha-\sigma} \xi^{\beta-\alpha +\sigma}.
\end{split}
\end{equation}
Hence we infer that
\begin{equation}\label{eqkey1}
\begin{split}
  & \Omega(\xi,t)\omega'(\xi) + D(\xi,t) \leq \\
  \leq\, & \frac{2 C_2 }{1-\beta}\beta \left( \kappa m(\delta^{-1})\delta^{1-\beta} \right)^2 \xi^{2\beta-1} + \frac{1}{2} D(\xi,t) \\
  \leq\, & \beta\kappa \left( m(\delta^{-1})\right)^2 \delta^{1-\beta+\alpha-\sigma} \xi^{\beta-\alpha+\sigma} \left( \frac{2 C_2 }{1-\beta} \kappa
  \left( \frac{\xi}{\delta}\right)^{\beta- 1+\alpha-\sigma} - \frac{C_1(1-\beta)}{16}\right) \\
  \leq\, & \beta\kappa \left( m(\delta^{-1})\right)^2 \delta^{1-\beta+\alpha-\sigma} \xi^{\beta-\alpha+\sigma} \left( \frac{2 C_2 }{1-\beta} \kappa - \frac{C_1(1-\beta)}{16}\right) <0,
\end{split}
\end{equation}
where the last inequality is from choosing $\kappa$ so that $\kappa < \frac{C_1}{32 C_2 }(1-\beta)^2$.

\textbf{Case 2: $\xi \geq \delta$.}

Taking advantage of \eqref{efact1}, we have
\begin{equation*}
\begin{split}
  \int_\xi^\infty \frac{\omega(\eta)}{\eta^2} \dd \eta & = \int_\xi^\infty \frac{\omega(\eta)}{\eta^\beta} \frac{1}{\eta^{2-\beta}} \dd \eta
  \leq \frac{\omega(\xi)}{\xi^\beta} \int_\xi^\infty \frac{1}{\eta^{2-\beta}}\dd \eta \leq \frac{1}{1-\beta} \frac{\omega(\xi)}{\xi} .
\end{split}
\end{equation*}
Thus from \eqref{Omega2} and $\omega'(\xi)=\gamma m(\xi^{-1})$ in this case, we obtain that by choosing $\gamma < 1/(2C_2)$,
\begin{equation}\label{eq:OmeEst}
  \Omega(\xi,t) \omega'(\xi) = -\gamma C_2 D(\xi,t) + \frac{2 C_2 }{1-\beta} \gamma\, m(\xi^{-1})\omega(\xi)
  \leq -\frac{1}{2} D(\xi,t) + \frac{2 C_2 }{1-\beta}\gamma\, m(\xi^{-1})\omega(\xi).
\end{equation}

For $D(\xi,t)$, noticing that $\omega(2\eta +\xi) - \omega(2\eta-\xi)\leq \omega(2\xi)<2\omega(\xi)$, we get
\begin{align}\label{eq:Dest}
  D(\xi,t) & \leq C_1 \big( \omega(2\xi) -2\omega(\xi)\big)\int_{\frac{\xi}{2}}^\infty \frac{m(\eta^{-1})}{\eta}\dd \eta \nonumber \\
  & \leq  C_1 \left( \omega(2\xi) -2\omega(\xi)\right) \left(\frac{\xi}{2}\right)^\alpha m\left(\frac{2}{\xi}\right)
  \int_{\frac{\xi}{2}}^\infty \frac{1}{\eta^{1+\alpha}} \dd \eta \nonumber \\
  & \leq C_1 \big( \omega(2\xi) -2\omega(\xi)\big) 2^{-\sigma} \xi^\alpha m(\xi^{-1})\frac{1}{\alpha} \left(\frac{ 2}{\xi} \right)^\alpha \nonumber \\
  & \leq  \frac{C_1 }{\alpha} \left( \omega(2\xi) -2\omega(\xi) \right) m(\xi^{-1}).
\end{align}
Next we claim that for $\gamma$ small enough, we have
\begin{equation}\label{claim1}
  \omega(2\xi) \leq \max\set{2^{1-\alpha+\sigma},3/2} \omega(\xi), \quad \forall \xi \geq \delta.
\end{equation}
Indeed, for $\xi=\delta$, we see that $\omega(\delta)= \kappa m(\delta^{-1})\delta$ and
\begin{equation*}
\begin{split}
  \omega(2\delta) & =\omega(\delta)+ \gamma \int_\delta^{2\delta} m(\eta^{-1})\dd\eta
  \leq \omega(\delta)  + \gamma \delta^{\alpha-\sigma} m(\delta^{-1})\int_\delta^{2\delta} \frac{1}{\eta^{\alpha-\sigma}} \dd \eta \\
  & \leq \kappa m(\delta^{-1})\delta + \gamma \delta^{\alpha-\sigma} m(\delta^{-1}) \frac{1}{1-\alpha+\sigma}
  \left( (2\delta)^{1-\alpha+\sigma} -\delta^{1-\alpha+\sigma}\right) \\
  & \leq \kappa m(\delta^{-1}) \delta + \frac{\gamma}{1-\alpha+\sigma} \left( 2^{1-\alpha+\sigma} -1\right) m(\delta^{-1}) \delta, 
\end{split}
\end{equation*}
which further yields that for all $ \gamma< \frac{\kappa}{2}$,
\begin{equation*}
\begin{split}
  \omega(2\delta) & \leq
  \begin{cases}
    \kappa m(\delta^{-1})\delta + 2\gamma (2^{1-\alpha+\sigma}-1) m(\delta^{-1})\delta, & \quad \textrm{if} \;\alpha-\sigma\leq 1/2, \\
    \kappa m(\delta^{-1})\delta + \gamma \left(\sup_{x\in ]0,1/2]}\frac{2^x-1}{x}\right)\, m(\delta^{-1})\delta, & \quad \textrm{if}\; \alpha-\sigma >1/2,
  \end{cases} \\
  & \leq \max \set{2^{1-\alpha+\sigma}, 3/2} \omega(\delta),
\end{split}
\end{equation*}
where we have used $\sup_{x\in ]0,1/2]}\frac{2^x-1}{x}\leq \max\set{\lim_{x\rightarrow0+}\frac{2^x-1}{x},\frac{2^{1/2}-1}{1/2}}\leq 1$.
Whereas for $\xi \in ]\delta,\infty[$, considering an auxiliary function
\begin{equation*}
  h(\xi) : = \omega(2\xi) - \max\{2^{1-\alpha+\sigma},3/2\} \omega(\xi),
\end{equation*}
and noting that
\begin{equation*}
\begin{split}
  h'(\xi) \leq 2 \omega'(2\xi) - 2^{1-\alpha+\sigma} \omega'(\xi)  = 2 m((2\xi)^{-1}) - 2^{1-\alpha+\sigma} m(\xi^{-1})\leq 0,
\end{split}
\end{equation*}
we deduce $h(\xi)\leq h(\delta) \leq 0$ for all $\xi \geq \delta$, which implies \eqref{claim1}. Hence plugging \eqref{claim1} into \eqref{eq:Dest} yields
\begin{equation}\label{eq:Dest2}
\begin{split}
  D(\xi,t) & \leq - \frac{C_1}{\alpha}  \left( 2- \max \set{2^{1-\alpha+\sigma},3/2}\right) m(\xi^{-1})\omega(\xi) \\
  & \leq - \frac{C_1}{2\alpha}  (1-2^{-\alpha+\sigma})m(\xi^{-1}) \omega(\xi)
  \leq - \frac{C_1 \tilde{c}}{4\alpha}  (\alpha-\sigma)  m(\xi^{-1})\omega(\xi),
\end{split}
\end{equation}
with $\tilde{c}:= \inf_{x\in ]0,1]}\set{\frac{2^x-1}{x}}>0$.

Collecting the above estimates yields that for all $\xi \geq \delta$,
\begin{equation}\label{eqkey2}
  \Omega(\xi,t)\omega'(\xi)+ D(\xi,t) \leq \left( \frac{2 C_2 }{1-\beta}\gamma - \frac{C_1 \tilde{c}(\alpha-\sigma)}{4\alpha}
   \right) m(\xi^{-1})\omega(\xi)<0,
\end{equation}
where the last inequality is ensured as long as $\gamma$ is satisfying $\gamma <\frac{C_1 \tilde{c} (1-\beta) (\alpha-\sigma) }{ 8 C_2 \alpha} $.

Therefore, thanks to \eqref{eqkey1} and \eqref{eqkey2}, we prove \eqref{Cbetaest} for every $\beta\in ]1-\alpha+\sigma,1[$ with each $\alpha\in ]0,1]$ and $\sigma\in[0,\alpha[$,
where $\delta>0$, and $\kappa$, $\gamma$ are some fixed positive constants satisfying
\begin{equation}\label{kap-gam-cd1}
  \kappa<\min\set{\frac{1}{2 C_2\beta},\frac{C_1 (1-\beta)^2}{32 C_2}},\quad
  \gamma < \min \set{\beta \kappa, \frac{\kappa}{2},\frac{1}{2C_2},\frac{C_1  \tilde{c} (1-\beta) (\alpha-\sigma) }{ 8 C_2 \alpha}}.
\end{equation}
Thus we finish the proof of Lemma \ref{lem-HoldRC}.
\end{proof}

Next we show Lemma \ref{lem-mocev}.
\begin{proof}[Proof of Lemma \ref{lem-mocev}]
According to Proposition \ref{prop-GC}, it suffices to prove that for all $t>0$ and $\xi >0$,
\begin{equation}\label{Targ3}
  - \partial_{\xi_0}\omega(\xi,\xi_0) \dot\xi_0(t)+ \Omega(\xi,t)\partial_\xi\omega(\xi,\xi_0) + D(\xi,t) + \epsilon \partial_{\xi\xi}\omega(\xi,\xi_0)<0,
\end{equation}
where $\omega(\xi,\xi_0)$ is given by \eqref{MOC3.1}-\eqref{MOC3.2} and
\begin{equation}\label{Dxit3}
\begin{split}
  D(\xi,t) \leq &
  \,C_1 \int_0^{\frac{\xi}{2}}\left(\omega(\xi+2\eta,\xi_0)+\omega (\xi-2\eta,\xi_0) -2\omega (\xi,\xi_0)\right) \frac{m(\eta^{-1})}{\eta}\textrm{d} \eta \\
  & +C_1 \int_{\frac{\xi}{2}}^\infty \big(\omega (2\eta+\xi,\xi_0)-\omega (2\eta-\xi,\xi_0) -2\omega (\xi,\xi_0)\big) \frac{m(\eta^{-1})}{\eta}\textrm{d}\eta,
\end{split}
\end{equation}
and
\begin{equation}\label{Omega3}
  \Omega(\xi,t) \leq -\frac{C_2}{m(\xi^{-1})}  D(\xi,t)+ C_2 \omega(\xi,\xi_0) + C_2 \xi \int_\xi^\infty \frac{\omega(\eta,\xi_0)}{\eta^2} \dd \eta.
\end{equation}
In \eqref{Targ3}, if $\partial_{\xi_0}\omega(\xi,\xi_0)$ or $\partial_\xi\omega(\xi,\xi_0)$ does not exist, the larger value of the one-sided derivative should be taken.

We divide into several cases to get \eqref{Targ3}, owing to the values of $\xi_0$ and $\xi$.

\textbf{Case 1: $\xi_0 > \delta$, $0 <\xi \leq \delta$.}

From $\omega(\xi,\xi_0)=(1-\beta)\kappa m(\delta^{-1})\delta+\gamma\int_\delta^{\xi_0}m(\eta^{-1})\dd \eta -\gamma m(\xi_0^{-1})(\xi_0 -\delta)+ \beta\kappa m(\delta^{-1})\xi$,
we have
\begin{equation}\label{omeges1}
\begin{split}
  \partial_{\xi_0}\omega(\xi,\xi_0)  = \gamma \xi_0^{-2} m'(\xi_0^{-1})\left( \xi_0 -\delta\right) \leq \gamma\alpha m(\xi_0^{-1}), \quad\textrm{and}\quad
  \partial_\xi \omega(\xi,\xi_0)  = \, \beta\kappa m(\delta^{-1}) ,
\end{split}
\end{equation}
and
\begin{align}\label{Mxi0}
  \omega(\xi,\xi_0)& \geq \omega(0+,\xi_0)
  = (1-\beta)\kappa m(\delta^{-1})\delta+ \gamma \int_\delta^{\xi_0}m(\eta^{-1})\dd\eta-\gamma m(\xi_0^{-1}) (\xi_0 -\delta) \nonumber \\
  & \geq (1-\beta)\kappa m(\delta^{-1})\delta +\gamma \xi_0^{\alpha-\sigma} m(\xi_0^{-1}) \int_\delta^{\xi_0} \frac{1}{\eta^{\alpha-\sigma}}\dd\eta - \gamma m(\xi_0^{-1}) (\xi_0-\delta) \nonumber \\
  & =  (1-\beta)\kappa m(\delta^{-1})\delta +\frac{\gamma}{1-\alpha+\sigma} m(\xi_0^{-1})\xi_0^{\alpha-\sigma}
  \left( \xi_0^{1-\alpha+\sigma} -\delta^{1-\alpha+\sigma} \right) -\gamma  m(\xi_0^{-1})\left(\xi_0 -\delta \right) \nonumber \\
  & =: M_{\xi_0,\delta},
\end{align}
and
\begin{equation}\label{omegdel}
  \omega(\xi,\xi_0)-\omega(0+,\xi_0)\leq \omega(\delta,\xi_0)-\omega(0+,\xi_0)=\beta\kappa m(\delta^{-1})\delta.
\end{equation}
Thus by using \eqref{xi0} and \eqref{omeges1}, we get
\begin{equation}\label{xi0-est0}
  -\partial_{\xi_0}\omega(\xi,\xi_0)\dot \xi_0(t)\leq \rho \alpha \gamma\left( m(\xi_0^{-1})\right)^2 \xi_0.
\end{equation}
In view of \eqref{MOC3.1}, we obtain
\begin{align}\label{omegint}
  & \int_\xi^\infty \frac{\omega(\eta,\xi_0)}{\eta^2} \dd\eta =\frac{\omega(\xi,\xi_0)}{\xi} + \int_\xi^\infty \frac{\partial_\eta\omega(\eta,\xi_0)}{\eta}\dd\eta \nonumber \\
  =\, & \frac{\omega(\xi,\xi_0)}{\xi} + \int_\xi^\delta \frac{\kappa \beta m(\delta^{-1})}{\eta} \dd\eta+\int_\delta^{\xi_0} \frac{\gamma m(\xi_0^{-1})}{\eta}\dd\eta + \int_{\xi_0}^\infty \frac{\gamma m(\eta^{-1})}{\eta} \dd\eta \nonumber \\
  \leq\, & \frac{\omega(\xi,\xi_0)}{\xi} + \kappa \beta m(\delta^{-1}) \left( \log \frac{\delta}{\xi}\right) + \gamma m(\xi_0^{-1})\left(\log\frac{\xi_0}{\delta}\right) + \gamma\xi_0^{\alpha-\sigma}m(\xi_0^{-1}) \int_{\xi_0}^\infty \frac{1}{\eta^{1-\alpha+\sigma}}\dd\eta \nonumber \\
  \leq\, & \frac{\omega(\xi,\xi_0)}{\xi} +\kappa \beta m(\delta^{-1}) \left( \log \frac{\delta}{\xi}\right)+\gamma m(\xi_0^{-1})\left(\log\frac{\xi_0}{\delta}\right) + \frac{\gamma}{\alpha-\sigma}m(\xi_0^{-1}).
\end{align}
Thus by using \eqref{Omega3}, \eqref{omegdel} and \eqref{omegint}, we find that for $\kappa \leq \frac{1}{4C_2\beta }$,
\begin{equation}\label{omeg-est0}
\begin{split}
  & \Omega(\xi,t)\partial_\xi \omega(\xi,\xi_0) \\
  \leq & - C_2\beta\kappa  D(\xi,t) + C_2\beta\kappa m(\delta^{-1}) \left( 2\omega(\xi,\xi_0)
  +\kappa \beta m(\delta^{-1}) \xi\left( \log \frac{\delta}{\xi}\right)
  +\gamma m(\xi_0^{-1}) \xi\left( \log \frac{\xi_0}{\delta} + \frac{1}{\alpha-\sigma} \right) \right)\\
  \leq & - C_2\beta\kappa  D(\xi,t)
  +  C_2 \beta \kappa m(\delta^{-1})\left( 2\omega(0+,\xi_0) + (C_0+2)\kappa\beta m(\delta^{-1})\delta +  \gamma  m(\xi_0^{-1}) \xi_0 \left( C_0 + \frac{1}{\alpha-\sigma} \right)\right) \, \\
  \leq & -\frac{1}{4}  D(\xi,t)
  + C_2\left( C_0  +2\right) \beta^2 \kappa^2 \left( m(\delta^{-1}) \right)^2 \delta
  + \frac{C_2\beta( C_0 + 1) \kappa\gamma }{\alpha-\sigma} m(\delta^{-1})  m(\xi_0^{-1}) \xi_0  \, + \\
  & \,+2C_2 \beta \kappa  m(\delta^{-1}) \omega(0+,\xi_0) ,
\end{split}
\end{equation}
where in the third line we also used $\frac{\xi}{\delta}\left( \log\frac{\delta}{\xi}\right)\leq C_0$ and $\frac{\xi}{\xi_0}\log \frac{\xi_0}{\delta}\leq C_0$.
For the contribution from the diffusion term, since the function $\omega(\eta,\xi_0)-\omega(0+,\xi_0)$ is still concave, we infer that
\begin{align}\label{D2xit00}
  D(\xi,t)
  \leq & -2 C_1 \omega(0+,\xi_0)\int_{\frac{\xi}{2}}^\infty \frac{m(\eta^{-1})}{\eta}\dd \eta \nonumber \\
  \leq & -2 C_1 \omega(0+,\xi_0) \left( \frac{\xi}{2}\right)^\alpha
  m\left(\frac{2}{\xi}\right) \int_{\frac{\xi}{2}}^\infty \frac{1}{\eta^{1+\alpha}} \dd\eta \nonumber \\
  \leq & -\frac{2C_1}{\alpha}  \omega(0+,\xi_0) m(\xi^{-1}),
\end{align}
and also by \eqref{Mxi0},
\begin{equation}\label{D2xit}
\begin{split}
  D(\xi,t) \leq  -\frac{2C_1}{\alpha} M_{\xi_0,\delta} m(\xi^{-1}).
\end{split}
\end{equation}
If $\xi_0 \geq N \delta$ with $N\in\N$ a suitable constant, we see that
\begin{equation*}
  \frac{1}{1-\alpha+\sigma}\left( \xi_0^{1-\alpha+\sigma} -\delta^{1-\alpha+\sigma}\right)\geq \frac{1-(1/N)^{1-\alpha+\sigma}}{1-\alpha+\sigma}\xi_0^{1-\alpha+\sigma}
  \geq \frac{1}{1-\frac{\alpha-\sigma}{2}}\xi_0^{1-\alpha+\sigma},
\end{equation*}
provided that $1-(1/N)^{1-\alpha+\sigma}\geq \frac{2(1-\alpha+\sigma)}{2-\alpha+\sigma}$, that is,
$N\geq \left(\frac{2-(\alpha-\sigma)}{\alpha-\sigma}\right)^{\frac{1}{1-(\alpha-\sigma)}}$, thus we may choose
\begin{equation}\label{Ncond}
  N := \Big[\Big(\frac{2-\alpha+\sigma)}{\alpha-\sigma}\Big)^{\frac{1}{1-\alpha+\sigma}}\Big]+1.
\end{equation}
Thus for the case $\xi_0\geq N\delta$, we get
\begin{equation}\label{Mxi0-est1}
\begin{split}
  M_{\xi_0,\delta} & \geq (1-\beta)\kappa m(\delta^{-1})\delta + \left(\frac{1}{1-(\alpha-\sigma)/2} -1 \right) \gamma\, m(\xi_0^{-1})\xi_0 \\
  & \geq (1-\beta)\kappa m(\delta^{-1})\delta + \gamma\, (\alpha-\sigma) m(\xi_0)\xi_0.
\end{split}
\end{equation}
Inserting the above estimate into \eqref{D2xit} leads to
\begin{equation}\label{D2xit-es}
\begin{split}
  D(\xi,t)\leq  -\frac{2C_1 (1-\beta) \kappa}{\alpha} m(\delta^{-1})\delta \,m(\xi^{-1}) - \frac{2C_1(\alpha-\sigma)\,\gamma}{\alpha} m(\xi_0^{-1})\xi_0 m(\xi^{-1}).
\end{split}
\end{equation}
Hence for $\xi_0\geq N \delta$ with $N$ satisfying \eqref{Ncond}, by \eqref{D2xit00} and setting $\kappa\leq \frac{C_1}{4C_2\beta\alpha}$ so that
\begin{equation}\label{diff-cont}
  2C_2\beta \kappa m(\delta^{-1}) \omega(0+,\xi_0)\leq \frac{C_1}{2\alpha} m(\xi^{-1})\omega(0+,\xi_0)\leq -\frac{1}{4}D(\xi,t),
\end{equation}
and by collecting \eqref{xi0-est0}, \eqref{omeg-est0} and \eqref{D2xit-es}, we deduce that
\begin{align*}
  \textrm{L.H.S. of \eqref{Targ3}}\;\leq & \,
  \kappa m(\delta^{-1})\delta \,m(\xi^{-1}) \Big( C_2\left( C_0 +2\right) \beta^2 \kappa -\frac{C_1 (1-\beta) }{\alpha} \Big) + \\
  & + \gamma m(\xi_0^{-1}) \xi_0 m(\xi^{-1})\Big( \rho \alpha
  +  \frac{C_2\beta( C_0 + 1)}{\alpha-\sigma}\kappa - \frac{C_1(\alpha-\sigma)}{\alpha}\Big)  < 0,
\end{align*}
where L.H.S. denotes left-hand side and the last inequality is guaranteed as long as $\rho$, $\kappa$ satisfy
\begin{equation}\label{rkg-cd1}
  \rho < \frac{C_1 (\alpha-\sigma)}{2\alpha^2}, \quad
  \kappa< \min \set{ \frac{1}{4C_2\beta}, \frac{C_1}{4C_2\beta\alpha},\frac{C_1 (\alpha-\sigma)^2}{2 C_2\left( C_0 + 1\right) \beta\alpha},
  \frac{C_1 (1-\beta)\,}{C_2(C_0+2)\beta^2\alpha}}.
\end{equation}
If $\xi_0\leq N\delta$ with $N$ satisfying \eqref{Ncond}, thanks to the fact
\begin{equation}\label{mfact2}
  m(\xi_0^{-1})\xi_0\leq m\left((N\delta)^{-1}\right) N \delta \leq N^{1-\alpha+\sigma} m(\delta^{-1})\delta\leq \frac{4}{\alpha-\sigma} m(\delta^{-1})\delta,
\end{equation}
and using \eqref{diff-cont} again, the positive contribution which is treated by \eqref{xi0-est0} and \eqref{omeg-est0} can further be bounded by
\begin{equation*}
\begin{split}
  &\,-\partial_{\xi_0}\omega(\xi,\xi_0)\dot\xi_0 + \Omega(\xi,t)\partial_\xi \omega(\xi,\xi_0) \\
  \leq\, & \, - \frac{1}{2} D(\xi,t) + \kappa\left( m(\delta^{-1})\right)^2\delta\,
  \bigg( \frac{4\rho \alpha}{\alpha-\sigma}\frac{\gamma}{\kappa} + C_2\left( C_0  +2\right) \beta^2 \kappa+\frac{4C_2(C_0 +1)\beta\, }{(\alpha-\sigma)^2}\gamma \bigg).
\end{split}
\end{equation*}
For the negative contribution from the diffusion term, from \eqref{Mxi0}, \eqref{D2xit} and \eqref{mfact2}, we directly get that by letting $\gamma\leq \frac{(1-\beta)(\alpha-\sigma)}{8}\kappa$,
\begin{equation}\label{D2xit-es2}
\begin{split}
  D(\xi,t) & \leq -\frac{2C_1}{\alpha}m(\xi^{-1})\Big( (1-\beta)\kappa\, m(\delta^{-1})\delta -\gamma  m(\xi_0^{-1})\xi_0 \Big) \\
  & \leq  -\frac{2C_1}{\alpha}\Big((1-\beta)\kappa- \frac{4\gamma}{\alpha-\sigma} \Big)\,\left( m(\delta^{-1})^2\delta  \right)
  \leq -\frac{C_1 \left( 1-\beta\right) \kappa}{\alpha} \left( m(\delta^{-1})^2 \delta\right).
\end{split}
\end{equation}
Hence for $\xi_0\leq N\delta$, we have
\begin{equation*}
\begin{split}
  \textrm{L.H.S. of \eqref{Targ3}}\,\leq &\, \kappa\left( m(\delta^{-1})\right)^2\delta\,
  \left( \frac{4\alpha}{\alpha-\sigma} \rho + C_2\left( C_0 +2\right) \beta^2 \kappa+\frac{4C_2(C_0 +1)\beta\, }{(\alpha-\sigma)^2}\gamma - \frac{C_1(1-\beta)}{2\alpha}\right) < 0,
\end{split}
\end{equation*}
where the last inequality is ensured if we set
\begin{equation}\label{rkg-cd2}
\begin{split}
  &\rho < \frac{C_1 (1-\beta)(\alpha-\sigma) }{24\alpha^2},\quad  \kappa< \min\set{\frac{C_1 (1-\beta)}{6C_2\left( C_0 +2\right) \beta^2\alpha},\frac{1}{4C_2\beta}}, \\
  & \gamma \leq \min\set{\frac{(1-\beta)(\alpha-\sigma)}{8}\kappa,\frac{C_1(1-\beta)(\alpha-\sigma)^2}{24C_2(C_0+1)\beta\alpha}}.
\end{split}
\end{equation}

\textbf{Case 2: $\xi_0>\delta$, $\delta<\xi\leq \xi_0$.}

From $\omega(\xi,\xi_0)= \kappa m(\delta^{-1})\delta + \gamma\int_\delta^{\xi_0}m(\eta^{-1})\dd \eta -\gamma m(\xi_0^{-1})\xi_0 + \gamma m(\xi_0^{-1})\xi$ in this case,
we have
\begin{equation*}
\begin{split}
  \partial_{\xi_0}\omega(\xi,\xi_0) =  \gamma m'(\xi_0^{-1})\xi_0^{-2} \left( \xi_0-\xi\right) \leq \alpha \gamma m(\xi_0^{-1}),\quad \textrm{and}\quad
  \partial_\xi \omega(\xi,\xi_0)  = \gamma m(\xi_0^{-1}),
\end{split}
\end{equation*}
and (recalling $M_{\xi_0,\delta}$ is defined in \eqref{Mxi0})
\begin{equation*}
\begin{split}
  \omega(\xi,\xi_0)\geq \omega(\delta,\xi_0)
  & \geq \kappa m(\delta^{-1})\delta+ \gamma \xi_0^{\alpha-\sigma}m(\xi_0^{-1}) \frac{1}{1-\alpha+\sigma} \left(\xi_0^{1-\alpha+\sigma} - \delta^{1-\alpha+\sigma} \right)
  -\gamma m(\xi_0^{-1})(\xi_0-\delta) \\
  & \geq \kappa m(\delta^{-1})\delta + \frac{\gamma}{1-\alpha+\sigma} m(\xi_0^{-1})\xi_0^{\alpha-\sigma} (\xi_0^{1-\alpha+\sigma}-\delta^{1-\alpha+\sigma}) -\gamma m(\xi_0^{-1})(\xi_0-\delta) \\
  & = M_{\xi_0,\delta} + \beta\kappa m(\delta^{-1})\delta,
\end{split}
\end{equation*}
and
\begin{equation}\label{omeg-xi0}
  \omega(\xi,\xi_0)-\omega(0+,\xi_0)\leq \omega(\xi_0,\xi_0)-\omega(0+,\xi_0)=\gamma m(\xi_0^{-1})(\xi_0-\delta) + \beta \kappa m(\delta^{-1})\delta.
\end{equation}
Thus by using \eqref{xi0}, we get
\begin{equation}\label{omeg-est1}
  -\partial_{\xi_0}\omega(\xi,\xi_0) \dot \xi_0(t) \leq \alpha \rho \gamma \left( m(\xi_0^{-1}) \right)^2 \xi_0.
\end{equation}
From the following estimate
\begin{equation*}
\begin{split}
  \int_\xi^\infty \frac{\omega(\eta,\xi_0)}{\eta^2} \dd\eta 
  & = \frac{\omega(\xi,\xi_0)}{\xi} + \int_\xi^{\xi_0} \frac{\gamma m(\xi_0^{-1})}{\eta}\dd\eta + \int_{\xi_0}^\infty \frac{\gamma m(\eta^{-1})}{\eta} \dd\eta \\
  & \leq \frac{\omega(\xi,\xi_0)}{\xi} + \gamma m(\xi_0^{-1})\left(\log\frac{\xi_0}{\xi}\right) + \gamma\xi_0^{\alpha-\sigma}m(\xi_0^{-1}) \int_{\xi_0}^\infty \frac{1}{\eta^{1-\alpha+\sigma}}\dd\eta \\
  & =  \frac{\omega(\xi,\xi_0)}{\xi} + \gamma m(\xi_0^{-1})\left(\log\frac{\xi_0}{\xi}\right) + \frac{\gamma}{\alpha-\sigma}m(\xi_0^{-1}),
\end{split}
\end{equation*}
and similarly as obtaining \eqref{omeg-est0}, we find that for $\gamma\leq \frac{1}{4C_2}$,
\begin{align}\label{omeg-est2}
  & \Omega(\xi,t)\partial_\xi \omega(\xi,\xi_0) \nonumber \\
  \leq & - C_2\gamma D(\xi,t) + 2 C_2 \gamma\omega(\xi,\xi_0) m(\xi_0^{-1})
  + C_2\left(\gamma m(\xi_0^{-1})\right)^2 \left( \xi \log \frac{\xi_0}{\xi} + \frac{\xi}{\alpha-\sigma} \right) \\
  \leq & -\frac{1}{4} D(\xi,t)
  + \frac{C_2 (C_0+3)}{\alpha-\sigma} \left( \gamma m(\xi_0^{-1})\right)^2 \xi_0 \, + 2C_2 \beta \gamma \kappa m(\delta^{-1})\delta\, m(\xi_0^{-1})
  + 2C_2\gamma m(\xi_0^{-1}) \omega(0+,\xi_0), \nonumber
\end{align}
where $C_0>0$ is the constant such that $\frac{\xi}{\xi_0}\log \frac{\xi_0}{\delta}\leq C_0$.

For the contribution from the diffusion term, we also have \eqref{D2xit00} and \eqref{D2xit}.

If $\xi_0\geq N\delta$ with $N\in\N$ defined by \eqref{Ncond},
by using \eqref{Mxi0-est1} and setting $\gamma<\frac{C_1}{4C_2\alpha}$, we deduce that
\begin{equation*}
\begin{split}
  \textrm{L.H.S. of \eqref{Targ3}}\;\leq
  & \;\kappa m(\delta^{-1})\delta \,m(\xi^{-1}) \left(2 C_2\beta \gamma -\frac{C_1(1-\beta) }{\alpha} \right) + \\
  & + \gamma m(\xi_0^{-1}) \xi_0 m(\xi^{-1})\left( \rho \alpha
  +  \frac{C_2( C_0 + 3)}{\alpha-\sigma}\gamma
  - \frac{C_1(\alpha-\sigma)}{\alpha}\right)  < 0,
\end{split}
\end{equation*}
where the last inequality is guaranteed as long as
\begin{equation}\label{rkg-cd3}
  \rho<\frac{C_1(\alpha-\sigma)}{2\alpha^2},\quad
  \gamma<\min\set{\frac{1}{4C_2},\frac{C_1}{4C_2\alpha},\frac{C_1 (1-\beta)}{2C_2 \beta \alpha},
  \frac{C_1(\alpha-\sigma)^2}{2C_2(C_0+3)\alpha}}.
\end{equation}
If $\xi_0\leq N\delta$ with $N$ satisfying \eqref{Ncond}, by applying \eqref{mfact2} and setting $\gamma<\frac{C_1}{4C_2\alpha}$,
the positive contribution treated by \eqref{omeg-est1} and \eqref{omeg-est2} can further be bounded as
\begin{equation*}
\begin{split}
  &\,-\partial_{\xi_0}\omega(\xi,\xi_0)\dot\xi_0(t) + \Omega(\xi,t)\partial_\xi \omega(\xi,\xi_0) \\
  \leq\, & \, - \frac{1}{2} D(\xi,t) + \frac{4\gamma}{\alpha-\sigma}  m(\delta^{-1}) \delta\,m(\xi_0^{-1})
  \Big( \rho \alpha + \frac{C_2(C_0 +3)\, }{\alpha-\sigma}\gamma \Big).
\end{split}
\end{equation*}
For the negative contribution from the diffusion term, by arguing as \eqref{D2xit-es2} we obtain that for $\gamma \leq \frac{(1-\beta)(\alpha-\sigma)}{8}\kappa$,
\begin{equation*}
\begin{split}
  D(\xi,t) & \leq -\frac{2C_1}{\alpha}m(\xi^{-1})\Big( (1-\beta)\kappa\, m(\delta^{-1})\delta -\frac{4\gamma}{\alpha-\sigma}  m(\delta^{-1})\delta \Big) \\
  & \leq  -\frac{2C_1}{\alpha} \frac{4\,\gamma}{\alpha-\sigma}m(\xi^{-1}) m(\delta^{-1})\delta .
\end{split}
\end{equation*}
Hence for $\xi_0\leq N\delta$ with $N$ given by \eqref{Ncond}, we have
\begin{equation*}
\begin{split}
  \textrm{L.H.S. of \eqref{Targ3}}\,\leq \, \frac{4\gamma}{\alpha-\sigma}m(\delta^{-1}) \delta\,m(\xi^{-1})
  \Big(  \rho \alpha  +\frac{C_2(C_0 +3)\, }{\alpha-\sigma}\gamma - \frac{C_1}{\alpha}\Big)  < 0,
\end{split}
\end{equation*}
where the last inequality is guaranteed if we set
\begin{equation}\label{rkg-cd4}
\begin{split}
  \rho < \frac{C_1 (\alpha-\sigma) }{2\alpha^2},\quad\gamma \leq \min\set{\frac{(1-\beta)(\alpha-\sigma)}{8}\kappa,\frac{1}{4C_2},\frac{C_1(\alpha-\sigma)}{2C_2(C_0+3)\alpha}}.
\end{split}
\end{equation}

\textbf{Case 3: $\xi_0>\delta$, $\xi> \xi_0$.}

In this case, from $\omega(\xi,\xi_0) = \kappa m(\delta^{-1})\delta + \gamma \int_\delta^\xi m(\eta^{-1})\dd \eta$, we see that $\partial_{\xi_0}\omega(\xi,\xi_0)=0$, $\partial_\xi\omega(\xi,\xi_0)=\gamma m(\xi^{-1})$, and
\begin{align*}
  & \int_\xi^\infty \frac{\omega(\eta,\xi_0)}{\eta^2}\dd\eta = \frac{\omega(\xi,\xi_0)}{\xi} + \int_\xi^\infty \frac{\gamma m(\eta^{-1})}{\eta}\dd \eta \\
  & \leq\,\frac{\omega(\xi,\xi_0)}{\xi} + \gamma \xi^{\alpha-\sigma} m(\xi^{-1})\int_\xi^\infty \frac{1}{\eta^{1+\alpha-\sigma}}\dd \eta
  \leq \frac{\omega(\xi,\xi_0)}{\xi} + \frac{\gamma}{\alpha-\sigma} m(\xi^{-1}).
\end{align*}
Thus in light of \eqref{Omega3}, we get
\begin{equation}\label{Oxe-est1}
  \Omega(\xi,t)\partial_\xi\omega(\xi,\xi_0) \leq - C_2 \gamma D(\xi,t) + C_2\Big( 2 \omega(\xi,\xi_0) + \frac{\gamma}{\alpha-\sigma}\xi m(\xi^{-1})\Big) \gamma m(\xi^{-1}).
\end{equation}

For the contribution from the diffusion term, since $\omega(2\eta+\xi,\xi_0)-\omega(2\eta-\xi,\xi_0)\leq \omega(2\xi,\xi_0)< 2\omega(\xi,\xi_0)$, by estimating as \eqref{D2xit00} we obtain
\begin{equation}\label{Dxe-est1}
  D(x,t) \leq C_1 \left( \omega(2\xi,\xi_0) - 2\omega(\xi,\xi_0) \right) \int_{\frac{\xi}{2}}^\infty\frac{m(\eta^{-1})}{\eta}\dd\eta
  \leq \frac{C_1}{\alpha} \left( \omega(2\xi,\xi_0)-2\omega(\xi,\xi_0)\right) m(\xi^{-1}).
\end{equation}
Observing that
\begin{equation*}
\begin{split}
  \omega(2\xi,\xi_0)-\omega(\xi,\xi_0)  = \gamma \int_\xi^{2\xi} m(\eta^{-1})\dd\eta
  \leq \gamma \xi^{\alpha-\sigma} m(\xi^{-1})\int_\xi^{2\xi} \frac{1}{\eta^{\alpha-\sigma}}\dd\eta
  \leq \frac{2^{1-\alpha+\sigma}-1}{1-\alpha+\sigma} \gamma\, m(\xi^{-1})\xi,
\end{split}
\end{equation*}
and
\begin{equation*}
\begin{split}
  \omega(\xi,\xi_0) \geq \gamma \int_\delta^\xi m(\eta^{-1})\dd\eta   \geq  \gamma \xi^{\alpha-\sigma} m(\xi^{-1})\int_\delta^\xi \frac{1}{\eta^{\alpha-\sigma}}\dd\eta
  \geq \gamma \xi^{\alpha-\sigma} m(\xi^{-1})\frac{\xi^{1-\alpha+\sigma} -\delta^{1-\alpha+\sigma}}{1-\alpha+\sigma},
\end{split}
\end{equation*}
thus if $\xi$ satisfies that $\xi \geq \delta \left( \frac{1}{2^{\alpha-\sigma}-1}\right)^{\frac{1}{1-\alpha+\sigma}}$, equivalently,
$\xi^{1-\alpha+\sigma} -\delta^{1-\alpha+\sigma}\geq \left( 2- 2^{\alpha-\sigma}\right) \xi^{1-\alpha+\sigma}$, we find
\begin{equation}\label{}
  \omega(\xi,\xi_0)\geq \frac{2-2^{\alpha-\sigma}}{1-\alpha+\sigma} \gamma m(\xi^{-1})\xi = 2^{\alpha-\sigma}  \frac{2^{1-\alpha+\sigma}-1}{1-\alpha+\sigma} \gamma m(\xi^{-1})\xi \geq\tilde{c} \gamma\, m(\xi^{-1})\xi,
\end{equation}
and
\begin{equation*}
  \omega(2\xi,\xi_0)-\omega(\xi,\xi_0)\leq \frac{2^{1-\alpha+\sigma} -1}{2-2^{\alpha-\sigma}} \omega(\xi,\xi_0) = 2^{-\alpha+\sigma} \omega(\xi,\xi_0),
\end{equation*}
and
\begin{equation}\label{om-fact1}
  \omega(2\xi,\xi_0) -2\omega(\xi,\xi_0) \leq -\left(1-2^{-\alpha+\sigma}\right)\omega(\xi,\xi_0)\leq - \frac{\tilde{c}(\alpha-\sigma)}{2}\omega(\xi,\xi_0),
\end{equation}
with $\tilde{c}:= \inf_{x\in ]0,1]}\set{\frac{2^x-1}{x}}>0$. Hence if $\xi\geq \delta \left( \frac{1}{2^{\alpha-\sigma}-1}\right)^{\frac{1}{1-\alpha+\sigma}}$,
and by gathering the above estimates and setting $\gamma\leq \frac{1}{2C_2}$, we deduce that
\begin{equation*}
  \Omega(\xi,t)\partial_\xi\omega(\xi,\xi_0)\leq -\frac{1}{2} D(\xi,t) + \frac{3C_2}{\tilde{c} (\alpha-\sigma)}\gamma\,\omega(\xi,\xi_0) m(\xi^{-1}),
\end{equation*}
and
\begin{equation*}
  D(\xi,t) \leq - \frac{C_1\tilde{c}(\alpha-\sigma)}{2\alpha} \omega(\xi,\xi_0) m(\xi^{-1}),
\end{equation*}
and
\begin{equation*}
  \Omega(\xi,t) \partial_\xi\omega(\xi,\xi_0)+ D(\xi,t)\leq \Big( \frac{3C_2}{\tilde{c} (\alpha-\sigma)}\gamma -\frac{C_1 \tilde{c}(\alpha-\sigma)}{2\alpha}\Big) \omega(\xi,\xi_0)m(\xi^{-1})<0,
\end{equation*}
where the last inequality is ensured if we set
\begin{equation}\label{rkg-cd4.1}
  \gamma < \min \set{\frac{1}{2C_2},\frac{C_1 \tilde{c}^2 (\alpha-\sigma)^2}{6 C_2 \alpha}}.
\end{equation}

On the other hand, if $\xi$ satisfies that $\xi \leq \delta \left( \frac{1}{2^{\alpha-\sigma}-1}\right)^{\frac{1}{1-\alpha+\sigma}}$,
since $\omega(\xi,\xi_0)-\omega(0+,\xi_0)$ is concave and $\omega(0+,\xi_0)\geq (1-\beta)\kappa m(\delta^{-1})\delta$,
we get
\begin{equation}\label{Dxe-est3}
  D(\xi,t)\leq - 2\omega(0+,\xi_0) \int_{\frac{\xi}{2}}^\infty \frac{m(\eta^{-1})}{\eta}\dd\eta\leq - \frac{2(1-\beta) \kappa}{\alpha}\delta m(\delta^{-1})m(\xi^{-1}),
\end{equation}
and by using $\xi^{1-\alpha+\sigma}-\delta^{1-\alpha+\sigma}\leq \delta^{1-\alpha+\sigma}\frac{ 2- 2^{\alpha-\sigma}}{2^{\alpha-\sigma}-1}$, we also infer that
\begin{equation*}
  m(\xi^{-1})\xi\leq \delta^{\alpha-\sigma}m(\delta^{-1}) \xi^{1-\alpha+\sigma}\leq m(\delta^{-1})\delta \frac{1}{2^{\alpha-\sigma}-1} \leq \frac{1}{\tilde{c}(\alpha-\sigma)} m(\delta^{-1})\delta,
\end{equation*}
and
\begin{equation}\label{om-est1}
\begin{split}
  \omega(\xi,\xi_0) 
  & \leq \kappa m(\delta^{-1})\delta + \gamma \delta^{\alpha-\sigma} m(\delta^{-1}) \int_\delta^\xi \frac{1}{\eta^{\alpha-\sigma}}\dd\eta \\
  & \leq \kappa m(\delta^{-1})\delta + \frac{\gamma}{1-\alpha+\sigma} \delta^{\alpha-\sigma}m(\delta^{-1}) \left( \xi^{1-\alpha+\sigma}-\delta^{1-\alpha+\sigma}\right) \\
  & \leq \left(\kappa + \frac{2^{\alpha-\sigma}(2^{1-\alpha+\sigma}-1)}{1-\alpha+\sigma} \frac{1}{2^{\alpha-\sigma}-1} \gamma  \right) m(\delta^{-1}) \delta
   \leq \left( \kappa + \frac{2\gamma}{\tilde{c}(\alpha-\sigma)} \right) m(\delta^{-1})\delta,
\end{split}
\end{equation}
where $\tilde{c}:= \inf_{x\in ]0,1]}\set{\frac{2^x-1}{x}}>0$ and we also used $\sup_{x\in ]0,1]}\frac{2^x-1}{x}\leq 1$. Hence if $\xi \leq \delta \left(\frac{1}{2^{\alpha-\sigma}-1}\right)^{\frac{1}{1-\alpha+\sigma}}$,
by collecting the above results and letting $\gamma \leq \min\set{\frac{1}{2C_2},\kappa}$, we obtain
\begin{equation*}
  \Omega(\xi,t)\partial_\xi\omega(\xi,\xi_0)\leq -\frac{1}{2}D(\xi,t) +
  \left( 2\gamma + \frac{4\gamma}{\tilde{c} (\alpha-\sigma)} + \frac{2\gamma}{\tilde{c} (\alpha-\sigma)^2}\right)\kappa\delta m(\delta^{-1}) m(\xi^{-1}),
\end{equation*}
and thus
\begin{equation*}
  \Omega(\xi,t)\partial_\xi\omega(\xi,\xi_0) + D(\xi,t) \leq \left(\frac{8\gamma}{\tilde{c} (\alpha-\sigma)^2} - \frac{1-\beta}{\alpha} \right) \kappa \delta m(\delta^{-1})m(\xi^{-1}),
\end{equation*}
where the last inequality is ensured by setting
\begin{equation}\label{rkg-cd4.2}
  \gamma <\min\set{\frac{\tilde{c}(1-\beta) (\alpha-\sigma)^2}{8 \alpha},\frac{1}{2C_2},\kappa}.
\end{equation}

\textbf{Case 4: $0<\xi_0\leq \delta$, $0<\xi<\xi_0$.}

In this case $\omega(\xi,\xi_0)=(1-\beta)\kappa m(\delta^{-1})\delta^{1-\beta}\xi_0^\beta + \beta \kappa m(\delta^{-1})\delta^{1-\beta}\xi_0^{\beta-1}\xi$, and thus
\begin{equation*}
\begin{split}
  \partial_{\xi_0}\omega(\xi,\xi_0) =\beta(1-\beta)\kappa m(\delta^{-1})\left( \frac{\delta}{\xi_0}\right)^{1-\beta} \left(1 -\frac{\xi}{\xi_0} \right), \quad \textrm{and}\quad
  \partial_\xi\omega(\xi,\xi_0)  = \beta \kappa m(\delta^{-1})\left( \frac{\delta}{\xi_0}\right)^{1-\beta},
\end{split}
\end{equation*}
and
\begin{equation}\label{om-est2}
\begin{split}
  \omega(\xi,\xi_0) & \geq \omega(0+,\xi_0)\geq (1-\beta)\kappa m(\delta^{-1})\delta^{1-\beta}\xi_0^\beta, \\
  \omega(\xi,\xi_0) & \leq \omega(\delta,\xi_0)\leq \kappa m(\delta^{-1})\delta^{1-\beta}\xi_0^\beta.
\end{split}
\end{equation}
Taking advantage of the following estimates
\begin{equation}\label{mfact}
  m(\delta^{-1})\leq \left(\frac{\xi}{\delta}\right)^{\alpha-\sigma} m(\xi^{-1}),\quad\textrm{and}\quad m(\xi_0^{-1})\leq\left(\frac{\xi}{\xi_0}\right)^{\alpha-\sigma} m(\xi^{-1}),
\end{equation}
we deduce
\begin{align}\label{omegxi0-3}
  -\partial_{\xi_0}\omega(\xi,\xi_0)\dot\xi_0(t) & \leq \rho \beta(1-\beta)\kappa m(\delta^{-1})\left( \frac{\delta}{\xi_0}\right)^{1-\beta} \left(\xi_0 -\xi \right) m(\xi_0^{-1}) \nonumber \\
  & \leq \rho \beta(1-\beta)\kappa m(\delta^{-1})\left( \frac{\delta}{\xi_0}\right)^{1-\beta} \xi_0 \left( \frac{\xi}{\xi_0}\right)^{\alpha-\sigma} m(\xi^{-1}) \nonumber \\
  & \leq \rho \beta(1-\beta)\kappa m(\delta^{-1}) \xi_0 \left( \frac{\delta}{\xi_0}\right)^{1+\alpha-\sigma-\beta} \left( \frac{\xi}{\delta}\right)^{\alpha-\sigma} m(\xi^{-1}) \nonumber \\
  & \leq \rho \beta(1-\beta) \kappa m(\delta^{-1}) \xi_0 m(\xi^{-1}),
\end{align}
and
\begin{equation*}
  -\frac{C_2}{m(\xi^{-1})}D(\xi,t) \partial_\xi\omega(\xi,\xi_0)\leq -C_2\beta\kappa \left( \frac{\delta}{\xi_0}\right)^{1-\beta-\alpha+\sigma} D(\xi,t)
  \leq -C_2\beta\kappa D(\xi,t).
\end{equation*}
In view of the integration by parts and \eqref{MOC3.2}, we see that
\begin{equation*}
\begin{split}
  \int_\xi^\infty \frac{\omega(\eta,\xi_0)}{\eta^2}\dd\eta & = \frac{\omega(\xi,\xi_0)}{\xi} + \int_\xi^\infty \frac{\partial_\eta\omega(\eta,\xi_0)}{\eta}\dd\eta \\
  & \leq \frac{\omega(\xi,\xi_0)}{\xi} + \int_\xi^{\xi_0}\frac{\beta \kappa m(\delta^{-1}) \delta^{1-\beta}\xi_0^{\beta-1}}{\eta}\dd\eta
  + \int_{\xi_0}^\infty \frac{\beta \kappa m(\delta^{-1})\delta^{1-\beta}\eta^{\beta-1}}{\eta}\dd\eta \\
  & \leq \frac{\omega(\xi,\xi_0)}{\xi} + \beta\kappa m(\delta^{-1})\left(\frac{\delta}{\xi_0}\right)^{1-\beta}\left( \log\frac{\xi_0}{\xi}\right) +
  \frac{\beta}{1-\beta}\kappa m(\delta^{-1})\left( \frac{\delta}{\xi_0}\right)^{1-\beta},
\end{split}
\end{equation*}
then gathering the above estimates and \eqref{Omega3} leads to that for $\kappa \leq \frac{1}{2C_2\beta}$,
\begin{equation}\label{Omeg-est3}
\begin{split}
  \Omega(\xi,t)\partial_\xi\omega(\xi,\xi_0) \leq & -C_2\beta\kappa D(\xi,t) + 2C_2 \beta\kappa m(\delta^{-1}) \left( \frac{\delta}{\xi_0}\right)^{1-\beta}\omega(\xi,\xi_0) \\
  & \, + C_2 \left( \beta\kappa m(\delta^{-1})\right)^2 \left(\frac{\delta}{\xi_0}\right)^{2(1-\beta)} \xi_0 \left( \frac{\xi}{\xi_0}\log\frac{\xi_0}{\xi} + \frac{\xi}{\xi_0}\frac{1}{1-\beta}\right) \\
  \leq & -\frac{1}{2} D(\xi,t) +C_2 \beta \left(\kappa m(\delta^{-1})\right)^2 \left(\frac{\delta}{\xi_0}\right)^{2(1-\beta)} \xi_0 \left( 2+ C_0\beta + \frac{\beta}{1-\beta} \right),
\end{split}
\end{equation}
where in the third line we used $\frac{\xi}{\xi_0}\leq 1$ and $\frac{\xi}{\xi_0}\left( \log\frac{\xi_0}{\xi}\right)\leq C_0$.
For the contribution from the diffusion term, by arguing as \eqref{D2xit}, we obtain
\begin{equation}\label{D2xit3}
\begin{split}
  D(\xi,t)  \leq - 2 C_1\omega(0+,\xi_0)\int_{\frac{\xi}{2}}^\infty \frac{m(\eta^{-1})}{\eta}\dd\eta
  \leq -\frac{2(1-\beta)C_1}{\alpha}\kappa m(\delta^{-1}) \left(\frac{\delta}{\xi_0}\right)^{1-\beta}\xi_0 m(\xi^{-1}).
\end{split}
\end{equation}
Collecting the estimates \eqref{omegxi0-3}, \eqref{Omeg-est3} and \eqref{D2xit3}, and using \eqref{mfact} again, we find that
\begin{equation*}
\begin{split}
  &\textrm{L.H.S. of \eqref{Targ3}}\,\\
  \leq\, & \kappa m(\delta^{-1}) \left(\frac{\delta}{\xi_0}\right)^{1-\beta}\xi_0 m(\xi^{-1})
  \bigg( \rho \beta(1-\beta) + \frac{C_2 \beta (2+C_0\beta)\kappa}{1-\beta} \left(\frac{\xi}{\xi_0}\right)^{\alpha-\sigma} \left( \frac{\delta}{\xi_0}\right)^{1-\beta-\alpha+\sigma}
  -\frac{C_1(1-\beta)}{\alpha}\bigg) \\
  \leq \, & \kappa m(\delta^{-1}) \left(\frac{\delta}{\xi_0}\right)^{1-\beta}\xi_0 m(\xi^{-1})
  \left( \rho \beta(1-\beta) + \frac{C_2 \beta (2+C_0\beta)\kappa}{1-\beta} -\frac{C_1(1-\beta)}{\alpha}\right),
\end{split}
\end{equation*}
which leads to the desired inequality \eqref{Targ3} as long as $\rho,\kappa$ are such that
\begin{equation}\label{rkg-cd5}
  \rho <\frac{C_1}{2\alpha\beta},\quad \kappa<\min\set{\frac{1}{2C_2\beta}, \frac{C_1(1-\beta)^2}{2C_2(2+C_0\beta)\beta\alpha}}.
\end{equation}

\textbf{Case 5: $0<\xi_0\leq \delta$, $\xi_0 \leq \xi\leq \delta$.}

Similarly as obtaining \eqref{omeg-est000}, we have that by setting $\kappa\leq \frac{1}{2C_2}$ and $\gamma\leq \kappa$,
\begin{equation}\label{Omeg-est-a}
  \Omega(\xi,t)\partial_\xi\omega(\xi,\xi_0) \leq -\frac{1}{2}D(\xi,t) + \frac{3C_2}{1-\beta} \beta \left( \kappa m(\delta^{-1})\delta^{1-\beta} \right)^2 \xi^{2\beta-1}.
\end{equation}
For the contribution from the diffusion term, if $\xi_0 \leq \frac{\xi}{4}$, then from $\xi-2\eta >\xi_0$ for all $\eta\in [0,\frac{\xi}{4}]$ and by arguing as \eqref{Dxi-es1}, we find
\begin{align}\label{Dxe-est4}
  D(\xi,t) & \leq C_1 \int_0^{\frac{\xi}{4}}\big(\omega(\xi+2\eta,\xi_0)+\omega (\xi-2\eta,\xi_0) -2\omega (\xi,\xi_0)\big) \frac{m(\eta^{-1})}{\eta}\textrm{d} \eta \nonumber \\
  & \leq C_1 \partial_{\xi\xi}\omega(\xi,\xi_0) \int_0^{\frac{\xi}{4}} \eta m(\eta^{-1}) \dd \eta \nonumber  \\
  & \leq - C_1 \beta (1-\beta)\kappa\, m(\delta^{-1})\delta^{1-\beta} \xi^{\beta-2} \int_0^{\frac{\xi}{4}} \eta m(\eta^{-1})\dd \eta \nonumber \\
  & \leq - C_1 \beta (1-\beta)\kappa\, \left( m(\delta^{-1})\right)^2 \delta^{1+\alpha-\sigma-\beta} \xi^{\beta-2} \int_0^{\frac{\xi}{4}} \eta^{1-\alpha+\sigma}\dd\eta \nonumber \\
  & \leq -\frac{C_1 \beta (1-\beta)\kappa}{32} \left( m(\delta^{-1})\right)^2 \delta^{1+\alpha-\sigma-\beta} \xi^{\beta-\alpha+\sigma}.
\end{align}
Thus for $\xi_0\leq \frac{\xi}{4}$, we get that by letting $\kappa< \frac{C_1 (1-\beta)^2}{192C_2}$,
\begin{equation*}
\begin{split}
  \Omega(\xi,t)\partial_\xi \omega(\xi,\xi_0) + D(\xi,t) & \leq \beta\kappa \left( m(\delta^{-1})\right)^2 \delta^{1+\alpha-\sigma-\beta} \xi^{\beta-\alpha +\sigma}
  \left( \frac{3C_2}{1-\beta} \left( \frac{\delta}{\xi}\right)^{1-\alpha+\sigma-\beta}\kappa -\frac{C_1 (1-\beta)}{64} \right) \\
  & \leq \beta\kappa \left( m(\delta^{-1})\right)^2 \delta^{1+\alpha-\sigma-\beta} \xi^{\beta-\alpha +\sigma}
  \left( \frac{3C_2}{1-\beta} \kappa -\frac{C_1 (1-\beta)}{64} \right) <0.
\end{split}
\end{equation*}
Whereas if $\xi_0\geq \frac{\xi}{4}$, by using \eqref{Dxit3}, the concavity of $\omega(\eta,\xi_0)-\omega(0+,\xi_0)$ for $\eta\geq 0$ and \eqref{mfact}, we get
\begin{equation}\label{Dxe-est5}
\begin{split}
  D(\xi,t) & \leq -C_1 2 \omega(0+,\xi_0) \int_{\frac{\xi}{2}}^\infty \frac{m(\eta^{-1})}{\eta}\dd\eta \\
  & \leq -\frac{2 C_1(1-\beta) \kappa m(\delta^{-1})\delta^{1-\beta}\xi_0^\beta}{\alpha}m(\xi^{-1}) \\
  & \leq -\frac{C_1 (1-\beta)\kappa m(\delta^{-1})\delta^{1-\beta}}{2\alpha} \xi^\beta m(\xi^{-1}) \\
  & \leq -\frac{C_1 (1-\beta)\kappa }{2\alpha} \left( m(\delta^{-1})\right)^2 \delta^{1+\alpha-\sigma-\beta}\xi^{\beta-\alpha+\sigma}.
\end{split}
\end{equation}
Thus combining this estimate with \eqref{Omeg-est-a} yields that
\begin{equation*}
  \Omega(\xi,t)\partial_\xi\omega(\xi,\xi_0) + D(\xi,t)\leq \kappa \left( m(\delta^{-1})\right)^2 \delta^{1+\alpha-\sigma-\beta} \xi^{\beta-\alpha+\sigma}
  \left( \frac{3C_2 \beta}{1-\beta}\kappa - \frac{C_1(1-\beta)}{2\alpha}\right)<0,
\end{equation*}
where the last inequality is ensured by setting $\kappa < \frac{C_1(1-\beta)^2}{6 C_2 \beta}$. Notice that in this case the conditions on $\kappa$ and $\gamma$ are
\begin{equation}\label{rkg-cd6}
  \kappa< \min\Big\{\frac{1}{2C_2},\frac{C_1 (1-\beta)^2}{192C_2}\Big\},\quad \gamma\leq \kappa.
\end{equation}

\textbf{Case 6: $0<\xi_0\leq \delta$, $\xi>\delta$.}

This case is almost the same as Case 2 in the proof of Lemma \ref{lem-HoldRC}, and we omit the details.
Note that the conditions on $\kappa$, $\gamma$ are given by \eqref{kap-gam-cd1}.

Therefore, for the MOC $\omega(\xi,\xi_0)$ defined by \eqref{MOC3.1}-\eqref{MOC3.2} and $\xi_0=\xi_0(t)$ defined by \eqref{xi0}
with $\rho,\kappa,\gamma$ are appropriate constants satisfying \eqref{kap-gam-cd1}, \eqref{rkg-cd1}, \eqref{rkg-cd2}, \eqref{rkg-cd3}, \eqref{rkg-cd4}, \eqref{rkg-cd4.1}, \eqref{rkg-cd5}, \eqref{rkg-cd6},
we justify \eqref{Targ3} for all $\xi>0$ and $t >0$ based on the above analysis, and thus conclude Lemma \ref{lem-mocev}. Observing that by suppressing the dependence on the constants $C_1=C_1(d)$, $C_2=C_2(d)$, $\tilde{c}$ and $C_0$,
the conditions on positive constants $\rho,\kappa,\gamma$ are as follows
\begin{equation}\label{rkg-cdsum}
\begin{split}
  \rho \leq\frac{1}{C}\frac{(1-\beta)(\alpha-\sigma)}{\alpha^2}, \quad \kappa\leq \frac{1}{C}(1-\beta)^2, \quad
  \gamma \leq \frac{1}{C} \min \set{\beta (1-\beta)^2 , (1-\beta)^3(\alpha-\sigma)},
\end{split}
\end{equation}
with $C>0$ some constant independent of $\alpha,\sigma,\beta$.

\end{proof}

\section{Proof of Theorems \ref{thmevRe} and \ref{thm:glb2}}\label{sec-thms}

Consider the following approximate system
\begin{equation}\label{appDDeq}
  \partial_t\theta^\epsilon + u^\epsilon \cdot\nabla \theta^\epsilon + \LL \theta^\epsilon -\epsilon \Delta\theta^\epsilon =0,\quad u^\epsilon=\mathcal{P}(\theta^\epsilon),
  \quad \theta^\epsilon|_{t=0}= \theta^\epsilon_0=\phi_\epsilon*(\theta_0 1_{B_{1/\epsilon}}),
\end{equation}
where $\mathcal{P}$ is composed of zero-order pseudo-differential operators defined by \eqref{uexp}, $1_{B_{1/\epsilon}}$ is the indicator function on the ball $B_{1/\epsilon}$,
$\phi_\epsilon(x)=\epsilon^{-d}\phi(\epsilon^{-1}x)$,
and $\phi\in C^\infty_c(\R^d)$ is a radial test function satisfying $\int_{\R^d}\phi=1$.

\subsection{Proof of Theorem \ref{thmevRe}: eventual regularity of vanishing viscosity solution}\label{subsec-evRe1}
From $\theta_0\in L^2(\R^d)$,
we have $\|\theta_0^\epsilon\|_{L^2}\leq \|\theta_0\|_{L^2}$, and $\|\theta_0^\epsilon\|_{H^s}\lesssim_{\epsilon,s}\|\theta_0\|_{L^2}$ for every $s>0$.
For $\epsilon>0$ and $s> d/2 +1$, we have the following lemma concerning the global well-posedness for the approximate system \eqref{appDDeq}.
\begin{lemma}\label{lem-apSy-glb}
  For every $\epsilon>0$, the Cauchy problem of the approximate drift-diffusion equation \eqref{appDDeq} admits a uniquely global smooth solution $\theta^\epsilon(x,t)$ such that
\begin{equation*}
  \theta^\epsilon\in C([0,\infty[; H^s(\R^d))\cap C^\infty(]0,\infty[\times \R^d),\quad \textrm{with}\;\; s> d/2+1.
\end{equation*}
\end{lemma}

The proof of this lemma is more or less standard, and one can refer to \cite[Theorem 1.4]{LMX} (at $\alpha=2$ case) for the use of the nonlocal maximum principle method,
and we omit the details here.

Since $u^\epsilon$ is divergence-free, we can also show the uniform-in-$\epsilon$ energy estimate. By taking the $L^2$-inner product of the equation \eqref{appDDeq}
with $\theta^\epsilon$, and using the integration by parts, we have
\begin{equation}\label{ene-est1}
  \frac{1}{2}\frac{d}{dt}\|\theta^\epsilon(t)\|_{L^2}^2 + \int_{\R^d} \LL(\theta^\epsilon)(x,t)\,\theta^\epsilon(x,t)\dd x + \epsilon \|\nabla \theta^\epsilon(t)\|_{L^2}^2 =0.
\end{equation}
Since the symbol of $\LL$ satisfies $A(\zeta)\geq 0$ from \eqref{LKf} and \eqref{Kcond3}-\eqref{mDec2}, we see that
\begin{equation}\label{pos-fac}
  \int_{\R^d} (\LL\theta^\epsilon )(x,t)\,\theta^\epsilon(x,t)\dd x = \int_{\R^d} A(\zeta) |\widehat{\theta^\epsilon}(\zeta,t)|^2\dd \zeta \geq 0.
\end{equation}
Inserting \eqref{pos-fac} into \eqref{ene-est1} leads to $\frac{d}{dt}\|\theta^\epsilon(t)\|_{L^2}^2 \leq 0$, which by integrating in time implies
\begin{equation}\label{ene-est2}
  \|\theta^\epsilon(t)\|_{L^2}\leq \|\theta^\epsilon_0\|_{L^2}\leq \|\theta_0\|_{L^2},\quad \forall t\geq 0.
\end{equation}
By applying Lemma \ref{lem-symb}, we also obtain
\begin{equation}\label{Lthe-est}
\begin{split}
  \int_{\R^d} (\LL\theta^\epsilon )(x,t)\,\theta^\epsilon(x,t)\dd x
  & \geq C^{-1}\int_{\R^d} |\zeta|^{\alpha-\sigma}|\widehat{\theta^\epsilon}(\zeta,t)|^2\dd\zeta - C \int_{\R^d} |\widehat{\theta^\epsilon}(\zeta,t)|^2\dd\zeta \\
  & \geq C^{-1} \|\theta^\epsilon\|_{\dot H^{\frac{\alpha-\sigma}{2}}}^2 - C \|\theta^\epsilon\|_{L^2}^2.
\end{split}
\end{equation}
Plugging this estimate into \eqref{ene-est1}, and using \eqref{ene-est2}, we find
\begin{equation*}\label{ene-key0}
  \frac{d}{dt}\|\theta^\epsilon(t)\|_{L^2}^2 + \frac{2}{C}\|\theta^\epsilon(t)\|_{\dot H^{\frac{\alpha-\sigma}{2}}}^2 \leq 2C \|\theta_0\|_{L^2}^2,
\end{equation*}
which ensures that for every $T>0$,
\begin{equation}\label{ene-key}
  \sup_{t\in [0,T]}\|\theta^\epsilon(t)\|_{L^2}^2 + \frac{2}{C}\int_0^T \|\theta^\epsilon (t)\|_{\dot H^{\frac{\alpha-\sigma}{2}}}^2\dd t \leq (1+ 2C T)\|\theta_0\|_{L^2}^2.
\end{equation}

Next based on the uniform $L^2$ estimate, we can use the De Giorgi's method to show the $L^\infty$-improvement, that is,
for any fixed $t_0>0$ and every $T\geq t_0$, there is a constant $C_*>0$ independent of $\epsilon$ and $T$ so that
\begin{equation}\label{theLinf}
  \sup_{t\in [t_0,T]}\|\theta^\epsilon(t)\|_{L^\infty_x}\leq C_*\left( \frac{1}{t_0}+C\right)^{\frac{d}{2(\alpha-\sigma)}} (1+2CT)^{\frac{1}{2}}\|\theta_0\|_{L^2},
\end{equation}
with $C>0$ the constant appearing in \eqref{ene-key}.
The proof is similar to that of \cite[Corollary 4]{CV} or \cite[Theorem 2.1]{ConWu2}, 
and here we sketch the main process in obtaining \eqref{theLinf}.
Since the operator $\LL$ defined by \eqref{Lexp} has nonnegative kernel $K$, by arguing as obtaining a corresponding inequality for fractional Laplacian operator in \cite{CorC},
we have that for every convex function $\psi$, $\psi'(\theta^\epsilon)\, \LL(\theta^\epsilon) \geq \LL (\psi(\theta^\epsilon))$.
We also find for every convex $\psi$, $-\psi'(\theta^\epsilon)\,\Delta\theta^\epsilon \geq - \Delta (\psi(\theta^\epsilon))$.
For $M>0$ chosen later (cf. \eqref{Mdef}), applying the above two inequalities with
\begin{equation}\label{}
  \psi(\theta^\epsilon)=(\theta^\epsilon-M_k)_+ =: \theta^\epsilon_k, \quad M_k:=M(1-2^{-k}),\;\; k\in \N,
\end{equation}
we obtain the following pointwise inequality from \eqref{appDDeq},
\begin{equation}\label{}
  \partial_t \theta^\epsilon_k + u^\epsilon\cdot \nabla \theta^\epsilon_k + \LL \theta^\epsilon_k -\epsilon \Delta \theta^\epsilon_k \leq 0.
\end{equation}
As deriving the energy estimate, we use \eqref{Lthe-est} to get
\begin{equation}\label{ene-est4}
  \frac{1}{2} \frac{d}{dt}\|\theta^\epsilon_k(t)\|_{L^2}^2 + C^{-1} \|\theta^\epsilon_k(t)\|_{\dot H^{\frac{\alpha-\sigma}{2}}}^2 +\epsilon \|\nabla \theta^\epsilon_k(t)\|_{L^2}^2 \leq C \|\theta^\epsilon_k(t)\|_{L^2}^2.
\end{equation}
Then for a fixed constant $t_0>0$ and every $T\geq t_0$, we denote $T_k:=t_0(1-2^{-k})$, $k\in \N$, and the level set of energy as
\begin{equation*}
  U^\epsilon_k := \sup_{t\in [T_k,T]} \|\theta^\epsilon_k(t)\|_{L^2}^2 + \frac{2}{C} \int_{T^k}^T \|\theta_k^\epsilon(t)\|_{\dot H^{\frac{\alpha-\sigma}{2}}}^2\dd t.
\end{equation*}
For some $s\in [T_{k-1}, T_k]$, we integrating \eqref{ene-est4} in time between $s$ and $t\in [T_k,T]$, and also between $s$ and $T$ to find
\begin{equation*}
\begin{split}
  &  \|\theta_k^\epsilon(t)\|_{L^2}^2 \leq \|\theta_k^\epsilon(s)\|_{L^2}^2 + 2C\int_s^t \|\theta_k^\epsilon(t)\|_{L^2}^2 \dd t,\quad \textrm{and} \\
  \frac{2}{C} & \int_s^T \|\theta_k^\epsilon(t)\|_{\dot H^{\frac{\alpha-\sigma}{2}}}^2\dd t \leq \|\theta_k^\epsilon(s)\|_{L^2}^2 + 2 C\int_s^T \|\theta_k^\epsilon(t)\|_{L^2}^2 \dd t,
\end{split}
\end{equation*}
which implies $U_k^\epsilon \leq 2 \|\theta_k^\epsilon(s)\|_{L^2}^2 + 4C\int_s^T \|\theta_k^\epsilon(t)\|_{\dot H^{\frac{\alpha-\sigma}{2}}}^2\dd t$.
Taking the mean in $s$ on $[T_{k-1},T_k]$, we infer
\begin{equation}\label{Uk-est}
  U_k^\epsilon \leq \left(\frac{2^{k+1}}{t_0}+ 4C\right) \int_{T_{k-1}}^T \|\theta_k^\epsilon(t)\|_{L^2}^2\dd t.
\end{equation}
The inequality \eqref{Uk-est} is almost identical with \cite[(A.3)]{ConWu2}, and we can proceed further to obtain
\begin{equation*}
  U_k^\epsilon \leq \left( \frac{2}{t_0} + 4C\right) \frac{2^{k(q-1)}}{M^{q-2}} (U_{k-1}^\epsilon)^{q/2},\quad \textrm{with}\;\; q:= 2+\frac{2(\alpha-\sigma)}{d}.
\end{equation*}
Since $U_0^\epsilon\leq (1+2CT)\|\theta_0\|_{L^2}^2$, by choosing $M$ (owing to \cite[Lemma 2.6]{WuX}) to be
\begin{equation}\label{Mdef}
  M = (1+2 CT)^{1/2}\|\theta_0\|_{L^2}\left( 2^{2+\frac{d}{\alpha-\sigma}}\left(\frac{2}{t_0} + 4C\right) \right)^{\frac{d}{2(\alpha-\sigma)}},
\end{equation}
we have $\lim_{k\rightarrow \infty}U_k^\epsilon = 0$, which ensures $\theta^\epsilon\leq M$ for all $t\in [t_0,T]$. The same result likewise holds for $-\theta^\epsilon$,
and thus we conclude \eqref{theLinf}.

Hence, the uniform estimate \eqref{ene-key} and \eqref{theLinf} guarantee that, for some $t_0>0$ and every $T\geq t_0$, up to a subsequence
$\theta^\epsilon$ converges weakly (weakly-$*$ in $L^\infty_t L^2_x$ and $L^\infty([t_0,T]\times\R^d)$) to some function $\theta$ belonging to
\begin{equation}\label{eq:space}
L^\infty([0,T]; L^2(\R^d))\cap L^2([0,T]; \dot H^{\frac{\alpha-\sigma}{2}}(\R^d))\cap L^\infty([t_0,T]\times \R^d).
\end{equation}
Moreover, by using the compactness argument (e.g. \cite[Proposition 6.3]{MX-gQG}), we can show that
$\theta^\epsilon\rightarrow \theta$ and $u^\epsilon\rightarrow u=\mathcal{P}(\theta)$ both strongly in $L^2([0,T]; L^2_{\textrm{loc}}(\R^d))$. Thus we can pass the weak limit $\epsilon\rightarrow 0$
in the approximate system \eqref{appDDeq} to show that $\theta(x,t)$ is a global weak solution for the original equation \eqref{DDeq}-\eqref{uexp}, which satisfies the energy estimate
\eqref{ene-key} and $L^\infty$-estimate \eqref{theLinf} with $\theta$ in place of $\theta^\epsilon$.

Now applying Proposition \ref{prop-evRe} to the approximate equation \eqref{appDDeq} (with $\tilde\theta^\epsilon(t):=\theta^\epsilon(t+t_0)$ replacing $\theta^\epsilon(t)$) and Fatou's lemma,
we get that for every $\beta\in ]1-\alpha+\sigma,1[$ and every $T> t_0+t_1$,
\begin{equation}\label{theCb-est}
  \sup_{t\in [t_0+t_1, T]}\|\theta(t)\|_{\dot C^\beta(\R^d)}\leq C(\|\theta_0\|_{L^2},t_0,\alpha,\beta,\sigma,d,T), 
\end{equation}
with $t_1$ the time introduced above. Furthermore, \eqref{theCb-est} yields that for every $\beta\in ]1-\alpha+\sigma,1[$,
\begin{equation*}
\begin{split}
  \sup_{t\in [t_0+t_1,T]}\|u(t)\|_{C^\beta} & \leq C_0\sup_{t\in [t_0+t_1,T]}\|\theta(t)\|_{L^2} +C_0\sup_{t\in [t_0+t_1,T]}\|\theta(t)\|_{\dot C^\beta} \\
  & \leq C_0\, C(\|\theta_0\|_{L^2},t_0,\alpha,\beta,\sigma,d,T),
\end{split}
\end{equation*}
which together with Lemma \ref{lem-RC} implies the $C^\infty_{x,t}$-regularity of $\theta(x,t)$ for all $t\in ]t_0+ t_1,T]$.

Besides, if $\alpha\in ]0,1[$ and $\sigma=0$ in the condition (A3), i.e. $m(y)\equiv C_0 |y|^\alpha$, $\forall |y|>0$,
from \eqref{Aest2}, we have that there is no term $-\|\theta^\epsilon\|_{L^2}^2$ in \eqref{Lthe-est} and the constant $C$ in the right-hand side of \eqref{ene-key}, \eqref{theLinf} and \eqref{Mdef} can be replaced with $0$,
which guarantees that $T$ in \eqref{eq:space}-\eqref{theCb-est} can be chosen to be $\infty$. Next by choosing $\beta=1-\frac{\alpha}{2}$, we see that $\gamma=\frac{\alpha^4}{C}$,
and \eqref{t1-bdd2} just reduces to
\begin{equation}\label{t1-bdd3}
  t_1\leq \frac{C}{\alpha} \left( \frac{C(1-\alpha)}{\alpha^5}\right)^{\frac{\alpha}{1-\alpha}} \|\theta(t_0)\|_{L^\infty}^{\frac{\alpha}{1-\alpha}};
\end{equation}
moreover \eqref{theLinf} becomes
\begin{equation}\label{the-t0}
  \|\theta(t_0)\|_{L^\infty}\leq \left( \frac{C 2^{d/\alpha} }{t_0} \right)^{d/(2\alpha)} \|\theta_0\|_{L^2},
\end{equation}
which combined with \eqref{t1-bdd3} leads to \eqref{t1-bdd}. 
Thus we finish the proof of Theorem \ref{thmevRe}.

\subsection{Proof of Theorem \ref{thm:glb2}: global regularity result in the logarithmically supercritical case}

Considering the $\epsilon$-regularized equation \eqref{appDDeq} under the assumptions of Theorem \ref{thm:glb2},
by virtue of the standard Bony's para-differential calculus and Lemma \ref{lem-symb}, there is a uniquely global smooth solution $\theta^\epsilon(x,t)$ to the system \eqref{appDDeq} so that
$\theta^\epsilon\in C([0,\infty[; H^s(\R^d))\cap C^\infty(]0,\infty[\times \R^d)$ with $s> d/2+1$.
Owing to Lemma \ref{lem-theMP}, we have the uniform-in-$\epsilon$ $L^\infty$-estimate
$\sup_{t\geq 0}\|\theta^\epsilon(t)\|_{L^\infty}\leq B_0$ with
\begin{equation}\label{B0}
B_0:=
\begin{cases}
\|\theta_0\|_{L^\infty},\quad & \textrm{if Case (II) is considered,} \\
C(\|\theta_0\|_{L^2\cap L^\infty},\sigma,d),\quad & \textrm{if Case (III) is considered,}
\end{cases}
\end{equation}
and the uniform energy estimate $\|\theta^\epsilon(t)\|_{L^2}\leq \|\theta_0\|_{L^2}$, $\forall t\geq 0$ if \textrm{Case (III)} is considered.

We will apply the method of nonlocal maximum principle as in Section \ref{sec-evRe} to show the uniform-in-$\epsilon$ global regularity result.
Let $A_0\leq \frac{c_0}{2}$ be a positive constant chosen later ($c_0$ is the constant appearing in \eqref{Kcond1}), then analogous with \eqref{MOC3.1}-\eqref{MOC3.2}, we introduce the following family of moduli of continuity that for $\xi_0\in ]\delta, A_0]$,
\begin{equation}\label{MOC4.1}
  \omega(\xi,\xi_0)=
  \begin{cases}
    (1-\beta)\kappa m(\delta^{-1})\delta + \gamma\int_\delta^{\xi_0}m(\eta^{-1})\dd \eta -\gamma m(\xi_0^{-1})(\xi_0-\delta)+ \beta\kappa m(\delta^{-1})\xi, &  \textrm{for}\; 0 \leq \xi \leq \delta, \\
    \kappa m(\delta^{-1})\delta + \gamma\int_\delta^{\xi_0}m(\eta^{-1})\dd \eta -\gamma m(\xi_0^{-1})\xi_0 + \gamma m(\xi_0^{-1})\xi,& \textrm{for}\;\delta<\xi\leq \xi_0, \\
    \kappa m(\delta^{-1})\delta + \gamma \int_\delta^\xi m(\eta^{-1}) \dd \eta, & \textrm{for}\; \xi_0<\xi\leq c_0, \\
    \omega(c_0,\xi_0), &\textrm{for}\; \xi >c_0,
  \end{cases}
\end{equation}
and for $\xi_0 \leq \delta$,
\begin{equation}\label{MOC4.2}
  \omega(\xi,\xi_0)=
  \begin{cases}
    (1-\beta)\kappa m(\delta^{-1})\delta^{1-\beta}\xi_0^\beta + \beta \kappa m(\delta^{-1})\delta^{1-\beta}\xi_0^{\beta-1}\xi,
    &\quad \textrm{for}\; 0\leq \xi< \xi_0, \\
    \kappa m(\delta^{-1}) \delta^{1-\beta} \xi^\beta, &\quad \textrm{for}\; \xi_0\leq \xi\leq \delta, \\
    \kappa m(\delta^{-1}) \delta + \gamma \int_\delta^\xi m(\eta^{-1})\,\dd \eta, & \quad \textrm{for}\; \delta<\xi\leq c_0 \\
    \omega(c_0,\xi_0),&\quad\textrm{for}\; \xi>c_0,
  \end{cases}
\end{equation}
where $\beta\in ]\sigma,1[$, $\kappa,\gamma,\delta$ are appropriate positive constants, and $\xi_0=\xi_0(t)$ is given by
\begin{equation}\label{xi0-2}
  \frac{d}{dt}\xi_0= - \rho \, m(\xi_0^{-1}) \xi_0, \quad \xi_0(0)= A_0,
\end{equation}
with $\rho>0$ some constant chosen later.

First we prove that under the condition \eqref{mcd1}, and by suitably choosing $\delta$, the initial data $\theta_0$ strictly obeys the initial MOC $\omega(\xi,\xi_0(0))=\omega(\xi,A_0)$.
Indeed, it suffices to show
\begin{equation}\label{eq:Targ}
  \omega(0+,A_0) > 2B_0.
\end{equation}
Without loss of generality we also assume $A_0\leq c_2^{-1}$ (with $c_2\geq 1$ appearing in \eqref{mcd1}), and we see that
\begin{align}\label{eq:calcu}
  \omega(0+,A_0) & =(1-\beta)\kappa m(\delta^{-1})\delta + \gamma\int_\delta^{A_0}m(\eta^{-1})\dd \eta -\gamma m(A_0^{-1})(A_0-\delta) \nonumber\\
  & \geq \frac{\gamma}{c_2} \int_\delta^{A_0} \frac{1}{\eta(\log \eta^{-1} )^\mu}\dd \eta - \gamma m(1)  \nonumber \\
  & \geq \frac{\gamma}{c_2} \int_{\frac{1}{A_0}}^{\frac{1}{\delta}} \frac{1}{\eta(\log \eta )^\mu}\dd \eta - \gamma m(1) \\
  & \geq \nonumber
  \begin{cases}
    \frac{\gamma}{c_2(1-\mu)}\left( \left( \log\frac{1}{\delta}\right)^{1-\mu}- \left( \log \frac{1}{A_0}\right)^{1-\mu} \right)-\gamma m(1) , & \quad \textrm{if   }\mu\in[0,1[, \\
    \frac{\gamma}{c_2} \left( \log\log \frac{1}{\delta} -\log\log \frac{1}{A_0}\right) -\gamma m(1), & \quad \textrm{if   }\mu=1.
  \end{cases}
\end{align}
In order to achieve \eqref{eq:Targ}, if $\mu\in [0,1[$, we need
\begin{equation*}
  \log \frac{1}{\delta} >  \left[\left( \log \frac{1}{A_0}\right)^{1-\mu} + \frac{c_2(1-\mu)}{\gamma}\big(2B_0 +\gamma m(1)\big)\right]^{\frac{1}{1-\mu}},
\end{equation*}
and from the inequality $(a+b)^{\frac{1}{1-\mu}}\leq C_\mu (a^{\frac{1}{1-\mu}} + b^{\frac{1}{1-\mu}})$ for $a,b>0$, it suffices to choose $\delta$ as
\begin{equation}\label{del-cd1}
  \delta = A_0^{C_\mu} \exp\Big(- C_\mu \Big(\frac{c_2(1-\mu)}{\gamma}\big(3B_0 +\gamma m(1)\big)\Big)^{1/(1-\mu)}\Big);
\end{equation}
whereas if $\mu=1$, it suffices to set $\delta$ as
\begin{equation*}
  \log\log\frac{1}{\delta} =\log \log \frac{1}{A_0} + \frac{c_2}{\gamma}\big(3B_0+\gamma m(1) \big),
\end{equation*}
that is,
\begin{equation}\label{del-cd2}
  \delta =A_0^{\exp \left(\frac{c_2}{\gamma}\big(3B_0+\gamma m(1) \big)\right)}.
\end{equation}

Next by using \eqref{eq:calcu}, we see that the MOC defined by \eqref{MOC4.1}-\eqref{MOC4.2} satisfies for all $0<\xi_0\leq A_0$,
$$\omega(A_0,\xi_0)\geq \omega(A_0,0+)> \gamma \int_\delta^{A_0} m(\eta^{-1}) \dd \eta > 2 B_0 ,$$
thus according to Proposition \ref{prop-GC}, we only need to justify the following criterion
\begin{equation}\label{Targ4}
  - \partial_{\xi_0}\omega(\xi,\xi_0) \dot\xi_0(t)+ \Omega(\xi,t)\partial_\xi\omega(\xi,\xi_0) + D(\xi,t) + \epsilon \partial_{\xi\xi}\omega(\xi,\xi_0)<0,
\end{equation}
for all $t>0$, $0<\xi_0\leq A_0$, $0<\xi\leq A_0$ with $A_0\leq \frac{c_0}{2}$, and  
\begin{equation}\label{Dxit.0}
\begin{split}
  D(\xi,t) \leq &
  \,C_1'\omega(\xi,\xi_0) + C_1 \int_0^{\frac{\xi}{2}}\left(\omega(\xi+2\eta,\xi_0)+\omega (\xi-2\eta,\xi_0) -2\omega (\xi,\xi_0)\right) \frac{m(\eta^{-1})}{\eta}\textrm{d} \eta \\
  & + C_1 \int_{\frac{\xi}{2}}^{A_0} \big(\omega (2\eta+\xi,\xi_0)-\omega (2\eta-\xi,\xi_0) -2\omega (\xi,\xi_0)\big) \frac{m(\eta^{-1})}{\eta}\textrm{d}\eta,
\end{split}
\end{equation}
and
\begin{equation}\label{Omega.0}
  \Omega(\xi,t) \leq -\frac{C_2}{m(\xi^{-1})}  D(\xi,t)+ (C_2+C_2') \omega(\xi,\xi_0) + C_2 \xi \int_\xi^\infty \frac{\omega(\eta,\xi_0)}{\eta^2} \dd \eta,
\end{equation}
and $C_1=C_1(d), C_2=C_2(d)>0$, and $C_1'$, $C_2'$ are just $0$ if {\textrm Case (II)} is assumed, and are the constants (depending only on $d,\tilde\alpha,c_0,c_1$) respectively appearing in \eqref{Dxi-est2} and \eqref{Ome-es2} if {\textrm Case (III)} is assumed.

By arguing as Lemma \ref{lem-mocev}, we indeed can prove the desired inequality \eqref{Targ4} as long as that $\rho,\kappa,\gamma$ are suitable constants satisfying \eqref{rkg-cdsum}
(maybe with slightly different $C$) and $A_0$ satisfying
\begin{equation}\label{A0bd1}
  0< A_0 \leq \min \Big\{\Big( \frac{C_1 \tilde{c}\,m(1) (1-\beta)(1-\sigma)}{64C_1'}\Big)^{1/(1-\sigma)},\frac{c_0}{2},c_2^{-1}\Big\},
\end{equation}
We will present the different points compared to the proof of Lemma \ref{lem-mocev} in the end of this subsection.

Then at the time $t_1$ satisfying $\xi_0(t_1)=0$, we have that $\theta^\epsilon(x,t_1)$ uniformly-in-$\epsilon$ obeys the MOC $\omega(\xi)=\omega(\xi,0+)$ given by
\begin{equation}\label{MOC4.3}
\omega(\xi)=
\begin{cases}
  \kappa\, m(\delta^{-1}) \delta^{1-\beta} \xi^\beta, & \quad \textrm{for}\;\; 0\leq \xi \leq \delta, \\
  \kappa\, m(\delta^{-1})\delta + \gamma \int_\delta^\xi m(\eta^{-1})\dd\eta, & \quad \textrm{for} \;\; \delta< \xi \leq c_0,\\
  \omega(c_0), & \quad \textrm{for} \;\; \xi > c_0,
\end{cases}
\end{equation}
with $\kappa,\gamma>0$ the suitable constants satisfying \eqref{rkg-cdsum}.
In a similar way as obtaining \eqref{eq:claim}, we moreover get that the MOC $\omega(\xi)$ given by \eqref{MOC4.3} is uniformly-in-$\epsilon$ strictly preserved by the solution $\theta^\epsilon(x,t_1)$.
Then we can argue as the proof of Lemma \ref{lem-HoldRC} and Steps 5-6 in the proof of \eqref{Targ4} to show that for every $t>t_1$ and $0<\xi\leq A_0$,
\begin{equation}\label{Targ5}
  \Omega(\xi,t)\omega'(\xi) + D(\xi,t) + \epsilon \omega''(\xi) < 0,
\end{equation}
where $D(\xi,t)$ and $\Omega(\xi,t)$ are given by \eqref{Dxit.0} and \eqref{Omega.0} with $\omega(\cdot)$ in place of $\omega(\cdot,\xi_0)$,
which guarantees that $\theta^\epsilon(x,t)$ uniformly-in-$\epsilon$ strictly preserve such a MOC $\omega(\xi)$ for all time $t\geq t_1$.

From \eqref{rkg-cdsum}, we can choose the positive constants $\rho,\kappa,\gamma$ as
\begin{equation*}
  \rho = (1-\beta)(1-\sigma)/\overline{C}, \quad \kappa = (1-\beta)^2/\overline{C}, \quad
  \gamma =  \min \set{\beta (1-\beta)^2 , (1-\beta)^3(1-\sigma)}/\overline{C},
\end{equation*}
with $\overline{C}>0$ the suitable constant depending only on $d$, $C_2'$ (and $a, |\Psi|$).
Thanks to \eqref{t1-est}, \eqref{mcd1} and \eqref{Cbetaest}, we find that the eventual regularity time $t_1$ satisfies
\begin{equation}\label{t1-bdd4}
  t_1\leq  \frac{1}{(1-\sigma)\rho \,m(A_0^{-1})} \leq \frac{ \overline{C} c_2}{(1-\sigma)^2(1-\beta)} A_0 (\log A_0^{-1})^\mu \leq \frac{ \overline{C} C_0 c_2}{(1-\sigma)^2(1-\beta)} A_0^{\frac{1}{2}},
\end{equation}
and for every $\beta\in ]\sigma,1[$,
\begin{equation}\label{the-est1}
\begin{split}
    \sup_{t\in [t_1,\infty[} &  \|\theta^\epsilon(t)\|_{\dot C^\beta(\R^d)}\leq\kappa m(\delta^{-1})\delta^{1-\beta} \leq \kappa\,m(1) \delta^{-\beta} \\
    & \leq
    \begin{cases}
      \frac{\overline{C}m(1)}{(1-\beta)^2}  A_0^{-C_\mu \beta}\exp\Big( \beta C_\mu \left(\frac{3  c_2(1-\mu)B_0 }{\gamma}+ c_2(1-\mu)m(1)\right)^{1/(1-\mu)}\Big), &\quad
      \textrm{if    } \mu\in [0,1[, \\
      \frac{\overline{C}m(1)}{(1-\beta)^2} \left( A_0^{-1}\right)^{\beta \exp \big(\frac{ 3 c_2B_0}{\gamma}+ c_2 m(1) \big)}, &\quad \textrm{if    }\mu=1.
    \end{cases}
\end{split}
\end{equation}

Now for any $t_*>0$, by virtue of \eqref{t1-bdd4}, we also need $A_0$ satisfies that $\frac{ \overline{C} C_0 c_2}{(1-\sigma)^2(1-\beta)} A_0^{\frac{1}{2}} \leq \frac{t_*}{2}$, i.e. $A_0\leq \left( \frac{(1-\sigma)^2(1-\beta) t_*}{2\overline{C} C_0 c_2}\right)^2$,
thus for each $\sigma\in [0,1[$ and $\beta\in ]\sigma,1[$, we can choose $A_0$ to be
\begin{equation}
  A_0 = \min \Big\{\Big(\frac{C_1 \tilde{c}\,m(1) (1-\beta)(1-\sigma)}{64C_1'}\Big)^{1/(1-\sigma)},\frac{c_0}{2}, c_2^{-1}, \Big( \frac{(1-\sigma)^2(1-\beta) t_*}{2\overline{C} C_0 c_2}\Big)^2\Big\},
\end{equation}
so that the uniform-in-$\epsilon$ H\"older estimate \eqref{the-est1} holds true.
According to Lemma \ref{lem-RC} and the Calder\'on-Zygmund theorem, we can further get $\theta^\epsilon\in C^\infty([t_*,\infty[\times \R^d)$ uniformly in $\epsilon$.

If Case (III) is considered and $\divg u=0$, $\theta_0\in L^2\cap L^\infty(\R^d)$, in a similar way as deriving \eqref{ene-key}, we apply Lemma \ref{lem-symb} to get that $\theta^\epsilon \in L^\infty([0,T]; L^2(\R^d))\cap L^2([0,T]; \dot H^{\frac{1-\sigma}{2}(\R^d)})$
uniformly-in-$\epsilon$ for every $T>0$. Since $u$ is divergence-free, for any $T>0$, similarly as the corresponding part at Subsection \ref{subsec-evRe1}, we can pass $\epsilon\rightarrow 0$ in \eqref{appDDeq}
to obtain the existence of weak solution $\theta \in L^\infty([0,T]; L^2(\R^d))\cap L^2([0,T]; \dot H^{\frac{1-\sigma}{2}}(\R^d))$ to the equation \eqref{DDeq}-\eqref{uexp} which also satisfies $\theta\in  C^\infty([t_*,T]\times \R^d)$, as desired.

If Case (II) is considered, $\theta_0\in C_0(\R^d)$, and there is no divergence-free condition of $u$, we can pass $\epsilon\rightarrow 0$ to get a limit function $\theta\in L^\infty([0,\infty[\times \R^d)\cap C^\infty([t_*,\infty[\times\R^d)$.
For any $t_*>0$, the limit function $\theta$ on the time period $[t_*,\infty[$ satisfies the equation \eqref{DDeq}-\eqref{uexp} (but it is not so clear whether $\theta$ is a weak solution to \eqref{DDeq}-\eqref{uexp} on $[0,t_*]$).\\

Finally, we state the different points of proving \eqref{Targ4} compared to the proof of \eqref{Targ3}.

Case 1: $\delta<\xi_0\leq A_0$, $0<\xi\leq \delta$.

Since $\partial_\eta\omega(\eta,\xi_0)=0$ for all $\eta>c_0$, we can prove the estimates analogous to \eqref{omegint} and \eqref{omeg-est0} with $C_2+C_2'$ in place of $C_2$. For the contribution from the diffusion term,
we have (noting that $\alpha=1$)
\begin{align}\label{Dxe-est2}
  D(\xi,t) &  \leq C_1' \omega(\xi,\xi_0) - 2C_1 \omega(0+,\xi_0) \int_{\frac{\xi}{2}}^{A_0} \frac{m(\eta^{-1})}{\eta}\dd\eta \nonumber \\
  & \leq C_1' \omega(\xi,\xi_0) - 2C_1\omega(0+,\xi_0) \left( \frac{\xi}{2}\right) m\left( \frac{2}{\xi}\right) \int_{\frac{\xi}{2}}^\xi  \frac{1}{\eta^2} \dd\eta \nonumber \\
  & \leq C_1' \left( \omega(0+,\xi_0) + \beta \kappa m(\delta^{-1})\delta\right) - C_1 \omega(0+,\xi_0) m(\xi^{-1}) \nonumber \\
  & \leq C_1' \beta \kappa m(\delta^{-1})\delta - \frac{C_1}{2} \omega(0+,\xi_0)m(\xi^{-1}),
\end{align}
where in the last line we used the estimate $m(\xi^{-1})\geq m(A_0^{-1})\geq \frac{2C_1'}{C_1}$, which is implied by a stronger condition $A_0\leq \left( \frac{C_1 m(1)}{2C_1'}\right)^{1/(1-\sigma)}$.
Since $\omega(0+,\xi_0)\geq M_{\xi_0,\delta}$ by \eqref{Mxi0}, thus if $\xi_0 \leq N\delta$ with $N\in\N$ defined in \eqref{Ncond}, thanks to \eqref{Mxi0-est1}, we get
\begin{equation*}{}
\begin{split}
  D(\xi,t) & \leq  C_1' \beta \kappa m(\delta^{-1})\delta - \frac{C_1 (1-\beta) \kappa}{4} m(\delta^{-1})\delta \,m(\xi^{-1}) - \frac{C_1(1-\sigma)\,\gamma}{4} m(\xi_0^{-1})\xi_0 m(\xi^{-1}) \\
  & \leq - \frac{C_1 (1-\beta) \kappa}{8} m(\delta^{-1})\delta \,m(\xi^{-1}) - \frac{C_1(1-\sigma)\,\gamma}{4} m(\xi_0^{-1})\xi_0 m(\xi^{-1}),
\end{split}
\end{equation*}
where in the second line we used $A_0 \leq \left( \frac{C_1 m(1) (1-\beta)}{8 C_1' \beta} \right)^{\frac{1}{1-\sigma}}$;
whereas if $\xi_0\leq N\delta$, by virtue of \eqref{D2xit-es2}, we see that through setting $\gamma \leq \frac{(1-\beta)(1-\sigma)\kappa}{16} $,
\begin{align*}
  D(\xi,t) & \leq C_1' \beta \kappa m(\delta^{-1})\delta - \frac{C_1}{4} \left( (1-\beta)\kappa\, m(\delta^{-1})\delta -\gamma  m(\xi_0^{-1})\xi_0 \right) m(\xi^{-1}) \\
  & \leq -\frac{C_1}{4} \left( \frac{(1-\beta) \kappa}{2}-\frac{4\gamma}{1-\sigma} \right) \left( m(\delta^{-1}) \right)^2 \delta \leq -\frac{C_1(1-\beta)\kappa}{16} \left( m(\delta^{-1})\right)^2 \delta,
\end{align*}
where we also used $A_0 \leq \left( \frac{C_1 m(1) (1-\beta)}{8 C_1' \beta} \right)^{\frac{1}{1-\sigma}}$. Hence under the conditions \eqref{rkg-cd1}, \eqref{rkg-cd2}
(up to some pure numbers and $C_2$ replaced by $C_2+C_2'$), we show that \eqref{Targ4} holds in this case.

Case 2: $\delta<\xi_0\leq A_0$, $\delta<\xi\leq \xi_0$.

The different points are quite similar to those stated in Case 1 above, and under the (slightly modified) conditions \eqref{rkg-cd3} and \eqref{rkg-cd4},
we can show \eqref{Targ4} in this case.

Case 3: $\delta<\xi_0\leq A_0$, $\xi_0<\xi\leq A_0$.

We obtain \eqref{Oxe-est1} with $C_2$ replaced by $C_2+C_2'$. For $D(\xi,t)$, similarly as \eqref{Dxe-est1} and \eqref{Dxe-est2}, we have
\begin{align*}
  D(\xi,t) & \leq C_1'\omega(\xi,\xi_0) + C_1 \big(\omega(2\xi,\xi_0)-2\omega(\xi,\xi_0) \big) \int_{\frac{\xi}{2}}^{A_0}\frac{m(\eta^{-1})}{\eta}\dd \eta \\
  & \leq C_1'\omega(\xi,\xi_0) + \frac{C_1}{2} \big(\omega(2\xi,\xi_0) -2\omega(\xi,\xi_0)\big) m(\xi^{-1}),
\end{align*}
thus if $\xi \geq \delta \left( \frac{1}{2^{1-\sigma}-1}\right)^{1/\sigma}$, by using \eqref{om-fact1}, we get
\begin{equation*}
\begin{split}
  D(\xi,t) \leq C_1' \omega(\xi,\xi_0) - \frac{C_1\tilde{c}(1-\sigma)}{4} \omega(\xi,\xi_0) m(\xi^{-1})
  \leq -\frac{C_1 \tilde{c} (1-\sigma)}{8}\omega(\xi,\xi_0) m(\xi^{-1}),
\end{split}
\end{equation*}
where $\tilde{c} = \inf_{x\in ]0,1]}\set{\frac{2^x-1}{x}}>0$ and in the last inequality we used
$ A_0 \leq \left( \frac{C_1 m(1)\tilde{c}\, (1-\sigma)}{ 8 C_1' }\right)^{\frac{1}{1-\sigma}}$.
If $\xi \leq \delta \left( \frac{1}{2^{1-\sigma}-1}\right)^{1/\sigma}$, by arguing as \eqref{Dxe-est3} and \eqref{Dxe-est2}, and using \eqref{om-est1}, we find
\begin{align*}
  D(\xi,t) & \leq C_1' \omega(\xi,\xi_0) - \frac{\omega(0+,\xi_0)}{2} m(\xi^{-1}) \\
  & \leq C_1' \left(\kappa + \frac{\tilde{c}\gamma}{2(1-\sigma)} \right) m(\delta^{-1})\delta - \frac{(1-\beta)\kappa }{2} m(\delta^{-1})\delta m(\xi^{-1}) \\
  & \leq -\frac{(1-\beta)\kappa}{4} m(\delta^{-1})\delta m(\xi^{-1}),
\end{align*}
where in the last line we used $\gamma\leq \kappa$ and
$A_0 \leq \left( \frac{C_1 m(1)\, (1-\beta) (1-\sigma)}{ 4 C_1' }\right)^{\frac{1}{1-\sigma}}$.
The remaining proof is similar to Case 3 in the proof of Lemma \ref{lem-mocev}, and \eqref{Targ4} holds in this case under (slightly modified) \eqref{rkg-cd4.1} and \eqref{rkg-cd4.2}.

Case 4: $0<\xi_0\leq \delta$, $0<\xi< \xi_0$.

We have \eqref{Omeg-est3} with $C_2$ replaced by $C_2 +C_2'$. In a similar treatment as \eqref{D2xit3} and \eqref{Dxe-est2}, we infer
\begin{equation*}
  D(x,t)\leq C_1' \omega(\xi,\xi_0) - \frac{C_1}{2}\omega(0+,\xi_0) m(\xi^{-1})\leq -\frac{C_1}{4}\omega(0+,\xi_0) m(\xi^{-1}),
\end{equation*}
where in the last inequality we used the fact $\omega(\xi,\xi_0)\leq \frac{1}{1-\beta}\omega(0+,\xi_0)$ (from \eqref{om-est2}) and the condition $A_0\leq \left( \frac{C_1 m(1)(1-\beta)}{4C_1'}\right)^{\frac{1}{1-\sigma}}$. Thus we can obtain \eqref{Targ4} in this case under (slightly modified) \eqref{rkg-cd5}.

Case 5: $0<\xi_0\leq \delta$, $\xi_0\leq \xi\leq \delta$.

We have \eqref{Omeg-est-a} with $C_2$ replaced by $C_2 +C_2'$. If $\xi_0\leq \frac{\xi}{4}$, \eqref{Dxe-est4}, \eqref{mfact} and the formula $\omega(\xi,\xi_0)= \kappa m(\delta^{-1}) \delta^{1-\beta} \xi^\beta$ lead to
\begin{align*}
  D(\xi,t) & \leq C_1' \omega(\xi,\xi_0) - \frac{C_1 \beta(1-\beta)\kappa}{32 } m(\delta^{-1})\delta^{1-\beta}\xi^\beta m(\xi^{-1}) \\
  & \leq - \frac{C_1 \beta(1-\beta)\kappa}{64 } (m(\delta^{-1}))^2\delta^{2-\sigma-\beta}\xi^{\beta-1+\sigma} ,
\end{align*}
where in the last line we used
$A_0 \leq \left( \frac{C_1 m(1) \beta(1-\beta) }{64 C_1'}\right)^{\frac{1}{1-\sigma}}$;
whereas if $\xi_0\geq \frac{\xi}{4}$, by arguing as \eqref{Dxe-est5} and \eqref{Dxe-est2}, we obtain
\begin{equation*}
\begin{split}
  D(\xi,t) & \leq C_1'\omega(\xi,\xi_0) - \frac{C_1}{2} \omega(0+,\xi_0) m(\xi^{-1}) \\
  & \leq C_1' \omega(\xi,\xi_0) - \frac{C_1(1-\beta)}{8} \omega(\xi,\xi_0) m(\xi^{-1}) \leq - \frac{C_1(1-\beta) \kappa}{16} \left( m(\delta^{-1})\right)^2
  \delta^{2-\sigma-\beta} \xi^{\beta-1+\sigma},
\end{split}
\end{equation*}
where the last inequality is deduced from using
$A_0 \leq \left( \frac{C_1 m(1) (1-\beta) }{16 C_1'}\right)^{\frac{1}{1-\sigma}}$.
Thus we can similarly obtain \eqref{Targ4} in this case under (slightly modified) \eqref{rkg-cd6}.

Case 6: $0<\xi_0\leq \delta$, $\delta< \xi\leq A_0$.

We have \eqref{eq:OmeEst} with $C_2$ replaced by $C_2+C_2'$. Similarly as \eqref{eq:Dest}, \eqref{eq:Dest2} and \eqref{Dxe-est2}, we get
\begin{align*}
  D(\xi,t) & \leq C_1' \omega(\xi,\xi_0) + C_1 (\omega(2\xi,\xi_0)-2\omega(\xi,\xi_0)) \int_{\frac{\xi}{2}}^\xi \frac{m(\eta^{-1})}{\eta} \dd \eta \\
  & \leq C_1' \omega(\xi,\xi_0) - \frac{C_1 \tilde{c}}{8}  (1-\sigma)  m(\xi^{-1})\omega(\xi,\xi_0)\leq - \frac{C_1 \tilde{c}}{16}  (1-\sigma)  m(\xi^{-1})\omega(\xi,\xi_0),
\end{align*}
where the last inequality is ensured by $A_0\leq  \left( \frac{C_1 \tilde{c}\, m(1) (1-\sigma) }{16 C_1'}\right)^{\frac{1}{1-\sigma}}$. Thus in this case we deduce \eqref{Targ4} under (slightly modified) \eqref{kap-gam-cd1}.

Therefore, gathering the above results concludes \eqref{Targ4} at all cases and thus Theorem \ref{thm:glb2} is followed.


\section{Appendix}\label{sec-app}


\begin{proof}[Justification of the statement in Remark \ref{rmk:glob}]
By using the standard Bony's para-differential calculus and Lemma \ref{lem-symb}, we first can prove the local well-posedness result that
there exists a time $T>0$ depending only on
$\|\theta_0\|_{H^s}$ and $d$ such that the equation \eqref{DDeq}-\eqref{uexp} admits a uniquely local smooth solution
$\theta \in C([0,T[; H^s(\R^d)) \cap C^\infty (]0,T[\times \R^d)$.
Moreover, let $T^*$ be the maximal existence time of this solution, then by Lemma \ref{lem-RC} and the Calder\'on-Zygmund theorem, we necessarily get that
\begin{equation}\label{blcr1}
  \textrm{if}\;\; T^*<\infty\quad \Rightarrow \quad  \|\theta\|_{L^\infty([0,T^*[;\dot C^\beta(\R^d))} =\infty,\quad \forall\beta\in]1-\alpha+\sigma,1[.
\end{equation}

Next we prove that the maximal lifespan solution $\theta(t,x)$ on the time period $[0,T^*[$ strictly preserves the MOC $\omega(\xi)$ given by \eqref{MOC4.3} with $\alpha\in ]0,1]$, $\sigma\in[0,\alpha[$, $\delta>0$,
$0<\gamma<\kappa<1$, which implies the desired uniform $\beta$-H\"older estimate of $\theta$, and further concludes the statement concerned.

Similarly as the deduction around \eqref{MOC2}, $\omega(\xi)$ is a MOC satisfying the needing properties, and the mapping $\xi\mapsto \frac{\omega(\xi)}{\xi^\beta}$ for every $\xi>0$ is non-increasing.

Then we prove that under the assumption \eqref{m-int}, the MOC \eqref{MOC4.3} with fixed $\kappa,\gamma>0$ can be obeyed by the data $\theta_0$
for $\delta$ small enough. For this purpose, noting that
$|\theta_0(x)-\theta_0(y)|\leq 2 \|\theta_0\|_{L^\infty}$, and $|\theta_0(x)-\theta_0(y)|\leq \|\theta_0\|_{\dot C^\beta} |x- y|^\beta$,
it suffices to prove
\begin{equation}\label{egoal1}
  \min\set{2\|\theta_0\|_{L^\infty}, \|\theta_0\|_{\dot C^\beta} \xi^\beta } < \omega(\xi).
\end{equation}
Denote by $a_0 := \left( \frac{2\|\theta_0\|_{L^\infty}}{\|\theta_0\|_{\dot C^\beta}}\right)^{1/\beta}$, and if $\xi \geq a_0$, then as long as
\begin{equation}\label{econd1}
  \omega(a_0)> 2\|\theta_0\|_{L^\infty},
\end{equation}
we have that \eqref{egoal1} holds for all $\xi\geq a_0$; while if $\xi \leq a_0$, by virtue of \eqref{econd1} and the fact
$\frac{\omega(\xi)}{\xi^\beta} \geq \frac{\omega(a_0)}{a_0^\beta}$ which is deduced from \eqref{efact1},
we also obtain \eqref{egoal1}, as the following deduction shows
\begin{equation*}
  \|\theta_0\|_{\dot C^\beta} \xi^\beta \leq \|\theta_0\|_{\dot C^\beta} \frac{a_0^\beta}{\omega(a_0)} \omega(\xi)
  \leq \frac{2\|\theta_0\|_{L^\infty}}{\omega(a_0)} \omega(\xi) <\omega(\xi).
\end{equation*}
Now we prove that for every $\theta_0$, the condition \eqref{econd1} can be guaranteed by the assumption \eqref{m-int}. Indeed, without loss of generality we assume that
$a_0 \geq \delta$, then we get
\begin{equation}\label{efact2}
  \omega(a_0)\geq \gamma \int_\delta^{a_0} m(\eta^{-1})\dd \eta\rightarrow \infty,\quad \textrm{as  $\delta\rightarrow 0+$},
\end{equation}
hence \eqref{econd1} is ensured for $\delta$ sufficiently small depending on $\gamma$ and $\|\theta_0\|_{\dot C^\beta\cap L^\infty}$.

Recalling that $B_0$ is the bound of $\|\theta(\cdot,t)\|_{L^\infty_x}$ given by \eqref{theMP1}-\eqref{theMP2}, and by letting $b_0\in ]\delta,\frac{c_0}{2}]$ be a constant chosen later,
we use \eqref{m-int} to deduce
\begin{equation}\label{efact3}
  \omega\left( b_0\right)\geq \gamma \int_\delta^{b_0} m(\eta^{-1})\dd\eta > 2 B_0,
\end{equation}
provided that $\delta$ is sufficiently small.
Thus according to Proposition \ref{prop-GC}, it suffices to show that for all $0<t<T^*$ and all $0<\xi\leq b_0$,
\begin{equation}\label{Targ6}
  \Omega(\xi,t)\omega'(\xi) + D(\xi,t) <0,
\end{equation}
where
$D(\xi,t)$ and $\Omega(\xi,t)$ (by Lemmas \ref{lem-mocdiss}, \ref{lem-mocdrf}) are respectively given by \eqref{Dxit.0} and \eqref{Omega.0} with $\{\omega(\cdot),b_0\}$ in place of $\{\omega(\cdot,\xi_0),A_0\}$.

In a similar way as the proof of \eqref{Targ4} at Steps 5-6 and the proof of \eqref{Targ1} (noting that $b_0$ plays the same role as $A_0$ in showing \eqref{Targ4}), we find that by setting
$b_0 = \min\Big\{\frac{c_0}{2},\Big( \frac{C_1 m(1) \tilde{c}^2 \beta (\alpha-\sigma)}{16C_1'}\Big)^{1/(\alpha-\sigma)}\Big\}$, and for $\kappa,\gamma$ fixed constants satisfying \eqref{kap-gam-cd1} (up to pure numbers and $C_2$ replaced by $C_2+ C_2'$), and for $\delta>0$ sufficiently small constant satisfying \eqref{efact2}-\eqref{efact3},
we conclude \eqref{Targ6} at all the considered cases. Hence, the statement concerned on the global well-posedness result for \eqref{DDeq}-\eqref{uexp} is justified.

\end{proof}

\vskip0.2cm

\textbf{Acknowledgements.}
C. Miao was supported by National Natural Science Foundation of China (grants Nos. 11671047, 11831004 and 11826005).
L. Xue was supported by National Natural Science Foundation of China (grants Nos. 11401027, 11671039, 11771043).

\end{document}